\documentclass[12pt,twoside]{amsart}

\usepackage[T1]{fontenc}

\usepackage[utf8]{inputenc}

\usepackage{amsmath,amsthm,amssymb,amsfonts,amscd}
\usepackage{url}
\usepackage{color}
\usepackage{setspace}
\usepackage{enumerate}

\usepackage{enumitem}
\setlist{topsep=-9pt} 

\makeatletter
\def\namedlabel#1#2{\begingroup
    #2%
    \def\@currentlabel{#2}%
    \phantomsection\label{#1}\endgroup
}
\makeatother

\usepackage[all]{xy}

\usepackage{epsfig}
\usepackage{graphicx}
\usepackage{mathrsfs}
\usepackage{pinlabel}
\usepackage[outercaption]{sidecap}

\definecolor{light-gray}{gray}{0.60}

\usepackage{epsfig}
\usepackage{graphicx}
\usepackage{mathrsfs}
\usepackage{pinlabel}
\usepackage[outercaption]{sidecap}

\usepackage{cancel}

\usepackage{ulem}
\usepackage{tikz}
\usepackage{tikz-cd}
\usepackage{fp}
\usetikzlibrary{calc}
\usetikzlibrary{shapes,arrows}
\usetikzlibrary{fit,positioning}
\usetikzlibrary{patterns,decorations.pathreplacing}
\usetikzlibrary{matrix}

\definecolor{darkred}{rgb}{0.9,0.,.2}
\definecolor{darkblue}{rgb}{0.,0.,.6}
\definecolor{darkgreen}{rgb}{0.1,0.7,0.1}

\usepackage[colorlinks=true,urlcolor=darkblue,linkcolor=darkred,citecolor=darkgreen,pagebackref=true]{hyperref}

\usepackage{geometry}
\usepackage{fouriernc}
\usepackage[scaled=0.875]{helvet}
\usepackage{courier}

\definecolor{light-gray}{gray}{0.60}

\theoremstyle{plain}
\newtheorem{theorem}{Theorem}
\newtheorem{proposition}[theorem]{Proposition}
\newtheorem{corollary}[theorem]{Corollary}
\newtheorem{lemma}[theorem]{Lemma}
\newtheorem{fact}[theorem]{Fact}

\theoremstyle{definition}
\newtheorem{definition}[theorem]{Definition}
\newtheorem{example}[theorem]{Example}

\newtheorem{remark}[theorem]{Remark}

\numberwithin{theorem}{section}
\numberwithin{equation}{section}

\newcommand{\C}{\mathcal{C}}

\newcommand{\NN}{\mathbb{N}}
\newcommand{\ZZ}{\mathbb{Z}}

\newcommand{\RR}{\mathbb{R}}

\newcommand{\HH}{\mathbb{H}}
\newcommand{\PP}{\mathbb{P}}

\newcommand{\SL}{\mathrm{SL}}
\newcommand{\GL}{\mathrm{GL}}
\newcommand{\OO}{\mathrm{O}}
\newcommand{\PGL}{\mathrm{PGL}}

\newcommand{\Hom}{\mathrm{Hom}}
\newcommand{\ie}{i.e.\ }
\newcommand{\eg}{e.g.\ }
\newcommand{\resp}{resp.\ }

\newcommand{\Cart}{\mathcal{A}}

\newcommand{\partiali}{\partial_{\mathrm{i}}}
\newcommand{\partialn}{\partial_{\mathrm{n}}}
\newcommand{\Lambdao}{\Lambda^{\mathsf{orb}}}
\newcommand{\Ccore}{\C^{\mathsf{cor}}}
\newcommand{\OhmVin}{\Omega_{\scriptscriptstyle\mathrm{TV}}}

\renewcommand{\leq}{\leqslant}
\renewcommand{\geq}{\geqslant}

\title{Convex cocompactness for Coxeter groups}

\author[J. Danciger]{Jeffrey Danciger}
\address{Department of Mathematics, The University of Texas at Austin, 1 University Station C1200, Austin, TX 78712, USA}
\email{jdanciger@math.utexas.edu}

\author[F. Gu\'eritaud]{Fran\c{c}ois Gu\'eritaud}
\address{CNRS and IRMA, Universit\'e de Strasbourg, 7 rue Ren\'e Descartes, 67084 Strasbourg Cedex, France}
\email{francois.gueritaud@unistra.fr}

\author[F. Kassel]{Fanny Kassel}
\address{CNRS and Laboratoire Alexander Grothendieck, Institut des Hautes \'Etudes Scientifiques, Universit\'e Paris-Saclay, 35 route de Chartres, 91440 Bures-sur-Yvette, France}
\email{kassel@ihes.fr}

\author[G.-S. Lee]{Gye-Seon Lee}
\address{Department of Mathematical Sciences and Research institute of Mathematics, Seoul National University, Seoul 08826, South Korea}
\email{gyeseonlee@snu.ac.kr}

\author[L. Marquis]{Ludovic Marquis}
\address{Universit\'e de Rennes, CNRS, IRMAR - UMR 6625, 35000 Rennes, France}
\email{ludovic.marquis@univ-rennes.fr}

\dedicatory{Dedicated to the memory of Ernest Borisovich Vinberg (1937--2020)}


\begin{document}

\begin{abstract}
We investigate representations of Coxeter groups into $\GL(n,\RR)$ as geometric reflection groups which are convex cocompact in the projective space $\PP(\RR^n)$.
We characterize which Coxeter groups admit such representations, and we fully describe the corresponding spaces of convex cocompact representations as reflection groups, in terms of the associated Cartan matrices.
The Coxeter groups that appear include all infinite, word hyperbolic Coxeter groups; for such groups the representations as reflection groups that we describe are exactly the projective Anosov ones.
We also obtain a large class of nonhyperbolic Coxeter groups, thus providing many examples for the theory of nonhyperbolic convex cocompact subgroups in $\PP(\RR^n)$ developed in \cite{dgk-proj-cc}.
\end{abstract}


\maketitle

\setlength{\parskip}{0em}
\tableofcontents
\setlength{\parskip}{1em}

\section{Introduction} \label{sec:intro}

In the setting of semisimple Lie groups of real rank one, convex cocompact subgroups are a well-studied middle ground between the very special class of cocompact lattices and the general class of all finitely generated discrete subgroups.
The actions of such groups on the associated Riemannian symmetric space and its visual boundary at infinity are known to exhibit many desirable geometric and dynamical properties.

In the setting of higher-rank semisimple Lie groups, however, it is not obvious how to generalize the notion of convex cocompact subgroup in a way that guarantees similar desirable properties while at the same time being flexible enough to admit large interesting classes of examples.
In particular, the most natural generalization of convex cocompactness, in higher-rank Riemannian symmetric spaces, turns out to be quite restrictive, see \cite{kl06,qui05}.

In \cite{dgk-proj-cc}, the first three authors investigated several notions of convex cocompactness for discrete subgroups of the higher-rank semisimple Lie group $\PGL(V)$ acting on the projective space $\PP(V)$, where $V$ is a finite-dimensional real vector space.
This generalized both the classical theory of convex cocompactness in real hyperbolic geometry, and the theory of \emph{divisible convex sets}, as developed in particular by Benoist (see \cite{ben08}).
These notions of convex cocompactness in $\PP(V)$ come with a rich world of examples: see \eg \cite{mes90,ben06,mar10,bm12,cm14,bar15,cl15,mar17,bdl18,clm-dehnfill,dgk-ccHpq,dgk-proj-cc,zim,lm18,dgk-ex-cc,clm-ecima}.
The goal of the present paper is to study these notions in the context of Vinberg's theory of discrete reflection groups.

Reflection groups are a major source of interesting examples in the study of Kleinian groups and hyperbolic geometry.
While cocompact reflection groups in the real hyperbolic space $\HH^m$ may exist only for dimensions~$m \leqslant 29$ (see \cite{vin84}), there are rich families of convex cocompact reflection groups in $\HH^m$ for any dimension $m\geq 2$ (see \eg \cite[\S\,4]{dh13}).

In~\cite{vin71}, Vinberg studied a more general notion of reflection group in the setting of general linear groups $\GL(V)$, giving a necessary and sufficient condition for the linear reflections in the hyperplanes bounding a convex polyhedral cone $\widetilde{\Delta}$ in~$V$ to generate a discrete subgroup $\Gamma$ of $\GL(V)$ such that the $\Gamma$-translates of $\operatorname{Int}\,(\widetilde{\Delta})$ are mutually disjoint.
We will call these discrete groups \emph{reflection groups}.\footnote{Vinberg referred to these groups as \emph{discrete linear groups generated by reflections}, or simply \emph{linear Coxeter groups} \cite[Def.\,2]{vin71}.}
When infinite, they naturally identify with discrete subgroups of $\PGL(V)$ via the projection $\GL(V) \to \PGL(V)$.
Vinberg's construction includes the reflection groups in~$\HH^m$ as a special case for which the reflections preserve a nondegenerate quadratic form of signature $(m,1)$ on~$V$.
However, the construction also gives large families of interesting discrete subgroups of $\GL(V)$ which are not contained in $\OO(m,1)$, including some that preserve nondegenerate quadratic forms of signature other than $(m,1)$, and many others that do not preserve any nonzero quadratic form at all.
This is a rich source of examples of discrete subgroups of infinite covolume in higher-rank semisimple Lie groups (see \eg \cite{kv67,benQI,ben06,clm-dehnfill,lm18}).

The goal of the paper is to give an explicit characterization of the notions of convex cocompactness from \cite{dgk-proj-cc} in the setting of Vinberg's reflection groups, with an application to the study of Anosov representations of word hyperbolic Coxeter groups.

\subsection{Convex cocompactness in projective space}

Let $V$ be a real vector space of dimension $n\geq 2$.
We say that an open subset $\Omega$ of the projective space $\PP(V)$ is \emph{convex} if it is contained and convex in some affine chart, and \emph{properly convex} if it is convex and bounded in some affine chart.
Here are the three notions of projective convex cocompactness from \cite{dgk-proj-cc} which we consider.

\begin{definition} \label{def:cc-group}
An infinite discrete subgroup $\Gamma$ of $\GL(V)$ is
\begin{enumerate}[label=(\roman*)]
  \item\label{item:def-naiv-cc} \emph{naively convex cocompact in $\PP(V)$} if it acts properly discontinuously on some properly convex open subset $\Omega$ of $\PP(V)$, and cocompactly on some nonempty $\Gamma$-invariant closed convex subset $\C$ of~$\Omega$;
  \item\label{item:def-cc} \emph{convex cocompact in $\PP(V)$} if \ref{item:def-naiv-cc} holds and $\C \subset \Omega$ can be taken ``large enough'', in the sense that the closure of $\C$ in $\PP(V)$ contains all accumulation points of all possible $\Gamma$-orbits $\Gamma \cdot y$ with $y \in \Omega$;
  \item\label{item:def-strong-cc} \emph{strongly convex cocompact in $\PP(V)$} if \ref{item:def-naiv-cc} holds and $\Omega$ can be taken so that its boundary $\partial \Omega$ is strictly convex (\ie does not contain any nontrivial projective segment) and $C^1$ (\ie every point has a unique supporting hyperplane).
\end{enumerate}
\end{definition}

We note that \ref{item:def-strong-cc} implies \ref{item:def-cc}, as all $\Gamma$-orbits of~$\Omega$ have the same accumulation points when $\Omega$ is strictly convex (Remark~\ref{rem:strict-convex-Lambda}).

The three notions of convex cocompactness in Definition~\ref{def:cc-group} may seem quite similar at first glance, and they are indeed equivalent for discrete subgroups of $\mathrm{Isom}(\HH^m) \subset \OO(m,1)$.
However, the three notions admit a number of subtle differences in general.
In particular, naive convex cocompactness is not always stable under small deformations (see \cite[Rem.\,4.5.(b)]{dgk-proj-cc} and \cite{dgk-ex-cc}), whereas convex cocompactness and strong convex cocompactness are \cite[Th.\,1.16]{dgk-proj-cc}.
By \cite[Th.\,1.15]{dgk-proj-cc}, an infinite discrete subgroup of $\GL(V)$ is strongly convex cocompact in $\PP(V)$ if and only if it is word hyperbolic and convex cocompact~in~$\PP(V)$.

In the case that $\C=\Omega$ in Definition~\ref{def:cc-group}, we say that $\Gamma$ \emph{divides}~$\Omega$, and that $\Omega$ is \emph{divisible}.
Divisible convex sets have been much studied since the 1960s.
Examples with $\Gamma$ nonhyperbolic (or equivalently $\partial\Omega$ not strictly convex --- hence \ref{item:def-cc} is satisfied but not \ref{item:def-strong-cc}) include the case that $\Omega$ is the projective model of the Riemannian symmetric space of $\SL(k,\RR)$ with $k\geq 3$, and $\Gamma$ a uniform lattice (see \eg \cite[\S\,2.4]{ben08}).
The first irreducible and nonsymmetric examples with $\Gamma$ nonhyperbolic were constructed by Benoist \cite{ben06} for $4\leq\dim V\leq 7$, taking $\Gamma$ to be a reflection group.
Later, further examples were found in \cite{mar10,bdl18,clm-dehnfill,clm-ecima}.

In the current paper, we give a necessary and sufficient condition (Theorem~\ref{thm:main}) for a reflection group to be convex cocompact in $\PP(V)$ in the sense of Definition~\ref{def:cc-group}.
As a consequence, in the setting of \emph{right-angled} reflection groups, the three notions of convex cocompactness in Definition~\ref{def:cc-group} are equivalent, and right-angled convex cocompact reflection groups are always word hyperbolic (Corollary~\ref{cor:main-racg}).
For general reflection groups, we prove (Theorems \ref{thm:naive-cc-exists} and~\ref{thm:main}) that naive convex cocompactness is always equivalent to convex cocompactness, but that there also exist many non-right-angled reflection groups which are convex cocompact without being strongly convex cocompact in $\PP(V)$ (beyond the examples of \cite{ben06,mar10,bdl18,clm-dehnfill,clm-ecima}); these groups are not hyperbolic anymore, but relatively hyperbolic (Corollary~\ref{cor:rel-hyp}).

\subsection{Vinberg's theory of reflection groups} \label{subsec:intro-Vinberg-theory}

Let
$$W_S = \big\langle s_1,\dots,s_N ~|~ (s_i s_j)^{m_{i,j}}=1,\,\, \forall \ 1\leq i,j\leq N\big\rangle$$
be a Coxeter group with generating set $S = \{s_1,\dots,s_N\}$, where $m_{i,i}=1$ and $m_{i,j} = m_{j,i} \in\{2,3,\dots,\infty\}$ for $i\neq j$.
(By convention, $(s_i s_j)^{\infty}=1$ means that $s_is_j$ has infinite order in the group~$W_S$.)
Each generator~$s_i$ has order two.
For any subset $S'$ of $S$, we denote by $W_{S'}$ the subgroup of $W_S$ generated by $S'$, which we call a \emph{standard subgroup} of $W_S$.
Recall that $W_S$ is said to be \emph{right-angled} if $m_{i,j}\in\{2,\infty\}$ for all $i\neq j$, and \emph{irreducible} if it cannot be written as the direct product of two nontrivial standard subgroups.
An irreducible Coxeter group is said to be \emph{spherical} if it is finite, and \emph{affine} if it is infinite and virtually abelian.
We refer to Appendix~\ref{appendix:classi_diagram} for the list of all spherical and all affine irreducible Coxeter groups.

In the whole paper, $V$ will denote a finite-dimensional real vector space, and we shall use the following terminology and notation.

\begin{definition} \label{def:refl-group}
A representation $\rho : W_S\to\GL(V)$ is a \emph{representation of~$W_S$ as a reflection group in~$V$} if
\begin{itemize}
  \item each element $\rho(s_i)$ is a hyperplane reflection, of the form $(x \mapsto x - \alpha_i(x) v_i)$ for some linear form $\alpha_i \in V^\ast$ and some vector $v_i \in V$ with $\alpha_i(v_i) = 2$;
  \item the convex polyhedral cone $\widetilde{\Delta} := \{ v\in V \,|\, \alpha_i(v)\leq 0 \ \forall i\}$ has nonempty interior;
  \item $\rho(\gamma) \cdot \operatorname{Int}\,(\widetilde \Delta) \cap \operatorname{Int}\,(\widetilde \Delta) = \emptyset$ for all $\gamma \in W_S \smallsetminus \{e\}$.
\end{itemize}
In this case the matrix $\Cart=(\alpha_i(v_j))_{1\leq i,j\leq N}$ is called a \emph{Cartan matrix} for~$\rho$.
\end{definition}

We denote by $\Hom^{\mathrm{ref}}(W_S, \GL(V))$ the set of all representations of~$W_S$ as a reflection group in~$V$.
Note that any $\rho \in \Hom^{\mathrm{ref}}(W_S, \GL(V))$ is faithful and discrete.
Vinberg \cite[Th.\,2]{vin71} showed that the group $\rho(W_S)$ preserves an open convex cone of~$V$, namely the interior $\widetilde{\Omega}_{\scriptscriptstyle\mathrm{TV}}$ of the union of all $\rho(W_S)$-translates of the fundamental polyhedral cone~$\widetilde \Delta$.
This cone also appears in the work of Tits \cite{tit61}; we call it the \emph{Tits--Vinberg cone}.
The group $\rho(W_S)$ acts properly discontinuously on $\widetilde{\Omega}_{\scriptscriptstyle\mathrm{TV}}$. 
The action of $\rho(W_S)$ is still properly discontinuous on the image $\OhmVin \subset \PP(V)$ of $\widetilde{\Omega}_{\scriptscriptstyle\mathrm{TV}}$, which, for $W_S$ infinite and irreducible, is a nonempty convex open subset of the projective space $\PP(V)$ (the \emph{Tits--Vinberg domain}).
The composition of $\rho$ with the projection $\GL(V) \to \PGL(V)$ is then still faithful and discrete.

The pair $(\alpha, v)\in V^* \times V$ defining a reflection is unique up to the action of $\RR^*$ by $t\cdot (\alpha, v)=(t\alpha, t^{-1}v)$.
The reflections $\rho(s_1),\dots, \rho(s_N)$ therefore determine the matrix $(\alpha_i(v_j))_{1\leq i,j \leq N}$ up to conjugation by a diagonal matrix.
In particular, they determine the real numbers $\alpha_i(v_j)\,\alpha_j(v_i)$.
Vinberg \cite[Th.\,1 \& Prop.\,6, 13, 17]{vin71} gave the following characterization: a representation $\rho : W_S \to \GL(V)$ belongs to $\Hom^{\mathrm{ref}}(W_S, \GL(V))$ if and only if the $\alpha_i$ and $v_j$ can be chosen so that the matrix $\Cart = (\Cart_{i,j})_{1\leq i,j\leq N} = (\alpha_i(v_j))_{1\leq i,j\leq N}$ satisfies the following five conditions:
\begin{enumerate}[label=(\roman*)]
  \item\label{item:cartan-matrix-entry=0} $\Cart_{i,j}=0$ for all $i\neq j$ with $m_{i,j}=2$;
  \item\label{item:cartan-matrix-entry<0} $\Cart_{i,j}<0$ for all $i\neq j$ with $m_{i,j} \neq 2$;
  \item\label{item:cartan-matrix-entry=4cos} $\Cart_{i,j}\Cart_{j,i} = 4 \cos^2(\pi/m_{i,j})$ for all $i\neq j$ with $2 < m_{i,j} < \infty$;
  \item \label{item:Vinberg-weak-inequality} $\Cart_{i,j}\Cart_{j,i}\geq 4$ for all $i\neq j$ with $m_{i,j} = \infty$; and 
  \item\label{item:nonempty-interior} $\widetilde{\Delta} := \{ v\in V \,|\, \alpha_i(v)\leq 0 \ \forall i\}$ has nonempty interior.
\end{enumerate}
In that case, $\Cart$ is a Cartan matrix for $\rho$ and is unique up to conjugation by positive diagonal matrices.
Conditions \ref{item:cartan-matrix-entry=0}--\ref{item:cartan-matrix-entry<0}--\ref{item:cartan-matrix-entry=4cos}--\ref{item:Vinberg-weak-inequality} are semialgebraic, and so is \ref{item:nonempty-interior} by quantifier elimination (see \cite[Prop.\,2.2.4]{bcr}).
Thus $\Hom^{\mathrm{ref}}(W_S, \GL(V))$ is a semialgebraic subset of $\Hom(W_S, \GL(V))$.

If $W_S$ is infinite, irreducible, and not affine, then \ref{item:nonempty-interior} is always satisfied (see Sections~\ref{subsec:Coxeter-type}--\ref{subsec:neg-type}); this implies in particular that $\Hom^{\mathrm{ref}}(W_S, \GL(V))$ is open and closed in the space $\Hom^{\mathrm{fd}}(W_S, \GL(V))$ of all faithful and discrete representations of $W_S$ into $\GL(V)$ (Remark~\ref{rem:open-closed-in-Hom-df}.\eqref{item:open-closed-in-Hom-df}).
We shall prove (Proposition~\ref{prop:maximal}) that in this case the Tits--Vinberg domain $\OhmVin$ is maximal in the sense that it contains every $\rho(W_S)$-invariant properly convex open subset of $\PP(V)$.
We shall also describe (Theorem~\ref{thm:minimal_convex}) the minimal nonempty $\rho(W_S)$-invariant properly convex open subset of $\PP(V)$.

If $W_S$ is irreducible and affine, then $\Hom^{\mathrm{ref}}(W_S, \GL(V))$ is not always closed in\linebreak $\Hom^{\mathrm{fd}}(W_S, \GL(V))$: see Remark~\ref{rem:JWST}.

\subsection{Coxeter groups admitting convex cocompact realizations} \label{subsec:intro-main}

Our first main result is a characterization of those Coxeter groups admitting representations as reflection groups in~$V$ which are convex cocompact in $\PP(V)$.
For this we consider the following two conditions on~$W_S$:
\begin{enumerate}
  \item[\namedlabel{item:no-Z2}{$\neg$\!\texttt{(IC)}}] there do not exist disjoint subsets $S',S''$ of~$S$ such that $W_{S'}$ and $W_{S''}$ are both infinite and commute;
  \item[\namedlabel{item:aff-only-Ak}{\texttt{($\widetilde{\text{\texttt{A}}}$)}}] for any subset $S'$ of~$S$ with $\#S'\geq 3$, if $W_{S'}$ is irreducible and affine, then it is of type $\widetilde{A}_k$ where $k=\#S'-1$ (see Table~\ref{table:affi_diag} in Appendix~\ref{appendix:classi_diagram}).
\end{enumerate}

\begin{theorem} \label{thm:naive-cc-exists}
For an infinite Coxeter group $W_S$, the following are equivalent:
\begin{enumerate}
  \item\label{item:naive-cc-exists} there exist a finite-dimensional real vector space~$V$ and a representation $\rho\in\Hom^{\mathrm{ref}}(W_S,\GL(V))$ such that $\rho(W_S)$ is naively convex cocompact in $\PP(V)$;
  \item\label{item:cond-a-b} $W_S$ satisfies conditions \ref{item:no-Z2} and~\ref{item:aff-only-Ak}.
\end{enumerate}
In this case, $\rho(W_S)$ is actually convex cocompact in $\PP(V)$ (not only naively convex cocompact), and we can take $V$ to be any vector space of dimension $\geq\# S$.
\end{theorem}

\begin{remark}
Conditions~\ref{item:no-Z2} and~\ref{item:aff-only-Ak} are readily checked on the diagram of~$W_S$ using the classification of spherical and affine Coxeter groups (see Appendix~\ref{appendix:classi_diagram}).
If $W_S$ satisfies \ref{item:no-Z2} and~\ref{item:aff-only-Ak}, then any standard subgroup $W_{S'}$ of $W_S$ still satisfies \ref{item:no-Z2} and~\ref{item:aff-only-Ak}.
\end{remark}

\begin{remark}
By a result of Krammer \cite[Th.\,6.8.3]{kra94}, conditions \ref{item:no-Z2} and~\ref{item:aff-only-Ak} together are equivalent to the fact that any subgroup of~$W_S$ isomorphic to $\ZZ^2$ is virtually contained in a conjugate of a standard subgroup $W_{S'}$ of type~$\widetilde{A}_k$ for some $k\geq 2$.
\end{remark}

\begin{remark} \label{rem:moussong}
Conditions \ref{item:no-Z2} and~\ref{item:aff-only-Ak} are always satisfied if $W_S$ is word hyperbolic.
In fact, Moussong \cite{mou87} proved that word hyperbolicity of~$W_S$ is equivalent to condition~\ref{item:no-Z2} together with the following condition (which in particular implies~\ref{item:aff-only-Ak}):
\begin{enumerate}
  \item[\namedlabel{item:non_aff}{$\neg$\!\texttt{(Af)}}] there does not exist a subset $S'$ of~$S$ with $\#S'\geq 3$ such that $W_{S'}$ is irreducible and affine.
\end{enumerate}
We recently used convex cocompact reflection groups in $\OO(p,q)$ to give a new proof of this hyperbolicity criterion: see \cite[Cor.\,8.5]{dgk-ccHpq} in the right-angled case and \cite[\S\,7]{lm18} in the general case.
\end{remark}

Theorem~\ref{thm:naive-cc-exists}, together with a result of Caprace \cite{cap09,cap15}, implies the following.

\begin{corollary} \label{cor:rel-hyp}
For an infinite Coxeter group $W_S$, the following are equivalent:
\begin{enumerate}
  \item\label{item:naive-cc-exists-again} there exist $V$ and $\rho\in\Hom^{\mathrm{ref}}(W_S,\GL(V))$ such that $\rho(W_S)$ is (naively) convex cocompact in $\PP(V)$;
  \item\label{item:relatively-hyperbolic} $W_S$ is relatively hyperbolic with respect to a collection of virtually abelian subgroups of rank $\geq 2$, which are the standard subgroups of~$W_S$ of the form $W_U \times W_{U^\perp}$ where $W_U$  is of type~$\widetilde{A}_k$ for some $k\geq 2$ and $W_{U^\perp}$ is the (finite) standard subgroup of $W_S$ generated by $U^\perp := \{ s \in S\smallsetminus U \, |\, m_{u,s}=2\ \forall u \in U\}$.
\end{enumerate}
\end{corollary}

\subsection{A characterization of convex cocompactness}

Our second main result is, for a Coxeter group $W_S$ as in Theorem~\ref{thm:naive-cc-exists}, a simple characterization of convex cocompactness for representations $\rho\in\Hom^{\mathrm{ref}}(W_S,\GL(V))$.

\begin{theorem} \label{thm:main}
Let $W_S$ be an infinite Coxeter group satisfying conditions \ref{item:no-Z2} and~\ref{item:aff-only-Ak} of Section~\ref{subsec:intro-main}.
For any~$V$ and any $\rho\in\Hom^{\mathrm{ref}}(W_S,\GL(V))$ with Cartan matrix $\Cart=(\Cart_{i,j})_{1\leq i,j\leq N}$, the following are equivalent:
\begin{enumerate}
  \item[\namedlabel{item:naive-cc-cox}{\texttt{(NCC)}}] $\rho(W_S)$ is naively convex cocompact in $\PP(V)$;
  \item[\namedlabel{item:cc-cox}{\texttt{(CC)}}] $\rho(W_S)$ is convex cocompact in $\PP(V)$;
  \item[\namedlabel{item:not-zero-type}{\texttt{$\neg$\!(ZT)}}] for any irreducible standard subgroup $W_{S'}$ of~$W_S$ with $\emptyset\neq S'\subset S$, the Cartan submatrix $\Cart_{S'}:=(\Cart_{i,j})_{s_i,s_j\in S'}$ is \emph{not} of zero type;
  \item[\namedlabel{item:not-zero-type-explicit}{\texttt{$\neg$\!(ZD)}}] for any $S' \subset S$ such that $W_{S'}$ is of type~$\widetilde{A}_k$ for some $k\geqslant 1$, we have $\det(\Cart_{S'})\neq 0$.
\end{enumerate}
(The $k=1$ case of~\ref{item:not-zero-type-explicit} just means that $\Cart_{i,j}\Cart_{j,i} > 4$ for all $i\neq j$ with $m_{i,j}=\infty$.)
\end{theorem}
We refer to Definition~\ref{def:type} for the notion of \emph{zero type}.

In Theorem~\ref{thm:main}, the implication \ref{item:cc-cox}~$\Rightarrow$~\ref{item:naive-cc-cox} holds by definition; the equivalence \ref{item:not-zero-type}~$\Leftrightarrow$~\ref{item:not-zero-type-explicit} follows from classical results of Vinberg, see Fact~\ref{fact:types} below.
We prove the other implications in Section~\ref{sec:main_thm}.

\begin{remark}
In \cite{mar17}, the last author studied groups generated by reflections in the codimension-$1$ faces of a polytope $\Delta \subset \PP(V)$ which is \emph{2-perfect} (\ie $\Delta \cap \partial \OhmVin$ is a subset of the vertices of~$\Delta$).
In the case that the group acts strongly irreducibly on~$V$, he gave a criterion~\cite[Th.\,A]{mar17} for naive convex cocompactness (Definition~\ref{def:cc-group}.\ref{item:def-naiv-cc}), as well as for a notion of geometric finiteness, in terms of links of the vertices of the polytope.
\end{remark}

Here is an easy consequence of Theorem~\ref{thm:main}.

\begin{corollary} \label{cor:cc-standard-subgroup}
Let $W_S$ be an infinite Coxeter group.
For any~$V$ and any $\rho\in\linebreak\Hom^{\mathrm{ref}}(W_S,\GL(V))$, if $\rho(W_S)$ is convex cocompact in $\PP(V)$, then so is $\rho(W_{S'})$ for any infinite standard subgroup $W_{S'}$ of $W_S$.
\end{corollary}

Theorem~\ref{thm:main} yields the following simple characterization of strong convex cocompactness for Coxeter groups.

\begin{corollary} \label{cor:main2}
Let $W_S$ be an infinite Coxeter group.
For any~$V$ and any $\rho\in\linebreak\Hom^{\mathrm{ref}}(W_S,\GL(V))$ with Cartan matrix $\Cart=(\Cart_{i,j})_{1\leq i,j\leq N}$, the following are equivalent:
\begin{enumerate}
  \item[\namedlabel{item:strong-cc-cox}{\texttt{(SCC)}}] $\rho(W_S)$ is strongly convex cocompact in $\PP(V)$;
  \item[\namedlabel{item:word-hyp-+}{\texttt{(WH+)}}] $W_S$ is word hyperbolic and $\Cart_{i,j}\Cart_{j,i} > 4$ for all $i\neq j$ with $m_{i,j}=\infty$.
\end{enumerate}
\end{corollary}

\begin{remark}
Corollary~\ref{cor:main2} generalizes a classical result stating that for $\rho$ with values in $\OO(m,1)$, the reflection group $\rho(W_S)$ is convex cocompact if and only if $W_S$ is word hyperbolic and $\Cart_{i,j}\Cart_{j,i} > 4$ for all $i\neq j$ with $m_{i,j}=\infty$ (see \eg \cite[Th.\,4.12]{dh13}).
\end{remark}

It easily follows from Theorems \ref{thm:naive-cc-exists} and~\ref{thm:main} and Corollary~\ref{cor:main2} that in the setting of right-angled Coxeter groups, the three notions of convex cocompactness in Definition~\ref{def:cc-group} are equivalent.

\begin{corollary} \label{cor:main-racg}
Let $W_S$ be an infinite Coxeter group with no standard subgroup of type $\widetilde{A}_k$ for $k\geq 2$, for instance an infinite right-angled Coxeter group.
For any~$V$ and any $\rho\in\Hom^{\mathrm{ref}}(W_S,\GL(V))$ with Cartan matrix $\Cart=(\Cart_{i,j})_{1\leq i,j\leq N}$, the following are equivalent:
\begin{enumerate}[label=(\roman*)]
  \item[\ref{item:naive-cc-cox}] $\rho(W_S)$ is naively convex cocompact in $\PP(V)$;
  \item[\ref{item:cc-cox}] $\rho(W_S)$ is convex cocompact in $\PP(V)$;
  \item[\ref{item:strong-cc-cox}] $\rho(W_S)$ is strongly convex cocompact in $\PP(V)$;
  \item[\ref{item:word-hyp-+}] $W_S$ is word hyperbolic and $\Cart_{i,j}\Cart_{j,i} > 4$ for all $i\neq j$ with $m_{i,j}=\infty$.
\end{enumerate}
\end{corollary}

\begin{remark}
Corollary~\ref{cor:main-racg} was previously proved in the case that $\rho$ preserves a nondegenerate quadratic form on~$V$, by the first three authors when $W_S$ is right-angled \cite[Th.\,8.2]{dgk-ccHpq} and by the last two authors \cite[Th.\,4.6]{lm18} in general.
\end{remark}

\subsection{The subspace of convex cocompact representations}\label{section:subspace_of_CC_rep}

By \cite[Th.\,1.16]{dgk-proj-cc}, the set of convex cocompact representations is open in $\Hom(W_S,\GL(V))$, but can we characterize it more precisely?

Suppose $W_S$ is irreducible and word hyperbolic and let $\rho\in\Hom^{\mathrm{ref}}(W_S,\GL(V))$ have Cartan matrix $\Cart=(\Cart_{i,j})_{1\leq i,j\leq N}$.
As we recalled in Section~\ref{subsec:intro-Vinberg-theory}, the weak inequality $\Cart_{i,j}\Cart_{j,i} \geq 4$ holds whenever $m_{i,j} = \infty$. 
By Corollary~\ref{cor:main-racg}, the group $\rho(W_S)$ is convex cocompact in $\PP(V)$ if and only if these inequalities are all strict.
It is then natural to ask if the open subset of convex cocompact representations is the full interior of $\Hom^{\mathrm{ref}}(W_S,\GL(V))$ in $\Hom(W_S, \GL(V))$.
We show that the answer is yes in large enough dimension (Corollary~\ref{coro:cc-int-refl}) but not in general (Example~\ref{ex:word-hyperbolic}).

Recall that a representation $\rho$ of a discrete group $\Gamma$ into $\GL(V)$ is said to be \emph{semisimple} if it is a direct sum of irreducible representations; equivalently, the orbit of~$\rho$ under the action of $\GL(V)$ by conjugation is closed.
The quotient of the space of semisimple representations by $\GL(V)$ is called the \emph{space of characters} of $\Gamma$ in $\GL(V)$; we denote it by $\chi(\Gamma,\mathrm{GL}(V))$.
If $W_S$ is a Coxeter group, then a character coming from a representation of $W_S$ as a reflection group in~$V$ is called a \emph{reflection character}, and we shall denote the space of such characters by $\chi^{\mathrm{ref}}(W_S,\mathrm{GL} (V) )$.

As a consequence of Corollary~\ref{cor:main-racg}, we prove the following.

\begin{corollary} \label{coro:nra_cc-int-refl}
Let $W_S$ be an infinite word hyperbolic irreducible Coxeter group.
For any~$V$ with $n:=\dim V\geq N:=\#S$, the set of characters $[\rho] \in \chi^{\mathrm{ref}}(W_S,\GL(V))$ for which $\rho(W_S)$ is convex cocompact in $\PP(V)$ is precisely the interior of $\chi^{\mathrm{ref}}(W_S,\GL(V))$ in\linebreak $\chi(W_S,\GL(V))$.
\end{corollary}

The main content of Corollary~\ref{coro:nra_cc-int-refl} is the case $n = N$.
Indeed, a semisimple representation $\rho \in \Hom^{\mathrm{ref}}(W_S,\GL(V))$ with Cartan matrix $\Cart$ splits as a direct sum of a representation $\rho'$ of dimension $\mathrm{rank}(\Cart)$ plus a trivial representation of complementary dimension. 
When $n \geqslant N$, the space $\chi^{\mathrm{ref}}(W_S,\GL(V))$ is homeomorphically parametrized by the space of equivalence classes of $N \times N$ Cartan matrices that are compatible with~$W_S$ (see Definition~\ref{def:compatible}).
Hence the structure of $\chi^{\mathrm{ref}}(W_S,\GL(V))$ is the same for all $n = \dim(V) \geq N$.
Note that the bound $n\geq N$ in Corollary~\ref{coro:nra_cc-int-refl} cannot be improved in general, since when $n < N$, the condition that $\mathrm{rank}(\Cart) \leq n$ imposes a nontrivial constraint so that the entries of the Cartan matrix may not be deformed freely.
It can happen that $4$ is an isolated local minimum value of the function $\Cart_{i,j} \Cart_{j,i}$ for some pair of generators with $m_{i,j} = \infty$ (see Example~\ref{ex:word-hyperbolic}).
Note also that Corollary~\ref{coro:nra_cc-int-refl} does not extend to nonhyperbolic groups (see Example~\ref{ex:non-hyperbolic}).

The structure of $\Hom^{\mathrm{ref}}(W_S,\GL(V))$ is more complicated than that of $\chi^{\mathrm{ref}}(W_S,\GL(V))$. 
While a Cartan matrix $\Cart$ determines at most one conjugacy class of semisimple representations in $\Hom^{\mathrm{ref}}(W_S,\GL(V))$, there may be many conjugacy classes of nonsemisimple representations in $\Hom^{\mathrm{ref}}(W_S,\GL(V))$ with Cartan matrix $\Cart$. 
Further, even in cases when $n \geq N$, it is difficult to determine if the map sending a representation $\rho$ generated by reflections to the (equivalence class of a) Cartan matrix for $\rho$ is open. 
We prove the following.

\begin{corollary} \label{coro:cc-int-refl}
Let $W_S$ be an infinite word hyperbolic irreducible Coxeter group in $N=\# S$ generators and let $V$ be a real vector space of dimension~$n$.
If
\begin{enumerate}
  \item $n\geq 2N-2$, or 
  \item $W_S$ is right-angled and $n\geq N$,
\end{enumerate}
then the set of representations $\rho \in \Hom^{\mathrm{ref}}(W_S,\GL(V))$ for which $\rho(W_S)$ is convex cocompact in $\PP(V)$ is precisely the interior of $\Hom^{\mathrm{ref}}(W_S,\GL(V))$ in $\Hom(W_S,\GL(V))$.
\end{corollary}

Corollary~\ref{coro:cc-int-refl} is a consequence of the equivalence \ref{item:cc-cox}~$\Leftrightarrow$~\ref{item:word-hyp-+} of Corollary~\ref{cor:main-racg} together with some analysis of how the entries of the Cartan matrix deform under the dimension restriction (see Section~\ref{subsec:cc-int}).

\subsection{Anosov representations for Coxeter groups}

We give an application of Corollary~\ref{cor:main2} to the topic of \emph{Anosov representations} of word hyperbolic groups into $\GL(V)$.
These are representations with strong dynamical properties which generalize convex cocompact representations into rank-one semisimple Lie groups (see in particular \cite{lab06,gw12,klp17,klp-survey,ggkw17,bps19}).
They have been much studied recently, especially in the setting of higher Teichm\"uller theory.

We shall not recall the definition of Anosov representations in this paper, nor assume any technical knowledge of them.
We shall just use the following relation with convex cocompactness in $\PP(V)$, proven in \cite{dgk-ccHpq,dgk-proj-cc} and also, in a slightly different form, in \cite{zim}.

\begin{fact}[{\cite[Th.\,1.4]{dgk-proj-cc}}] \label{fact:Anosov}
Let $\Gamma$ be an infinite discrete subgroup of $\GL(V)$ preserving some nonempty properly convex open subset $\Omega$ of $\PP(V)$.
Then $\Gamma$ is strongly convex cocompact in $\PP(V)$ (Defini\-tion~\ref{def:cc-group}.\ref{item:def-strong-cc}) if and only if $\Gamma$ is word hyperbolic and the natural inclusion $\Gamma \hookrightarrow \GL(V)$ is $P_1$-Anosov (\ie Anosov with respect to the stabilizer of a line~in~$V$).
\end{fact}

This relation allows for the construction of new examples of discrete subgroups of $\GL(V)$ which are strongly convex cocompact in $\PP(V)$, by using classical examples of Anosov representations: see \eg \cite[Prop.\,1.7]{dgk-proj-cc} or \cite[Cor.\,1.33]{zim}.
Conversely, Fact~\ref{fact:Anosov} gives a way to obtain new examples of Anosov representations by constructing strongly convex cocompact groups: this is the point of view adopted in \cite[\S\,8]{dgk-ccHpq} and in the present paper.

More precisely, in \cite[\S\,8]{dgk-ccHpq} and \cite[\S\,7]{lm18} we constructed, for any infinite, word hyperbolic, irreducible Coxeter group~$W_S$, examples of representations $\rho\in\linebreak\Hom^{\mathrm{ref}}(W_S,\GL(V))$ which are convex cocompact in $\PP(V)$ for some~$V$; by Fact~\ref{fact:Anosov} this yielded examples of $P_1$-Anosov representations of $W_S$ (or any quasi-isometrically embedded subgroup) into $\GL(V)$.
These representations took values in the orthogonal group of a nondegenerate quadratic form.
Obtaining such representations is interesting because, while Anosov representations of free groups and surface groups are abundant in the literature, examples of Anosov representations of more complicated hyperbolic groups have been much less commonly known outside the realm of Kleinian groups.

Consider the set of $P_1$-Anosov representations of a word hyperbolic (not necessarily right-angled) Coxeter group into $\GL(V)$ for an arbitrary~$V$.
Corollaries~\ref{cor:main2} and~\ref{coro:cc-int-refl}, together with Fact~\ref{fact:Anosov}, yield the following description of those Anosov representations that realize the group as a reflection group  in~$V$ (see Section~\ref{subsec:proof-Ano-cc-Coxeter}).

\begin{corollary} \label{cor:hyp-CG-Anosov}
Let $W_S$ be an infinite word hyperbolic Coxeter group in $N=\# S$ generators and let $V$ be a real vector space of dimension~$n$.
For any $\rho\in\Hom^{\mathrm{ref}}(W_S,\GL(V))$ with Cartan matrix $\Cart=(\Cart_{i,j})_{1\leq i,j\leq N}$, the following are equivalent:
\begin{itemize}
  \item $\rho$ is $P_1$-Anosov;
  \item $\Cart_{i,j}\Cart_{j,i} > 4$ for all $i\neq j$ with $m_{i,j}=\infty$.
\end{itemize}
Assume moreover that $W_S$ is irreducible.
If either $n \geq 2N-2$, or $n \geq N$ and $W_S$ is right-angled, then the space of $P_1$-Anosov representations of $W_S$ as a reflection group is precisely the interior of $\Hom^{\mathrm{ref}}(W_S, \GL(V))$ in $\Hom(W_S, \GL(V))$.
\end{corollary}

\begin{remark}
If $n<N$, then the space of $P_1$-Anosov representations of $W_S$ as a reflection group may be smaller than the interior of $\Hom^{\mathrm{ref}}(W_S, \GL(V))$.
This is the case in Example~\ref{ex:word-hyperbolic} below, where $W_S$ is word hyperbolic and the interior of $\Hom^{\mathrm{ref}}(W_S, \GL(V))$ contains a unique representation which is faithful and discrete but not Anosov.
The existence of such a non-Anosov representation of a word hyperbolic group $\Gamma$ admitting a neighborhood in $\Hom(\Gamma,\GL(V))$ consisting entirely of faithful and discrete representations provides a negative answer to a question asked in \cite[\S\,8]{kas19} and \cite[\S\,4.3]{pot19}.
\end{remark}

\subsection{Organization of the paper}

In Section~\ref{sec:cc} we recall some well-known facts in convex projective geometry, as well as some results from \cite{dgk-ccHpq,dgk-proj-cc}.
In Section~\ref{sec:deform-Tits-repr} we recall Vinberg's theory of linear reflection groups and provide proofs of some basic results.
In Section~\ref{sec:Omega-Vin-max} we establish the maximality of the Tits--Vinberg domain (Proposition~\ref{prop:maximal}), and in Section~\ref{sec:Omega-min} we describe the minimal invariant convex domain of a reflection group.
In Section~\ref{sec:main_thm} we prove Theorem~\ref{thm:main}, and in Section~\ref{sec:mainproof-easy} we establish Theorem~\ref{thm:naive-cc-exists} and Corollaries \ref{cor:rel-hyp}, \ref{cor:cc-standard-subgroup}, \ref{cor:main2}, and~\ref{cor:main-racg}.
In Section~\ref{sec:deformation} we conclude with the proofs of Corollaries \ref{coro:nra_cc-int-refl}, \ref{coro:cc-int-refl}, and~\ref{cor:hyp-CG-Anosov}, and give examples showing that the dimension condition in Corollary~\ref{coro:nra_cc-int-refl} is optimal.

\subsection*{Acknowledgements}

Upon hearing about the first three authors' results on the right-angled case, Mike Davis kindly shared with them some personal notes he had been gathering on the topic over the past years.
We would like to thank him for his keen interest.
The last two authors are especially grateful to Ryan Greene, since they started thinking about this topic with him several years ago.
We also thank Anna Wienhard for helpful conversations. Finally, we thank the referee for many useful comments which helped improve the paper.

This project received funding from the European Research Council (ERC) under the European Union's Horizon 2020 research and innovation programme (ERC starting grant DiGGeS, grant agreement No 715982, and ERC consolidator grant GeometricStructures, grant agreement No 614733).
The authors also acknowledge support from the GEAR Network, funded by the National Science Foundation under grants DMS 1107452, 1107263, and 1107367 (``RNMS: GEometric structures And Representation varieties'').
J.D. was partially supported by an Alfred P. Sloan Foundation fellowship and by the National Science Foundation under grants DMS 1510254, DMS 1812216, and DMS 1945493.
F.G. and F.K. were partially supported by the Agence Nationale de la Recherche through the grant DynGeo (ANR-16-CE40-0025-01) and the Labex CEMPI (ANR-11-LABX-0007-01), and F.G. through the Labex IRMIA (11-LABX-0055).
Part of this work was completed while F.K. was in residence at the MSRI in Berkeley, California, for the program \emph{Geometric Group Theory} (Fall 2016) supported by NSF grant DMS 1440140, and at the INI in Cambridge, UK, for the program \emph{Nonpositive curvature group actions and cohomology} (Spring~2017) supported by EPSRC grant EP/K032208/1.
G.L. was partially supported by the DFG research grant ``Higher Teichm\"{u}ller Theory'' and by the National Research Foundation of Korea (NRF) grant funded by the Korean government (MSIT) (No 2020R1C1C1A01013667).
L.M. acknowledges support by the Centre Henri Lebesgue (ANR-11-LABX-0020 LEBESGUE).

\section{Reminders: Actions on properly convex subsets of projective space} \label{sec:cc}

In the whole paper, we fix a real vector space $V$ of dimension $n\geq 2$. We first recall some well-known facts in convex projective geometry, as well as some results from \cite{dgk-ccHpq,dgk-proj-cc}.
Here we work with the projective general linear group $\PGL(V)$, which naturally acts on the projective space $\PP(V)$; in the rest of the paper, we shall work with subgroups of $\GL(V)$, which also acts on $\PP(V)$ via the projection $\GL(V) \to \PGL(V)$.
All of the infinite discrete subgroups of $\GL(V)$ that we consider project injectively to discrete subgroups of $\PGL(V)$.

\subsection{Properly convex domains in projective space} \label{subsec:prop-conv-proj}

Recall that an open domain $\Omega$ in the projective space $\PP(V)$ is said to be \emph{convex} if it is contained and convex in some affine chart, \emph{properly convex} if it is convex and bounded in some affine chart, and \emph{strictly convex} if in addition its boundary does not contain any nontrivial projective line segment.
It is said to have \emph{$C^1$ boundary} if every point of the boundary of~$\Omega$ has a unique supporting hyperplane.

Let $\Omega$ be a properly convex open subset of $\PP(V)$, with boundary $\partial\Omega$.
Recall the \emph{Hilbert metric} $d_{\Omega}$ on~$\Omega$:
$$d_{\Omega}(y,z) := \frac{1}{2} \log \, [a,y,z,b] $$
for all distinct $y,z\in\Omega$, where $[\cdot,\cdot,\cdot,\cdot]$ is the cross-ratio on $\PP^1(\RR)$, normalized so that $[0,1,z,\infty]=z$, and where $a,b$ are the intersection points of $\partial\Omega$ with the projective line through $y$ and~$z$, with $a,y,z,b$ in this order (see Figure~\ref{disttt}).
The metric space $(\Omega,d_{\Omega})$ is complete (\ie Cauchy sequences converge) and proper (\ie closed balls are compact), and the group
$$\mathrm{Aut}(\Omega) := \{g\in\PGL(V) ~|~ g\cdot\Omega=\Omega\}$$
acts on~$\Omega$ by isometries for~$d_{\Omega}$.
As a consequence, any discrete subgroup of $\mathrm{Aut}(\Omega)$ acts properly discontinuously on~$\Omega$.

\begin{figure}[h!]
\centering
\begin{tikzpicture}
\filldraw[draw=black,fill=gray!20]
 plot[smooth,samples=200,domain=0:pi] ({4*cos(\x r)*sin(\x r)},{-4*sin(\x r)});
cycle;
\draw (-1,-2.5) node[anchor=north west] {$y$};
\fill [color=black] (-1,-2.5) circle (2.5pt);
\draw (1,-2) node[anchor=north west] {$z$};
\fill [color=black] (1,-2) circle (2.5pt);
\draw [smooth,samples=200,domain=-4:4] plot ({\x},{0.25*\x-2.25});
\draw (-1.97,-2.75) node[anchor=north west] {$a$};
\fill [color=black] (-1.98,-2.75) circle (2.5pt);
\draw (1.73,-1.75) node[anchor=north west] {$b$};
\fill [color=black] (1.63,-1.85) circle (2.5pt);
\draw (0.7,-3) node[anchor=north west] {$\Omega$};
\end{tikzpicture}
\caption{Hilbert distance}
\label{disttt}
\end{figure}

Let $V^*$ be the dual vector space of~$V$.
By definition, the \emph{dual convex set} of~$\Omega$ is
$$\Omega^* := \PP\big(\big\{ \varphi \in V^* ~|~ \varphi(x)<0\quad \forall x\in\overline{\widetilde{\Omega}}\big\}\big),$$
where $\overline{\widetilde{\Omega}}$ is the closure in $V\smallsetminus \{0\}$ of an open convex cone of $V$ lifting~$\Omega$.
The set $\Omega^*$ is a properly convex open subset of $\PP(V^*)$ which is preserved by the dual action of $\mathrm{Aut}(\Omega)$ on $\PP(V^*)$.

Projective line segments in $\Omega$ are always geodesics for the Hilbert metric $d_{\Omega}$.
When $\Omega$ is not strictly convex, there may be other geodesics as well.
However, a biinfinite geodesic of $(\Omega,d_{\Omega})$ always has well-defined, distinct endpoints in $\partial \Omega$, by \cite[Th.\,3]{fk05} (see also \cite[Lem.\,2.6]{dgk-ccHpq}).

\begin{remark} \label{rem:strict-convex-Lambda}
If $\Omega$ is strictly convex, then all $\Gamma$-orbits of~$\Omega$ have the same accumulation points in $\partial \Omega$.
Indeed, let $y\in\Omega$ and $(\gamma_n) \in \Gamma^{\NN}$ be such that $(\gamma_n\cdot y)$ converges to some $y_{\infty}\in \partial\Omega$. 
Let $z$ be  another point of $\Omega$.
After possibly passing to a subsequence, $(\gamma_n \cdot z)$ converges to some $z_\infty \in \partial \Omega$. By properness of the action of $\Gamma$ on $\Omega$, the limit $[y_\infty, z_\infty]$ of the sequence of compact intervals $(\gamma_n\cdot [y,z])$ is contained in $\partial \Omega$.
Strict convexity then implies that $y_\infty = z_\infty$.
This shows that any accumulation point of $\Gamma\cdot y$ is also an accumulation point of $\Gamma\cdot z$.
\end{remark}

\subsection{The proximal limit set and the full orbital limit set}

Recall that an element of $\PGL(V)$ is called \emph{proximal} in $\PP(V)$ if it has a unique eigenvalue of largest modulus and if this eigenvalue (which is then necessarily real) has multiplicity~1.
We shall use the following terminology as in \cite{dgk-ccHpq,dgk-proj-cc}.

\begin{definition} \label{def:prox-lim-set}
Let $\Gamma$ be a discrete subgroup of $\PGL(V)$.
The \emph{proximal limit set} of $\Gamma$ in $\PP(V)$ is the closure $\Lambda_{\Gamma}$ of the set of attracting fixed points of elements of~$\Gamma$ which are proximal in $\PP(V)$.
\end{definition}

When $\Gamma$ acts \emph{irreducibly} on $\PP(V)$, the proximal limit set $\Lambda_{\Gamma}$ was studied in \cite{gui90,ben97,ben00}.
In that setting, the action of $\Gamma$ on $\Lambda_{\Gamma}$ is minimal (\ie any orbit is dense) and $\Lambda_{\Gamma}$ is contained in any nonempty closed $\Gamma$-invariant subset of $\PP(V)$ (see \cite[Lem.\,3.6]{ben97} and \cite[Lem.\,2.5 \& Prop.\,3.1]{ben00}).

We shall use the following terminology from \cite{dgk-proj-cc}.

\begin{definition} \label{def:Lambdao}
Let $\Gamma$ be an infinite discrete subgroup of $\PGL(V)$ preserving a nonempty properly convex open subset $\Omega$ of $\PP(V)$.
\begin{itemize}
  \item The \emph{full orbital limit set} of $\Gamma$ in~$\Omega$ is the set $\Lambdao_{\Omega}(\Gamma) \subset \overline{\Omega}$ of all accumulation points of all $\Gamma$-orbits of~$\Omega$; it is $\Gamma$-invariant and contained in $\partial\Omega$.
  \item The \emph{convex core} of $\Gamma$ in~$\Omega$ is the convex hull $\Ccore_\Omega(\Gamma)$ of $\Lambdao_{\Omega}(\Gamma)$ in~$\Omega$.
  \item The action of $\Gamma$ on~$\Omega$ is \emph{convex cocompact} if $\Gamma$ acts cocompactly on~$\Ccore_\Omega(\Gamma)$.
\end{itemize}
\end{definition}

Thus an infinite discrete subgroup $\Gamma$ of $\PGL(V)$ is convex cocompact in $\PP(V)$ (Definition~\ref{def:cc-group}.\ref{item:def-cc}) if and only if it acts convex cocompactly (Definition~\ref{def:Lambdao}) on some properly convex open subset $\Omega \subset \PP(V)$. In this case the proximal limit set $\Lambda_{\Gamma}$ is always nonempty: see \cite[Prop.\,2.3.15]{bla-PhD}.

In the proof of Theorem~\ref{thm:main} (more specifically, Proposition~\ref{prop:naive-cc-implies-cc}), we shall use the following classical fact.

\begin{fact}[{\cite[Prop.\,3]{vey70}}] \label{fact:vey}
Let $\Gamma$ be a discrete subgroup of $\PGL(V)$ dividing (\ie acting properly discontinuously and cocompactly on) some properly convex open subset $\Omega$ of $\PP(V)$.
Then the convex hull of $\Lambda_{\Gamma}$ in~$\Omega$ is equal to~$\Omega$.
\end{fact}

\subsection{Maximal and minimal domains}

We shall use the following properties, which are due to Benoist \cite{ben00} in the irreducible case.
We denote by $\overline{Z}$ the closure of a set $Z$ in $\PP(V)$ or $\PP(V^*)$.

\begin{proposition}[{see \cite[Prop.\,3.1]{ben00} and \cite[Prop.\,3.10 \& Lem.\,3.13]{dgk-proj-cc}}] \label{prop:max-inv-conv}
Let $\Gamma$ be a discrete subgroup of $\PGL(V)$ preserving a nonempty properly convex open subset $\Omega$ of $\PP(V)$ and containing an element which is proximal in $\PP(V)$.
Let $\Lambda_{\Gamma}$ (\resp $\Lambda_{\Gamma}^*$) be the proximal limit set of $\Gamma$ in $\PP(V)$ (\resp $\PP(V^*)$).
Then
\begin{enumerate}
  \item\label{item:Lambda-prox-in-boundary} $\Lambda_{\Gamma} \subset \partial \Omega \cap \overline{\Gamma\cdot y}$ for all $y\in \Omega$, and $\Lambda_{\Gamma}^* \subset \partial \Omega^* \cap \overline{\Gamma\cdot [\varphi]}$ for all $[\varphi] \in\Omega^*$;
  \item\label{item:Lambda-prox-lift-Omega} more specifically, $\Omega$ and $\Lambda_{\Gamma}$ lift to cones $\widetilde{\Omega}$ and $\widetilde{\Lambda}_\Gamma$ of $V\smallsetminus\{0\}$ with $\widetilde{\Omega}$ properly convex containing $\widetilde{\Lambda}_{\Gamma}$ in its boundary, and $\Omega^*$ and $\Lambda_{\Gamma}^*$ lift to cones $\widetilde{\Omega}^*$ and $\widetilde{\Lambda}_\Gamma^*$ of $V^*\smallsetminus\{0\}$ with $\widetilde{\Omega}^*$ properly convex containing $\widetilde{\Lambda}_{\Gamma}^*$ in its boundary, such that $\varphi(x)\geq 0$ for all $x\in\widetilde{\Lambda}_{\Gamma}$ and $\varphi\in\widetilde{\Lambda}_{\Gamma}^*$;
  \item\label{item:Omega-max} for $\widetilde{\Lambda}_{\Gamma}^*$ as in \eqref{item:Lambda-prox-lift-Omega}, the set
  $$\Omega_{\max} := \PP\big(\big\{ x\in V ~|~ \varphi(x)>0\quad \forall\varphi\in\widetilde{\Lambda}_{\Gamma}^*\big\}\big)$$
is the connected component of $\PP(V) \smallsetminus \bigcup_{[\varphi]\in\Lambda_{\Gamma}^*} \PP\mathrm{Ker}(\varphi)$ containing~$\Omega$; it is $\Gamma$-inva\-riant, convex, and open in $\PP(V)$; any $\Gamma$-invariant properly convex open subset $\Omega'$ of $\PP(V)$ containing~$\Omega$ is contained in~$\Omega_{\max}$.
\end{enumerate}
\end{proposition}

In \eqref{item:Omega-max}, the convex set $\Omega_{\max}$ is not always properly convex. However, in the irreducible case it is:

\begin{proposition}[{\cite[Prop.\,3.1]{ben00}}] \label{prop:benoist-irred-Gamma}
Let $\Gamma$ be a discrete subgroup of $\PGL(V)$ preserving a nonempty properly convex open subset $\Omega$ of $\PP(V)$.

If $\Gamma$ acts irreducibly on $\PP(V)$, then $\Gamma$ always contains a proximal element and the set $\Omega_{\max}$ of Proposition~\ref{prop:max-inv-conv}.\eqref{item:Omega-max} is always properly convex.
Moreover, there is a smallest nonempty $\Gamma$-invariant convex open subset $\Omega_{\min}$ of $\Omega_{\max}$, namely the interior of the convex hull of $\Lambda_{\Gamma}$ in~$\Omega_{\max}$.

If moreover $\Gamma$ acts strongly irreducibly on $\PP(V)$ (\ie all finite-index subgroups of~$\Gamma$ act irreducibly), then $\Omega_{\max}$ is the unique maximal $\Gamma$-invariant properly convex open set in $\PP(V)$; it contains all other $\Gamma$-invariant properly convex open subsets.
\end{proposition}

\subsection{Convex cocompactness in projective space}

Recall the notions of naive convex cocompactness, convex cocompactness, and strong convex cocompactness from Definition~\ref{def:cc-group}, as well as the notion of convex cocompact action from Definition~\ref{def:Lambdao}.
We shall use the following characterizations and properties from \cite{dgk-proj-cc}.

\begin{proposition}{\cite[Th.\,1.15]{dgk-proj-cc}} \label{prop:cc-hyp-strong-cc}
Let $\Gamma$ be an infinite discrete subgroup of $\PGL(V)$.
Then $\Gamma$ is strongly convex cocompact in $\PP(V)$ if and only if it is convex cocompact in $\PP(V)$ and word hyperbolic.
\end{proposition}

\begin{proposition}[{\cite[Prop.\,10.3]{dgk-proj-cc}}] \label{prop:cc-no-unipotent}
Let $\Gamma$ be an infinite discrete subgroup of $\PGL(V)$.
If $\Gamma$ is naively convex cocompact in $\PP(V)$, then it does not contain any unipotent element.
\end{proposition}

\begin{proposition}[{\cite[Lem.\,4.7]{dgk-proj-cc}}] \label{prop:finite-index}
Let $\Gamma$ be a discrete subgroup of $\PGL(V)$ preserving a properly convex open subset $\Omega$ of $\PP(V)$, and let $\Gamma'$ be a finite-index subgroup of~$\Gamma$.
Then $\Gamma'$ is convex cocompact (\resp naively convex cocompact) in $\PP(V)$ if and only if $\Gamma$~is.
\end{proposition}

\begin{proposition}[{\cite[Th.\,1.16.(A)]{dgk-proj-cc}}] \label{prop:cc-dual}
Let $\Gamma$ be an infinite discrete subgroup of $\PGL(V)$.
The group $\Gamma$ is convex cocompact in $\PP(V)$ if and only if it is convex cocompact in $\PP(V^*)$ (for the dual action).
\end{proposition}

\begin{proposition}[{\cite[Prop.\,10.11]{dgk-proj-cc}}] \label{prop:cc-quotient}
Let $\Gamma$ be an infinite discrete subgroup of $\SL^{\pm}(V)$ acting trivially on some linear subspace $V_0$ of~$V$.
The group $\Gamma$ is convex cocompact in $\PP(V)$ if and only if the induced representation $\Gamma \to \SL^{\pm}(V/V_0)$ is injective and its image is convex cocompact in $\PP(V/V_0)$.
\end{proposition}

The last two statements imply the following.

\begin{corollary} \label{cor:cc-subspace}
Let $\Gamma$ be an infinite discrete subgroup of $\SL^{\pm}(V)$ preserving a linear subspace $V_1$ of~$V$, such that the induced action on $V/V_1$ is trivial.
The group $\Gamma$ is convex cocompact in $\PP(V)$ if and only if the induced representation $\Gamma \to \SL^{\pm}(V_1)$ is injective and its image is convex cocompact in $\PP(V_1)$.
\end{corollary}

\begin{proof}
For any subspace $V'$ of~$V$, we denote by $\mathrm{Ann}(V')\subset V^*$ the annihilator of $V'$, \ie the subspace of linear forms $\varphi\in V^*$ that vanish on~$V'$.
The dual ${V'}^*$ identifies with $V^*/\mathrm{Ann}(V')$.

Let $V_2$ be a complementary subspace of $V_1$ in $V$.
In a basis adapted to the decomposition $V=V_1\oplus V_2$, the group $\Gamma$ is expressed as a group of matrices of the form $\big ( \begin{smallmatrix} i(\gamma) & * \\ 0 & \mathrm{Id} \end{smallmatrix} \big )$, where $i : \Gamma\to\SL^{\pm}(V_1)$ is the restricted representation.
In the dual basis, which is adapted to the decomposition $V^* = \mathrm{Ann}(V_2)\oplus\mathrm{Ann}(V_1)$, the dual action of $\Gamma$ on~$V^*$ is given by $\big (\begin{smallmatrix} {}^{\top}\! i(\gamma)^{-1} & 0 \\ * & \mathrm{Id} \end{smallmatrix} \big)$.
In other words, $\Gamma$ acts trivially on the subspace $\mathrm{Ann}(V_1)$ of~$V^*$.
By Proposition~\ref{prop:cc-quotient}, the group $\Gamma$ is convex cocompact in $\PP(V^*)$ if and only if the induced representation $\gamma \mapsto {}^{\top}\! i(\gamma)^{-1}$ of $\Gamma$ is injective and its image is convex cocompact in $\PP(V^*/\mathrm{Ann}(V_1)) \simeq \PP(V_1^*)$.
Dualizing again, by Proposition~\ref{prop:cc-dual}, the group $\Gamma$ is convex cocompact in $\PP(V)$ if and only if the restricted representation $i$ is injective and its image is convex cocompact in $\PP(V_1)$.
\end{proof}

\section{Reminders: Vinberg's theory for Coxeter groups} \label{sec:deform-Tits-repr}

In this section we set up some notation and recall the basics of Vinberg's theory \cite{vin71} of linear reflection groups.

\subsection{Coxeter groups} \label{subsec:Coxeter-groups}

Let $W_S$ be a Coxeter group generated by a finite set of involutions $S = \{s_1, \ldots, s_N\}$, with presentation
\begin{equation} \label{eqn:WS}
W_S = \big\langle s_1,\dots,s_N ~|~ (s_i s_j)^{m_{i,j}}=1,\,\, \forall \ 1\leq i,j\leq N\big\rangle
\end{equation}
where $m_{i,i}=1$ and $m_{i,j} = m_{j,i} \in\{2, 3, \dotsc, \infty\}$ for $i\neq j$.
The \emph{Coxeter diagram} for $W_S$ is a labeled graph $\mathcal{G}_{W_S}$ such that
\begin{enumerate}[label=(\roman*)]
  \item the set of nodes (\ie vertices) of $\mathcal{G}_{W_S}$ is the set~$S$;
  \item two nodes $s_i, s_j \in S$ are connected by an edge $\overline{s_i s_j}$ of $\mathcal{G}_{W_S}$ if and only if $m_{i,j} \in \{ 3, 4, \dotsc, \infty \}$;
  \item the label of the edge $\overline{s_i s_j}$ is $m_{i,j}$.
\end{enumerate}
It is customary to omit the label of the edge $\overline{s_i s_j}$ if $m_{i,j}=3$.

For any subset $S'$ of $S$, the subgroup $W_{S'}$ of $W_S$ generated by $S'$ is the Coxeter group with generating set $S'$ and exponents $m'_{i,j}=m_{i,j}$ for $s_i,s_j \in S'$, see \cite[Chap.\,IV, Th.\,2]{bourbaki_coxeter}.
Such a subgroup $W_{S'}$ is called a \emph{standard subgroup} of~$W_S$.

The connected components of the graph $\mathcal{G}_{W_S}$ are graphs of the form $\mathcal{G}_{W_{S_{\ell}}}$, $\ell\in L$, where the $S_{\ell}$ form a partition of~$S$.
The subsets $S_{\ell}$ are called the \emph{irreducible components} of~$S$.
The group $W_S$ is the direct product of the standard subgroups $W_{S_{\ell}}$ for $\ell\in L$; these subgroups are called the \emph{irreducible factors} of~$W_S$.
The Coxeter group $W_S$ is called irreducible if it has only one irreducible factor, \ie $\mathcal{G}_{W_S}$ is connected.

\subsection{Representing the generators of $W_S$ by reflections in hyperplanes} \label{subsec:refl-basics}

We shall use the following terminology.

\begin{definition} \label{def:weak-compat}
An $N \times N$ real matrix $\Cart = (\Cart_{i,j})_{1\leq i,j\leq N}$ is \emph{weakly compatible} with the Coxeter group $W_S$ if
\begin{equation}\label{eqn:conn-comp-ref}
\left\{ \begin{array}{ll}
\Cart_{i,i} =2 &\mathrm{for} \ \mathrm{all}\ i, \\
\Cart_{i,j}=0\ &\mathrm{for}\ \mathrm{all}\ i\neq j\ \mathrm{with}\ m_{i,j} = 2, \\
\Cart_{i,j} < 0\ &\mathrm{for}\ \mathrm{all}\ i\neq j\ \mathrm{with}\ m_{i,j} \neq 2, \\
\Cart_{i,j}\Cart_{j,i} = 4 \cos^2(\pi/m_{i,j}) \ &\mathrm{for}\ \mathrm{all}\ i\neq j\ \mathrm{with}\ 2 < m_{i,j} < \infty.
\end{array} \right.
\end{equation}
\end{definition}

Consider $N$-tuples $\alpha = (\alpha_1,\dots,\alpha_N) \in {V^*}^N$ of linear forms and $v = (v_1,\dots,v_N) \in V^N$ of vectors.
If the matrix $\Cart = \left(\alpha_i(v_j)\right)_{1\leq i,j\leq N}$ is weakly compatible, then for any~$i$,
\begin{equation} \label{eqn:rho-s-i}
\rho(s_i) := \big(x \longmapsto x - \alpha_i(x) v_i\big)
\end{equation}
is a linear reflection of~$V$ in the hyperplane $\mathrm{Ker}(\alpha_i)$, and \eqref{eqn:rho-s-i} defines a group homomorphism $\rho: W_S \to \GL(V)$, \ie $\rho(s_i s_j)^{m_{i,j}} = \mathrm{Id}$ for all $1 \leq i,j \leq N$ (see \cite[Prop.\,6--7]{vin71}).

\begin{definition} \label{def:Cartan-matrix}
A representation $\rho: W_S \to \GL(V)$ is \emph{generated by weakly compatible reflections} if it is defined by \eqref{eqn:rho-s-i} for some $\alpha = (\alpha_1,\dots,\alpha_N) \in {V^*}^N$ and $v = (v_1,\dots,v_N) \in\nolinebreak V^N$ such that $\Cart = \left(\alpha_i(v_j)\right)_{1\leq i,j\leq N}$ is weakly compatible with~$W_S$.
In this case we say that $\Cart$ is the \emph{Cartan matrix} of $(\alpha,v)$, and a Cartan matrix for~$\rho$. 
\end{definition}

We note that for any~$i$, the pair $(\alpha_i,v_i)\in V^*\times V$ defining the reflection $\rho(s_i)$ in \eqref{eqn:rho-s-i} is not uniquely determined by $\rho(s_i)$.
Indeed, for any $\lambda_i \neq 0$ the pair $(\lambda_i \alpha_i, \lambda_i^{-1} v_i)$ yields the same reflection.
Changing $(\alpha_i,v_i)$ into $(\lambda_i \alpha_i, \lambda_i^{-1} v_i)$ changes the Cartan matrix $\Cart$ into its conjugate $D \Cart D^{-1}$ where $D = \mathrm{Diag}(\lambda_1, \ldots, \lambda_N)$ is diagonal.
Assuming that $\Cart$ is weakly compatible with $W_S$, we note that $D \Cart D^{-1}$ is also weakly compatible with $W_S$ if and only if the diagonal entries of $D$ associated to any irreducible factor of $W_S$ have constant sign.
In that case $D \Cart D^{-1} = |D| \Cart |D|^{-1}$, where $|D|$ denotes the positive diagonal matrix whose entries are the absolute values of those of $D$.

\begin{definition} \label{def:equivalent}
Two $N \times N$ matrices which are weakly compatible with $W_S$ are considered \emph{equivalent} if they differ by conjugation by a positive diagonal matrix.
\end{definition}

While the Cartan matrix $\Cart$ is not uniquely determined by the representation $\rho$, its equivalence class is. 
In particular, for any $1\leq i,j\leq N$, the product $\Cart_{i,j} \Cart_{j,i}$ is uniquely determined by~$\rho$.
The quantity $\Cart_{i,j} \Cart_{j,i}$ varies as an algebraic function, invariant under conjugation, as $\rho$ ranges over the semialgebraic set of representations of $W_S$ to $\GL(V)$ generated by weakly compatible reflections.

\begin{remark} \label{rem:Cartan-matrix}
There exist representations of $W_S$ taking the generators to reflections as in \eqref{eqn:rho-s-i} for which the matrix $\Cart = (\alpha_i(v_j))_{1\leq i,j\leq N}$ is not weakly compatible with~$W_S$: for instance, when $\Cart$ satisfies~\eqref{eqn:conn-comp-ref} with $\pi$ replaced by $2\pi$.
The representations generated by weakly compatible reflections form an open and closed subset of $\Hom(W_S, \GL(V))$, containing the set $\mathrm{Hom}^{\mathrm{ref}}(W_S,\mathrm{GL}(V))$ of Sections~\ref{subsec:intro-Vinberg-theory} and~\ref{subsec:Vinberg-constr} below.
\end{remark}

The following notation and terminology will be used throughout the rest of the paper.

\begin{definition} \label{def:V-v-alpha}
Let $\rho: W_S \to \GL(V)$ be a representation generated by weakly compatible reflections, associated to some $\alpha = (\alpha_1,\dots,\alpha_N) \in {V^*}^N$ and $v = (v_1,\dots,v_N) \in V^N$.
Let $V_v$ be the span of $v_1,\dots,v_N$, and $V_{\alpha}$ the intersection of the kernels of $\alpha_1,\dots,\alpha_N$. 
We say that the representation $\rho$ (or its image $\rho(W_S)$) is \emph{reduced} if $V_{\alpha} = \{0\}$, and \emph{dual-reduced} if $V_v=V$.
\end{definition}

Note that $V_\alpha$ is the subspace of~$V$ of $\rho(W_S)$-fixed vectors, and $V_v$ is another $\rho(W_S)$-invariant subspace of $V$.

\begin{remark} \label{rem:notation_rhos}
In general, a representation $\rho : W_S \to \mathrm{GL}(V)$ generated by weakly compatible reflections induces three representations:
\begin{itemize}
  \item $\rho_v$ on $V_v$ which is a dual-reduced representation,
  \item $\rho^\alpha$ on $V^\alpha:=V/V_\alpha$ which is a reduced representation,
  \item $\rho_v^\alpha$ on $V^\alpha_v := V_v/(V_v \cap V_\alpha)$ which is a reduced and dual-reduced representation.
\end{itemize}
These three representations are still generated by weakly compatible reflections with the same Cartan matrix $\Cart = (\alpha_i(v_j))_{1\leq i,j\leq N}$.
\end{remark}

\subsection{Representations of $W_S$ as a reflection group} \label{subsec:Vinberg-constr}

Let $\rho : W_S\to\GL(V)$ be a representation of $W_S$ generated by weakly compatible reflections (Definition~\ref{def:Cartan-matrix}), associated to some $\alpha = (\alpha_1,\dots,\alpha_N) \in {V^*}^N$ and $v = (v_1,\dots,v_N) \in V^N$; by definition, the Cartan matrix $\Cart = (\alpha_i(v_j))_{1\leq i,j\leq N}$ satisfies~\eqref{eqn:conn-comp-ref}.
The cone
\begin{equation} \label{eqn:Delta-tilde}
\widetilde{\Delta} := \big\{ x \in V ~|~ \alpha_i(x) \leq 0 \quad\forall 1\leq i\leq N\big\}
\end{equation}
is called the closed \emph{fundamental polyhedral cone} in $V$ associated to $(\alpha,v)$.
As mentioned in Section~\ref{subsec:intro-Vinberg-theory}, assuming that
\begin{equation} \label{eqn:not_empt_int}
\text{the convex polyhedral cone } \widetilde{\Delta} \text{ has nonempty interior},
\end{equation}
Vinberg \cite{vin71} proved that this cone has $N$ sides, and that $\rho(\gamma)\cdot \operatorname{Int}\,(\widetilde \Delta) \cap \operatorname{Int}\,(\widetilde \Delta) = \emptyset$ for all $\gamma \in W_S \smallsetminus \{e\}$ if and only if the Cartan matrix $\Cart$ further satisfies
\begin{equation} \label{eqn:refl-group}
\Cart_{i,j}\Cart_{j,i} \geq 4\ \mathrm{for}\ \mathrm{all}\ i\neq j\ \mathrm{with}\ m_{i,j} = \infty.
\end{equation}
In that case, $\rho$ is injective, and the stabilizer of any face of~$\widetilde{\Delta}$ coincides with the standard subgroup generated by the reflections in the hyperplanes $\mathrm{Ker}(\alpha_i)$ containing that face.

\begin{definition} \label{def:compatible}
An $N \times N$ real matrix $\Cart = (\Cart_{i,j})_{1\leq i,j\leq N}$ is \emph{compatible} with the~Cox\-eter group $W_S$ if $\Cart$ is weakly compatible with $W_S$ (Definition~\ref{def:weak-compat}) and satisfies \eqref{eqn:refl-group}.
\end{definition}

A representation $\rho : W_S\to\GL(V)$ generated by reflections in hyperplanes such that the Cartan matrix $\Cart = (\alpha_i(v_j))_{1\leq i,j\leq N}$ is compatible and such that \eqref{eqn:not_empt_int} holds is a \emph{representation of~$W_S$ as a reflection group  in~$V$} as in Definition~\ref{def:refl-group}.
We emphasize that \eqref{eqn:not_empt_int} always holds if $W_S$ is irreducible and nonaffine (see Remark~\ref{rem:neg-type-interpret}).
As in Section~\ref{subsec:intro-Vinberg-theory}, we denote by $\mathrm{Hom}^{\mathrm{ref}}(W_S,\mathrm{GL}(V))$ the set of representations of $W_S$ as a reflection group  in~$V$.

Consider $\rho\in\mathrm{Hom}^{\mathrm{ref}}(W_S,\mathrm{GL}(V))$.
By \cite[Th.\,2]{vin71}, the group $\rho(W_{S})$ acts properly discontinuously on a nonempty convex open cone $\widetilde{\Omega}_{\scriptscriptstyle\mathrm{TV}}$ of~$V$, namely the interior of the union of all $\rho(W_{S})$-translates of the fundamental polyhedral cone~$\widetilde \Delta$.
We shall call this convex cone the \emph{Tits--Vinberg cone}; it is sometimes called simply the \emph{Tits cone}.
It is equal to~$V$ if $W_S$ is finite, and does not contain~$0$ if $W_S$ is infinite and irreducible.

\begin{fact}[{\cite[Th.\,2]{vin71}}] \label{fact:Delta-flat}
The set $\widetilde \Delta^{\flat} := \widetilde{\Delta} \cap \widetilde{\Omega}_{\scriptscriptstyle\mathrm{TV}}$ is equal to $\widetilde \Delta$ minus its faces of infinite stabilizer; it is a fundamental domain for the action of $\rho(W_S)$ on~$\widetilde{\Omega}_{\scriptscriptstyle\mathrm{TV}}$.
The quotient $\mathcal{Q} := \rho(W_S)\backslash\widetilde{\Omega}_{\scriptscriptstyle\mathrm{TV}}$ is an orbifold homeomorphic to $\widetilde \Delta^{\flat}$ with mirrors on the reflection walls.
(In particular, any subset $\mathcal{Q}' \subset \mathcal{Q}$ is homeomorphic to its preimage in $\widetilde{\Delta}^\flat$, which can therefore be seen as a fundamental domain for the action of $\rho(W_S)$ on the full preimage of $\mathcal{Q}'$ in~$\widetilde{\Omega}_{\scriptscriptstyle\mathrm{TV}}$.)
\end{fact}

The cone $\widetilde{\Omega}_{\scriptscriptstyle\mathrm{TV}} \smallsetminus \{0\}$ projects to a nonempty $\rho(W_S)$-invariant open subset $\OhmVin$ of $\PP(V)$, which is convex whenever $W_S$ is infinite and irreducible; we shall call it the \emph{Tits--Vinberg domain}.
The action of $W_S$ on $\OhmVin$ is still properly discontinuous.
A fundamental domain $\Delta^{\flat}$ for the action on $\OhmVin$ is obtained from the projection $\Delta$ of $\widetilde \Delta$ to $\PP(V)$, as before, 
by removing the faces of infinite stabilizer.

\begin{remark} \label{rem:convex_chamber}
In fact, \cite[Th.\,2]{vin71} shows that the union of all $\rho(W_{S})$-translates of~$\widetilde{\Delta}$ (not only of $\widetilde{\Delta}^{\flat}$) is also convex.
\end{remark}

\begin{example}[{see \cite[\S\,2]{vin71} and Figure~\ref{fig:Z2xZ2}}] \label{ex:N=2}
Suppose $N=2$ and $m_{1,2} = \infty$.
Then $W_S \simeq (\ZZ/2\ZZ) * (\ZZ/2\ZZ)$ is of type~$\widetilde{A}_1$ (see Appendix~\ref{appendix:classi_diagram}).
The infinite cyclic subgroup generated by $s_1s_2$ has index two in~$W_S$.
Consider $\rho\in\mathrm{Hom}^{\mathrm{ref}}(W_S,\mathrm{GL}(V))$ with Cartan matrix $\Cart = (\Cart_{i,j})_{1\leq i,j\leq 2}$.

(i) Suppose $t := \Cart_{1,2}\Cart_{2,1} > 4$.
Then $v_1, v_2$ are linearly independent and $V = V_v \oplus V_{\alpha}$.
The polyhedral cone $\widetilde{\Delta}$ is the set
\begin{equation} \label{eqn:Delta-N=2}
\left\{ \lambda_1 v_1 + \lambda_2 v_2 ~\Big|~ \lambda_1,\lambda_2>0 \:\:\mathrm{and}\:\: \frac{2}{|\Cart_{2,1}|} \leq \frac{\lambda_1}{\lambda_2} \leq \frac{|\Cart_{1,2}|}{2} \right\} + V_{\alpha}.
\end{equation}
The element $\rho(s_1s_2)$ acts on $V_{\alpha}$ trivially, and on $V_v$ as the matrix $\big(\begin{smallmatrix} -1+\Cart_{1,2}\Cart_{2,1} & \Cart_{1,2}\\ -\Cart_{2,1} & -1\end{smallmatrix}\big)$ (in the basis $(v_1,v_2)$).
In particular, $\rho(s_1s_2)$ and $\rho(s_1s_2)^{-1}$ are proximal in $\PP(V)$, with attracting fixed points $[x_+],[x_-]\in\PP(V)$ where $x_{\pm} = (t\pm\sqrt{t(t-4)})v_1 -2\Cart_{2,1}v_2$, and
$$\widetilde{\Omega}_{\scriptscriptstyle\mathrm{TV}} = \RR_{>0}\,x_+ + \RR_{>0}\,x_- + V_{\alpha}.$$
Since the induced action of $\rho(W_S)$ on the image of $\widetilde{\Omega}_{\scriptscriptstyle\mathrm{TV}}$ in $\PP(V/V_{\alpha})$ is cocompact and the induced representation $\Gamma \to \SL^{\pm}(V/V_\alpha)$ is injective, Proposition~\ref{prop:cc-quotient} implies that $\rho(W_S)$ acts convex cocompactly on some properly convex open subset
$\Omega$ of $\PP(V)$ (Definition~\ref{def:Lambdao}), with $\Lambdao_\Omega(\rho(W_S)) = \{\,[x_+],[x_-]\,\}$.

Suppose now that $t := \Cart_{1,2}\Cart_{2,1} = 4$.
By considering the Cartan matrix $\big(\begin{smallmatrix} 2 & \Cart_{1,2}\\ 4/\Cart_{1,2} & 2\end{smallmatrix}\big)$, we see that $u := \Cart_{1,2} v_1 - 2  v_2 \in V_v \cap V_{\alpha}$.
We have assumed that \eqref{eqn:not_empt_int} holds, hence $\alpha_1$ and~$\alpha_2$ are linearly independent.
Choose $w \in \mathrm{Ker}(\alpha_2)$ such that $\alpha_1(w) = 1$.

(ii) On the one hand, if $u = 0$, then $V = \mathrm{span}(v_1,w) \oplus V_{\alpha}$; the elements $\rho(s_1)$ and $\rho(s_2)$ act on $V_{\alpha}$ trivially and on $\mathrm{span}(v_1,w)$ as the matrices $(\begin{smallmatrix} -1 & -1\\ ~0 & 1\end{smallmatrix})$ and $(\begin{smallmatrix} -1 & 0\\ ~~0 & 1\end{smallmatrix})$, respectively. 

(iii) On the other hand, if $u \neq 0$, then $V = \mathrm{span}(u,v_1,w) \oplus V'_{\alpha}$ where $V'_{\alpha}$ is any complement of $\mathbb{R} u$ in $V_{\alpha}$; the elements $\rho(s_1)$ and $\rho(s_2)$ act on $V'_{\alpha}$ trivially and on $\mathrm{span}(u,v_1,w)$ as the matrices $\left(\begin{smallmatrix} 1 & ~0 & ~0 \\ 0 & -1 & -1 \\ 0 & ~0 & ~1\end{smallmatrix}\right)$ and $\left(\begin{smallmatrix} 1 & 2/\Cart_{1,2} & 0 \\ 0 & -1 & 0 \\ 0 & ~0 & 1\end{smallmatrix}\right)$, respectively.

In both cases (ii) and~(iii), the element $\rho(s_1 s_2)\in\GL(V)$ is unipotent (with a unique nontrivial Jordan block, of size $2$ and $3$ respectively) and the Tits--Vinberg domain $\OhmVin$ is an affine chart of $\PP(V)$, namely the complement of the projective hyperplane $\PP(V_v + V_{\alpha})$. 
The Tits--Vinberg cone $\widetilde{\Omega}_{\scriptscriptstyle\mathrm{TV}}$ is the connected component of $V \smallsetminus (V_v+V_{\alpha})$ (an open halfspace of~$V$) that contains $-w$.

In all three cases, any $\rho(W_S)$-invariant properly convex open cone of~$V$ is contained in either~$\widetilde{\Omega}_{\scriptscriptstyle\mathrm{TV}}$ or~$-\widetilde{\Omega}_{\scriptscriptstyle\mathrm{TV}}$. 
Indeed, in case (i) the orbit of any point of $\mathbb{R}_{<0}x_- + \mathbb{R}_{>0}x_+ + V_\alpha$ escapes in the directions of both $x_+$ and $-x_+$ (along two branches of hyperbola), causing the closure of its convex hull to contain a full line.
In case~(ii), there is no $\rho(W_S)$-invariant properly convex open cone in~$V$ at all: any orbit in $V\smallsetminus (\mathbb{R} v_1 + V_\alpha)$ escapes in the directions of both $v_1$ and $-v_1$.
In case~(iii), any orbit in $(V_v+V_\alpha)\smallsetminus (\mathbb{R} u +V_\alpha)$ escapes in the directions of both $u$ and $-u$, hence every $\rho(W_S)$-invariant properly convex open cone must stay disjoint from $V_v+V_\alpha=\partial \widetilde{\Omega}_{\scriptscriptstyle\mathrm{TV}}$.
\end{example}

\begin{figure}[h]
\centering
\labellist
\small\hair 2pt
\pinlabel {(i)} [u] at 9 130
\pinlabel {$0$} [u] at 80 119
\pinlabel {$v_1$} [u] at 24 141
\pinlabel {\rotatebox{-68}{$\mathrm{Ker}(\alpha_1)$}} [u] at 61 197
\pinlabel {$v_2$} [u] at 136 140
\pinlabel {\rotatebox{70}{$\mathrm{Ker}(\alpha_2)$}} [u] at 99 197
\pinlabel {$x_+$} [u] at 35 160
\pinlabel {$x_-$} [u] at 130 160
\pinlabel {$\widetilde{\Delta}$} [u] at 80 214
\pinlabel {$\widetilde{\Omega}_{\scriptscriptstyle\mathrm{TV}}$} [u] at 80 230
\pinlabel {(ii)} [u] at 179 130
\pinlabel {$0$} [u] at 261 121
\pinlabel {$v_1$} [u] at 206 123
\pinlabel {\rotatebox{-68}{$\mathrm{Ker}(\alpha_1)$}} [u] at 242 202
\pinlabel {$v_2$} [u] at 316 123
\pinlabel {\rotatebox{70}{$\mathrm{Ker}(\alpha_2)$}} [u] at 280 202
\pinlabel {$\widetilde{\Delta}$} [u] at 261 213
\pinlabel {$\widetilde{\Omega}_{\scriptscriptstyle\mathrm{TV}}$} [u] at 261 230
\pinlabel {(iii)} [u] at 75 12
\pinlabel {\rotatebox{-87}{$\mathbb{P}\mathrm{Ker}(\alpha_1)$}} [u] at 150 40
\pinlabel {\rotatebox{-87}{$\mathbb{P}\mathrm{Ker}(\alpha_2)$}} [u] at 184 40
\pinlabel {$[v_1]$} [u] at 78 60
\pinlabel {$[v_1]$} [u] at 257 100
\pinlabel {$[v_2]$} [u] at 78 100
\pinlabel {$[v_2]$} [u] at 257 60
\pinlabel {$[u]$} [u] at 168 0
\pinlabel {$\Delta$} [u] at 167 96
\pinlabel {$\OhmVin$} [u] at 195 98
\endlabellist
\includegraphics[scale=.9]{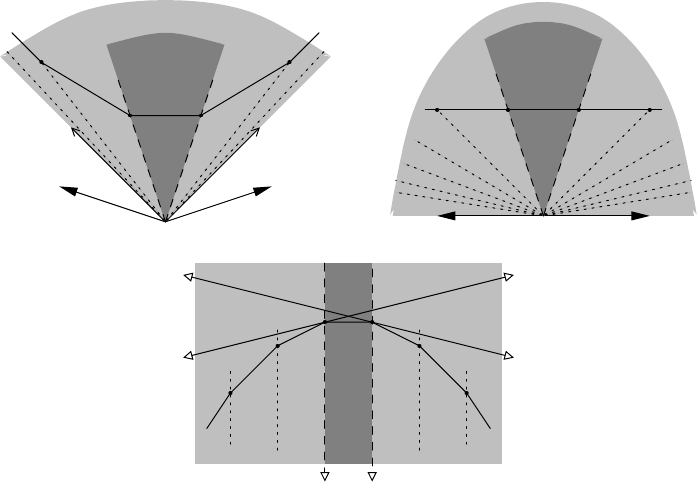}
\caption{Infinite Coxeter groups on two generators as in Example~\ref{ex:N=2}, cases (i)--(ii)--(iii).
The groups are shown acting on $\RR^2$, $\RR^2$ and $\mathbb{P}(\RR^3)$ respectively.
In the third panel, the points $[v_1]$, $[v_2]$ and $[u]$ are at infinity, and $\OhmVin$ is the full affine chart.}
\label{fig:Z2xZ2}
\end{figure}

\begin{remark} \label{rem:Tits-geom}
We shall say that $\rho : W_S\to\GL(V)$ is \emph{symmetrizable} if the Cartan matrix $\Cart$ is equivalent (Definition~\ref{def:equivalent}) to a symmetric matrix; in this case, $\rho$ preserves a (possibly degenerate) symmetric bilinear form on~$V$.
When $n:=\dim V=N:=\#S$, one may obtain a symmetrizable representation $\rho : W_S\to\GL(V)$ of~$W_S$ as a reflection group (which is reduced in the sense of Definition~\ref{def:V-v-alpha}) by choosing any nonnegative numbers $(\lambda_{i,j})_{1\leq i,j\leq N}$, setting $\Cart_{i,j} = \Cart_{j,i}$ to be $-2-\lambda_{i,j}$ if $m_{i,j}=\infty$ and $-2\cos(\pi/m_{i,j})$ otherwise, and taking for $\alpha_i$ the $i$-th element of the dual canonical basis of~$V$, and for $v_i$ the $i$-th column of~$\Cart$.
When all the $\lambda_{i,j}$ vanish, $\Cart$ is the matrix $\mathrm{Cos}(W_S) := (-2\cos(\pi/m_{i,j}))_{1\leq i,j\leq N}$ (with the convention $\pi/\infty=0$) and $\rho$ is the so-called \emph{Tits geometric representation} \cite{tit61}.
\end{remark}

\subsection{Types of Coxeter groups} \label{subsec:Coxeter-type}

By \cite{mv00}, every irreducible Coxeter group $W_S$ is either
\begin{itemize}
  \item \emph{spherical} (\ie finite),
  \item \emph{affine} (\ie infinite and virtually abelian), or
  \item \emph{large} (\ie there exists a surjective homomorphism of a finite-index subgroup of $W_S$ onto a nonabelian free group).
\end{itemize}

Any spherical (\resp affine) irreducible Coxeter group $W_S$ in $N=\#S$ generators acts irreducibly on the sphere of dimension $N-1$ (\resp properly discontinuously and cocompactly on Euclidean space of dimension $N-1$) with a simplex as a fundamental domain (see \cite{classi_cox_poly_coxeter} and the next Section~\ref{subsec:neg-type}).

\subsection{Types of compatible matrices} \label{subsec:neg-type}

We now assume that the Coxeter group $W_S$ is irreducible.
Let $\Cart \in \mathcal{M}_N(\RR)$ be a matrix \emph{compatible} with $W_S$ (Definition~\ref{def:compatible}).
By construction, the matrix $\Cart' := 2\,\mathrm{Id}-\Cart \in \mathcal{M}_N(\RR)$ has nonnegative entries.
By the Perron--Frobenius theorem, $\Cart'$ admits an eigenvector with positive entries corresponding to the highest eigenvalue of~$\Cart'$; since $W_S$ is irreducible, this vector is unique up to scale (see \eg \cite[Th.\,1.5]{sen06}).
It is also an eigenvector for the lowest eigenvalue of~$\Cart$.

\begin{definition} \label{def:type}
The compatible matrix $\Cart$ is of \emph{negative} (\resp \emph{zero}, \resp \emph{positive}) \emph{type} if the lowest eigenvalue of~$\Cart$ is negative (\resp zero, \resp positive).
\end{definition}

The type of~$\Cart$ only depends on its equivalence class (Definition~\ref{def:equivalent}).

\begin{remark}
Vinberg \cite{vin71} uses the following additional terminology: for an irreducible Coxeter group and a representation $\rho:W_S \to \GL(V)$ of $W_S$ as a reflection group, $\rho(W_S)$ is \emph{elliptic} (\resp \emph{parabolic}) if $\rho$ is reduced and $\Cart$ is of positive type (\resp zero type); $\rho(W_S)$ is \emph{hyperbolic} if $\rho$ is reduced, dual-reduced, and preserves a Lorentzian quadratic form on~$V$.
We shall not need this terminology in the current paper.
\end{remark}

\begin{remark} \label{rem:neg-type-interpret}
Suppose $\Cart = (\alpha_i(v_j))_{1\leq i,j\leq N}$ for some $\alpha = (\alpha_1,\dots,\alpha_N) \in {V^*}^N$ and $v = (v_1,\dots,v_N) \in V^N$.
Let $t=(t_1,\dots,t_N)\in (\RR_{>0})^N$ be the Perron--Frobenius eigenvector of $\Cart' = 2\,\mathrm{Id}-\Cart$.
Let $x = \sum_{j=1}^N t_j v_j \in V_v$.
The $i$-th entry of $\Cart t$ is $\alpha_i(x)$.
\begin{enumerate}
  \item\label{item:neg-type-interpret} If $\Cart$ is of negative (\resp positive) type, then $x$ (\resp $-x$) belongs to the interior $\mathrm{Int} (\widetilde \Delta)$ of the fundamental polyhedral cone $\widetilde{\Delta}$ of \eqref{eqn:Delta-tilde}.
  In particular, $\mathrm{Int} (\widetilde \Delta)$ is nonempty, \ie \eqref{eqn:not_empt_int} holds.
  \item If $\Cart$ is of zero type, then $x$ belongs to $V_v\cap V_{\alpha}$.
  Determining whether \eqref{eqn:not_empt_int} holds is more subtle: see \cite[Prop.\,13]{vin71} (and also Remarks \ref{rem:rho-v-alpha-refl-group} and~\ref{rem:JWST} below).
\end{enumerate}
\end{remark}

\begin{fact}[{\cite[Th.\,3]{vin71}}] \label{fact:type-sign-entries}
Let $\Cart \in \mathcal{M}_N(\RR)$ be a matrix compatible with the irreducible Coxeter group~$W_S$, and let $\tau = (\tau_1,\dots,\tau_N) \in \RR^N$.
\begin{enumerate}
  \item\label{item:neg-type-sign-entries} If $\Cart$ is of negative type and both $\tau$ and $\Cart \tau$ have all their entries $\geq 0$, then $\tau=0$.
  \item\label{item:zero-type-sign-entries} If $\Cart$ is of zero type and $\Cart \tau$ has all its entries $\geq 0$, then $\Cart \tau=0$.
  \item\label{item:neg-type-conclusion} If $\tau$ and $-\Cart \tau$ have all their entries $> 0$, then $\Cart$ is of negative type.
  \item\label{item:zero-type-conclusion} If $\tau$ has all its entries $> 0$ and $\Cart \tau = 0$, then $\Cart$ is of zero type.
\end{enumerate}
\end{fact}

\begin{remark} \label{rem:rho-v-alpha-refl-group}
Let $\rho : W_S\to\GL(V)$ be a representation of $W_S$ as a reflection group with Cartan matrix $\Cart$, and let $\rho_v$, $\rho^{\alpha}$, and $\rho_v^{\alpha}$ be the induced representations of Remark~\ref{rem:notation_rhos}.
If $\Cart$ is of negative or positive type, then $\rho_v$, $\rho^{\alpha}$, and $\rho_v^{\alpha}$ are representations of $W_S$ as a reflection group, by Remark~\ref{rem:neg-type-interpret}.\eqref{item:neg-type-interpret}; in particular, these representations are injective.
If $\Cart$ is of zero type, then $\rho^{\alpha}$ is a representation of $W_S$ as a reflection group (see Fact~\ref{fact:types}.\eqref{item:affine_case_zero_type}), but $\rho_v$ and $\rho_v^{\alpha}$ are never representations of~$W_S$ as a reflection group, as \eqref{eqn:not_empt_int} fails by Fact~\ref{fact:type-sign-entries}.\eqref{item:zero-type-sign-entries}.
\end{remark}

As in Section~\ref{section:subspace_of_CC_rep}, recall that a representation of a group into $\GL(V)$ is \emph{semisimple} if it is a direct sum of irreducible representations.
The following important fact will be used several times throughout the paper; we explain how it follows from Vinberg's work in Section~\ref{subsec:proof-fact-types} below.

\begin{fact}[{\cite[Lem.\,13 \& Prop.\,16--21--22--23]{vin71}}] \label{fact:types}
Let $W_S$ be an irreducible Coxeter group and $\Cart \in \mathcal{M}_N(\RR)$ a compatible Cartan matrix for~$W_S$ (Definition~\ref{def:compatible}).
\begin{enumerate}
  \item\label{item:spherical_case} The group $W_S$ is spherical if and only if $\Cart$ is of positive type; in this case $\Cart$ is symmetrizable.
  \item\label{item:affine_case} Suppose $W_S$ is affine.
  \begin{enumerate}
    \item\label{item:affine_case_zero_type} If $\Cart$ is of zero type, then $\Cart$ is symmetrizable and any representation\linebreak $\rho : W_S\to\GL(V)$ of~$W_S$ as a reflection group (Definition~\ref{def:refl-group}) with Cartan matrix $\Cart$ is non-semisimple; the induced representation $\rho^{\alpha} : W_S\to\GL(V^{\alpha})$ of Remark~\ref{rem:notation_rhos} is a representation of $W_S$ as a reflection group whose image acts properly discontinuously and cocompactly on the affine chart $\PP(V^{\alpha}) \smallsetminus \PP(V_v^{\alpha})$ of $\PP(V^{\alpha})$, preserving some Euclidean metric; $\PP(V^{\alpha}) \smallsetminus \PP(V_v^{\alpha})$ is the Tits--Vinberg domain for $\rho^\alpha (W_S)$.
    \item\label{item:affine_case_negative_type} If $\Cart$ is of negative type, then $\det \Cart\neq 0$ and $W_S$ is of type~$\widetilde{A}_{N-1}$ with $N\geq 2$ (see Appendix~\ref{appendix:classi_diagram}); if furthermore $N\geq 3$, then $\Cart$ is not symmetrizable.
  \end{enumerate}
  \item\label{item:large_case} If $W_S$ is large, then $\Cart$ is of negative type.
\end{enumerate}
\end{fact}

\begin{example} \label{ex:N=2bis}
Suppose $N=2$ and $W_S$ is infinite.
Then $m_{1,2} = \infty$ and $W_S$ is affine of type~$\widetilde{A}_1$ as in Example~\ref{ex:N=2}.
Let $\Cart = (\Cart_{i,j})_{1\leq i,j\leq 2}$ be a compatible Cartan matrix for~$W_S$, and let $t := \Cart_{1,2} \Cart_{2,1}$.
If $t>4$ (case~(i) of Example~\ref{ex:N=2}), then $\Cart$ is of negative type and $\det \Cart = 4-t \neq 0$.
If $t=4$ (cases (ii)--(iii) of Example~\ref{ex:N=2}), then $\Cart$ is of zero type; the group $\rho^{\alpha}(W_S)$ acts properly discontinuously and cocompactly on the affine chart of $\PP(V^{\alpha}) = \PP(V/V_{\alpha}) \simeq \PP(\RR^2)$ which is the complement of $\PP((V_v+V_{\alpha})/V_{\alpha})$; this affine chart is the image of the Tits--Vinberg domain $\OhmVin = \PP(V)\smallsetminus\PP(V_v+V_{\alpha})$ by the natural projection $\PP(V)\smallsetminus\PP(V_{\alpha}) \to \PP(V^{\alpha})$.
\end{example}

\begin{remark} \label{rem:open-closed-in-Hom-df}
Let $\Hom^{\mathrm{fd}}(W_S, \GL(V))$ be the space of faithful and discrete representations of $W_S$ into $\GL(V)$ and $\Hom^{\mathrm{wc}}(W_S, \GL(V))$ (\resp $\Hom^{\mathrm{c}}(W_S, \GL(V))$) the space of representations generated by reflections for which the corresponding Cartan matrix $\Cart = (\alpha_i(v_j))_{1\leq i,j\leq N}$ is weakly compatible (\resp compatible) as in Definitions \ref{def:weak-compat} and~\ref{def:Cartan-matrix} (\resp Definition~\ref{def:compatible}).
\begin{enumerate}
  \item\label{item:Hom-ref-fd-wc-c} For any irreducible Coxeter group $W_S$ we have
$$\Hom^{\mathrm{ref}}(W_S, \GL(V)) \subset \Hom^{\mathrm{fd}}(W_S, \GL(V)) \,\cap\, \Hom^{\mathrm{wc}}(W_S, \GL(V)) \subset \Hom^{\mathrm{c}}(W_S, \GL(V)).$$
Indeed, the first inclusion holds by \cite{vin71} as explained in Section~\ref{subsec:Vinberg-constr}.
For the second inclusion, note that if $\rho \in \Hom^{\mathrm{fd}}(W_S, \GL(V)) \cap \Hom^{\mathrm{wc}}(W_S, \GL(V))$, then for any $1\leq i,j\leq N$ with $m_{i,j}=\infty$ we must have $\Cart_{i,j}\Cart_{j,i}\geq 4$: indeed, otherwise $\rho(s_i s_j)$ would be conjugate to a rotation of finite order (contradicting faithfulness) or to a rotation of infinite order (contradicting discreteness).
  \item\label{item:open-closed-in-Hom-df} If the irreducible Coxeter group $W_S$ is not affine, then $\Hom^{\mathrm{ref}}(W_S, \GL(V))$ is open and closed in $\Hom^{\mathrm{fd}}(W_S, \GL(V))$.
Indeed, Remark~\ref{rem:neg-type-interpret}.\eqref{item:neg-type-interpret} and Fact~\ref{fact:types}.\eqref{item:spherical_case}--\eqref{item:large_case} imply that in this case $\Hom^{\mathrm{ref}}(W_S, \GL(V))$ is equal to $\Hom^{\mathrm{c}}(W_S, \GL(V))$, hence to $\Hom^{\mathrm{fd}}(W_S, \GL(V)) \,\cap\, \Hom^{\mathrm{wc}}(W_S, \GL(V))$ by \eqref{item:Hom-ref-fd-wc-c} above; on the other hand, in general $\Hom^{\mathrm{wc}}(W_S, \GL(V))$ is open and closed in $\Hom(W_S, \GL(V))$ as observed in Remark~\ref{rem:Cartan-matrix}.
\end{enumerate}
\end{remark}

By contrast with Remark~\ref{rem:open-closed-in-Hom-df}.\eqref{item:open-closed-in-Hom-df}, if $W_S$ is affine, then $\Hom^{\mathrm{ref}}(W_S, \GL(V))$ is open but not always closed in $\Hom^{\mathrm{fd}}(W_S, \GL(V))$: see Remark~\ref{rem:JWST}.

\subsection{Remarks for nonirreducible Coxeter groups}

Here are partial generalizations of Remarks~\ref{rem:neg-type-interpret}.\eqref{item:neg-type-interpret} and~\ref{rem:rho-v-alpha-refl-group}.

\begin{remark} \label{rem:neg-type-interpret-non-irred}
Let $W_S = W_{S_1} \times \cdots \times W_{S_r}$ be an infinite Coxeter group which is the product of $r \geq 1$ irreducible factors.
Let $\rho: W_S \to \GL(V)$ be a representation generated by weakly compatible reflections, defined by $\alpha = (\alpha_1,\dots,\alpha_N) \in {V^*}^N$ and $v = (v_1,\dots,v_N) \in\nolinebreak V^N$.
Suppose that for each irreducible factor $W_{S_{\ell}}$ of $W_S$, the Cartan submatrix $\Cart_{S_{\ell}} = (\alpha_i(v_j))_{s_i,s_j\in S_{\ell}}$ is of negative type, and consider as in Remark~\ref{rem:neg-type-interpret}.\eqref{item:neg-type-interpret} the vector $x_{\ell} = \sum_{s_j \in S_{\ell}} t_j v_j$ where $(t_j)_{s_j\in S_{\ell}} \in (\RR_{>0})^{S_{\ell}}$ is the Perron--Frobenius vector of $2\,\mathrm{Id}-\Cart_{S_{\ell}}$.
Then $x = x_1+\cdots+ x_r \in \mathrm{Int} (\widetilde{\Delta})$, and so $\rho$ is a representation as a reflection group.
\end{remark}

\begin{remark} \label{rem:rho-v-alpha-refl-group-non-irred}
Let $\rho : W_S\to\GL(V)$ be a representation of the (possibly nonirreducible) Coxeter group $W_S$ as a reflection group and let $\rho_v$, $\rho^{\alpha}$, and $\rho_v^{\alpha}$ be the induced representations of Remark~\ref{rem:notation_rhos}.
If the Cartan submatrix $\Cart_{S_{\ell}}$ is of negative type for each irreducible factor $W_{S_{\ell}}$ of $W_S$, then $\rho_v$, $\rho^{\alpha}$, and $\rho_v^{\alpha}$ are representations of $W_S$ as a reflection group, by Remark~\ref{rem:neg-type-interpret-non-irred}; in particular, these representations are injective.
\end{remark}

\subsection{Affine groups with Cartan matrices of negative type} \label{subsec:AffineGroups}

Let $W_S$ be an irreducible Coxeter group of type $\widetilde{A}_{N-1}$ with $N\geq 2$ (see Appendix~\ref{appendix:classi_diagram}).
We now give more details about representations of $W_S$ as a reflection group with a Cartan matrix of negative type as in Fact~\ref{fact:types}.\eqref{item:affine_case_negative_type}.
The case $N=2$ is described in Example~\ref{ex:N=2}, so we assume $N\geq 3$.
The following is stated in \cite[Lem.\,8]{mv00} or \cite[Thm.\,2.18]{mar17}; we give a short proof for the reader's convenience.

\begin{lemma} \label{lem:out_of_vinberg}
Let $W_S$ be a Coxeter group of type~$\widetilde{A}_{N-1}$ with $N\geq 3$ and $\rho: W_S \to \GL(V)$ a representation of $W_S$ as a reflection group in~$V$, associated to some $\alpha = (\alpha_1,\dots,\alpha_N) \in {V^*}^N$ and $v = (v_1,\dots,v_N) \in V^N$.
Suppose the Cartan matrix $\Cart = (\alpha_i(v_j))_{1\leq i,j\leq N}$ is of negative type and $\rho(W_S)$ is reduced and dual-reduced (Definition~\ref{def:V-v-alpha}).
Then
\begin{enumerate}
  \item\label{item:Atilde-1} $\dim(V)=N$;
  \item\label{item:Atilde-2} $\rho(W_S)$ divides (\ie acts properly discontinuously and cocompactly on) an open simplex of $\PP(V)$, namely $\OhmVin$;
  \item\label{item:Atilde-3} $\OhmVin$ is the unique nonempty $\rho(W_S)$-invariant convex open subset of $\PP(V)$.
\end{enumerate}
\end{lemma}

\begin{proof}
\eqref{item:Atilde-1}
By Fact~\ref{fact:types}, we have $\det\Cart\neq 0$, hence the vectors $v_1,\dots,v_N\in V$ are linearly independent.
Since $\rho(W_S)$ is dual-reduced, these vectors span~$V$, hence $\dim V=N$.

\eqref{item:Atilde-2}
Without loss of generality, we may assume that $\Cart = \Cart'$ as in \eqref{eqn:Cart-affine-neg-type} and, since $\det\Cart\neq 0$, that in some appropriate basis of $V\simeq\RR^N$ we have
$$ \left ( \begin{array}{ccc}
&\alpha_1 &\\ \hline &\vdots& \\ \hline &\alpha_N& \end{array} \right )=
\begin{pmatrix} 
1 		& -1 		& 	 	&  	 \\
		& \ddots 	& \ddots	& 	 \\
		&		&\ddots	& -1	 \\
-a^{-1}	&		&		& 1	
\end{pmatrix}
~\text{ and }~
\left ( \!\!\!\begin{array}{c|c|c} && \\ v_1 & \cdots & v_N \\ && \end{array} \!\!\!\right )=
\begin{pmatrix} 
1 	&  		& 		&  -a	 \\
-1	& \ddots 	&		& 	 \\
	&\ddots	&\ddots	& 	 \\
	&		& -1		& 1	
\end{pmatrix}.
$$
The formula $\rho(s_i) = (x \mapsto x - \alpha_i(x) v_i)$ yields
$$\{\rho(s_i)\}_{1\leq i \leq N} = \left \{ 
\begin{pmatrix} 0&1&&&\\1&0&&0&\\&&1&&\\&0&&\ddots&\\&&&&1 \end{pmatrix}, \dots,
\begin{pmatrix} 1&&&&\\&\ddots&&0&\\&&1&&\\&0&&0&1\\&&&1&0 \end{pmatrix}, 
\begin{pmatrix} 0&&&&a\\&1&&0&\\&&\ddots&&\\&0&&1&\\a^{-1}&&&&0 \end{pmatrix} \right \}.$$
The first $N-1$ of these elements generate a copy of the symmetric group $\mathfrak{S}_N$.
We have
\begin{equation} \label{eqn:zigzag}
\rho\big((s_1 s_2 \dotsm s_{N-1}) (s_{N-2} \dotsm s_2 s_1) s_N\big) = \mathrm{Diag}(a^{-1},1,\dots,1,a);
\end{equation}
from this it is easy to see that $\rho(W_S) \simeq \widetilde{A}_{N-1}$ is isomorphic to $\mathfrak{S}_N\ltimes \mathbb{Z}^{N-1}$.
Indeed, the normal subgroup $\mathbb{Z}^{N-1}$ consists of all matrices of the form $\mathrm{Diag}(a^{\nu_1}, \dots, a^{\nu_N})$ with $\nu_i\in \mathbb{Z}$ and $\nu_1+\dots+\nu_N=0$; we get a system of generators for this $\mathbb{Z}^{N-1}$ by conjugating~\eqref{eqn:zigzag} under $\big \langle \rho(s_1), \dots, \rho(s_{N-1}) \big \rangle \simeq \mathfrak{S}_N$.

If $a<1$ (\resp $a>1$), then the convex cone $\widetilde{\Delta}=\bigcap_{i=1}^N \{\alpha_i \leq 0 \}$ is the $\RR_{\geq 0}$-span (\resp $\RR_{\leq 0}$-span) of the column vectors $(a,\dots,a,1,\dots,1)$.
In either case, $\Delta$ is a closed (compact) projective simplex contained in the projectivized positive orthant $\mathbb{P} \left( \RR_{>0}^N \right)$.
The Tits--Vinberg domain $\OhmVin$, which is by definition (Section~\ref{subsec:Vinberg-constr}) the interior of the union of the $\rho(W_S)$-translates of~$\Delta$, is equal to $\mathbb{P} \left( \RR_{>0}^N \right)$.
Thus $\rho(W_S)$ divides $\OhmVin$.

\eqref{item:Atilde-3}
The properly convex open cones of $\RR^N$ preserved by the $\ZZ^{N-1}$ factor of $\rho(W_S) \simeq \mathfrak{S}_N\ltimes \mathbb{Z}^{N-1}$ (acting as diagonal matrices) are precisely the orthants. 
The convex (not necessarily properly convex) cones preserved by $\ZZ^{N-1}$ are the Cartesian products of $N$ factors each equal to $\RR_{>0}$, $\RR_{<0}$, or $\RR$.
The only nontrivial such product which is $\mathfrak{S}_N$-invariant after projectivization is $ \mathbb{P} \left( \RR_{>0}^N \right)$.
\end{proof}

\subsection{Strong irreducibility} \label{subsec:strong-irred}

The following proposition will be used several times below.
For reflection groups (Definition~\ref{def:refl-group}), the first two items were proved in \cite[Prop.\,19]{vin71} and \cite[Cor.\,to\,Prop.\,19]{vin71}, and the third one  in \cite[Th.\,2.18]{mar17}.

\begin{proposition} \label{prop:rho-irred}
Let $W_S$ be a Coxeter group in $N$ generators as in~\eqref{eqn:WS}, and assume that it is irreducible.
Let $\rho : W_S\to\GL(V)$ be a representation of $W_S$ generated by weakly compatible reflections (Definition~\ref{def:Cartan-matrix}), associated to some $\alpha = (\alpha_1,\dots,\alpha_N) \in {V^*}^N$ and $v = (v_1,\dots,v_N) \in V^N$.
Let $V_v$ and~$V_\alpha$ be as in Definition~\ref{def:V-v-alpha}.
Then
\begin{enumerate}
  \item\label{item:inv-subspace} a linear subspace $V'$ of~$V$ is $\rho(W_S)$-invariant if and only if $V_v\subset V'$ or $V'\subset V_\alpha$;
  \item\label{item:irred-reduced} the representation $\rho$ is irreducible if and only if $V_v = V$ and~$V_\alpha=\{0\}$, \ie if and only if $\rho(W_S)$ is reduced and dual-reduced (Definition~\ref{def:V-v-alpha});
  \item\label{item:irred-large} if $\rho$ is irreducible and $W_S$ is large, then $\rho$ is actually strongly irreducible.
\end{enumerate}
\end{proposition}

By \emph{strongly irreducible} we mean that the restriction of~$\rho$ to any finite-index subgroup of~$W_S$ is irreducible; equivalently, the group $\rho(W_S)$ does not preserve a finite union of nontrivial subspaces of~$V$.

\begin{proof}
\eqref{item:inv-subspace} If $V_v\subset V'$, then $V'$ is invariant under each $\rho(s_i) = (x\mapsto x-\alpha_i(x)v_i)$, hence under $\rho(W_S)$; if $V'\subset V_{\alpha}$, then $V'$ is pointwise fixed by each $\rho(s_i)$, hence by $\rho(W_S)$.
Conversely, suppose $V'$ is $\rho(W_S)$-invariant.
Let $S'$ be the set of generators $s_j$ of~$W_S$ such that $v_j\in V'$, and $S''$ its complement in~$S$.
For any $s_i\in S''$ and $v'\in V'$ we have $\rho(s_i)(v') = v' - \alpha_i(v') v_i \in V'$, hence $\alpha_i(v')=0$.
In particular, for any $s_i\in S''$ and $s_j\in S'$ we have $\alpha_i(v_j)=0$, hence $m_{i,j}=2$ by weak compatibility of~$\Cart$.
Since $W_S$ is irreducible, we have $S'=S$, in which case $V'\supset V_v$, or $S''=S$, in which case $V'\subset V_\alpha$.

\eqref{item:irred-reduced} Note that $V_v$ and $V_\alpha$ are invariant subspaces of~$V$ with $V_v\neq\{0\}$ and $V_{\alpha}\neq V$.
Therefore, if $\rho$ is irreducible, then $V_v = V$ and $V_\alpha = \{0\}$.
Conversely, if $V_v = V$ and $V_\alpha = \{0\}$, then \eqref{item:inv-subspace} implies that any invariant subspace is trivial, hence $\rho$ is irreducible.

\eqref{item:irred-large} Suppose by contradiction that $W_S$ is large and $\rho$ is irreducible but not strongly irreducible: this means that there is a finite collection $F$ of nontrivial subspaces of~$V$ which is preserved by the action of~$\rho$.
We may assume that $F$ is closed under intersection (excluding the trivial subspace).

We first claim that $F$ does not contain any one-dimensional subspace~$U$.
Indeed, if $U \in F$ is one-dimensional, then for any nonzero $u \in U$, the set $\rho(W_S)\cdot u$ must span $V$ since $\rho$ is irreducible, hence this set contains a basis $\mathcal{B}$ of~$V$.
Consider the set of elements of $\rho(W_S)$ whose action preserves each individual subspace of~$F$; it is a finite-index subgroup $H$ of $\rho(W_S)$.
The basis $\mathcal{B}$ is a simultaneous eigenbasis for all elements of~$H$, hence $H$ is abelian, and so $W_S$ is virtually abelian.
This contradicts the assumption that $W_S$ is a large irreducible Coxeter group.

Next, let $U$ be a subspace in $F$ of minimal dimension.
We have $\dim U\geq 2$ by the previous claim.
For each $i$, the subspace $\rho(s_i)\cdot U$ also belongs to~$F$.
Since $\rho(s_i)$ is a reflection, if we had $\rho(s_i)\cdot U \neq U$, then $\rho(s_i)\cdot U \cap U$ would be a subspace in $F$ of dimension $\dim U -1$, nontrivial since $\dim U\geq 2$, contradicting the minimality of~$U$.
Thus $\rho(s_i)\cdot U = U$ for all~$i$, and so $\rho(W_S)\cdot U = U$, contradicting the irreducibility of~$\rho$.
\end{proof}

Here is an analogue of Proposition~\ref{prop:rho-irred}.\eqref{item:inv-subspace} for Coxeter groups that are not necessarily irreducible.

\begin{proposition} \label{prop:rho-irred-non-irred-version}
Let $W_S$ be a Coxeter group in $N$ generators as in \eqref{eqn:WS}, and let $\rho : W_S\to\GL(V)$ be a representation of $W_S$ generated by weakly compatible reflections (Definition~\ref{def:Cartan-matrix}), associated to some $\alpha = (\alpha_1,\dots,\alpha_N) \in {V^*}^N$ and $v = (v_1,\dots,v_N) \in V^N$.
For any $\rho(W_S)$-invariant subspace $V'$ of~$V$, there exists a decomposition $S = T \sqcup U$ such that $W_S = W_T \times W_{U}$ and
$$\mathrm{span} \{ v_i \,|\, s_i \in T\} =: (V_v)_T  \subset V' \subset (V_\alpha)_U := \bigcap_{s_j \in U} \mathrm{Ker}(\alpha_j) .$$
\end{proposition}

\begin{proof}
Write $W_S$ as a product of irreducible factors: $W_S = W_{S_1} \times \cdots \times W_{S_r}$ where $S = S_1 \sqcup \cdots \sqcup S_r$.
By Proposition~\ref{prop:rho-irred}.\eqref{item:inv-subspace}, for each $1\leq\ell\leq r$, either $V' \supset (V_v)_{S_{\ell}}$ or $V' \subset (V_\alpha)_{S_{\ell}}$, and both cannot happen at once since $v_i \in (V_v)_{S_\ell}\smallsetminus (V_{\alpha})_{S_\ell}$ for all $s_i\in S_\ell$.
We can take for $T$ the union of those $S_{\ell}$ such that $V' \supset (V_v)_{S_{\ell}}$ and for $U$ the union of those $S_{\ell}$ such that $V' \subset (V_\alpha)_{S_{\ell}}$.
\end{proof}

The following consequence of Proposition~\ref{prop:rho-irred-non-irred-version} will be used in Sections~\ref{subsec:Vinberg-max}, \ref{subsec:rho-red}, \ref{subsec:large-enough}, and~\ref{subsec:naive-cc->not-IC}.

\begin{corollary} \label{cor:C-meets-P(Vv)}
Let $W_S$ be a (not necessarily irreducible) Coxeter group in $N$ generators, and $\rho: W_S \to \GL(V)$ a representation of $W_S$ as a reflection group  in~$V$, associated to some $\alpha = (\alpha_1,\dots,\alpha_N) \in {V^*}^N$ and $v = (v_1,\dots,v_N) \in V^N$.
Suppose that for each irreducible component $S_{\ell}$ of~$S$, the corresponding Cartan submatrix $\Cart_{S_{\ell}} = (\alpha_i(v_j))_{s_i,s_j\in S_{\ell}}$ is of negative type.
Then
\begin{enumerate}
  \item\label{item:prox-lim-set-nonempty} the proximal limit set $\Lambda_{\rho(W_S)}$ of $\rho(W_S)$ is nonempty and contained in $\PP(V_v)$;
  \item\label{item:C-meets-P(Vv)} for any nonempty $\rho(W_S)$-invariant closed properly convex subset $\C$ of $\OhmVin$, the set $\C \cap \PP(V_v)$ has nonempty interior in $\PP(V_v)$.
\end{enumerate}
\end{corollary}

\begin{proof}
We first assume that $W_S$ is irreducible.

\eqref{item:prox-lim-set-nonempty} 
Since the group $\rho(W_S)$ is generated by hyperplane reflections whose $(-1)$-eigen\-spaces lie in~$V_v$, we have $\rho(\gamma)\cdot x - x \in V_v$ for all $x\in V$ and $\gamma\in W_S$, and so the proximal limit set $\Lambda_{\rho(W_S)}$ of $\rho(W_S)$ is contained in $\PP(V_v)\subset\PP(V)$.
Let $\rho_v^{\alpha} : W_S\to\GL(V_v^{\alpha})$ be the representation induced by~$\rho$, as in Remark~\ref{rem:notation_rhos}.
It is easy to check that for any $\gamma \in W_S$, the element $\rho(\gamma)\in\GL(V)$ is proximal in $\PP(V)$ if and only if the element $\rho_v^\alpha(\gamma)\in\GL(V_v^{\alpha})$ is proximal in $\PP(V_v^\alpha)$ (see \eg \eqref{eqn:ninezeros} below).
Thus we only need to check that $\rho_v^\alpha(W_S)$ contains a proximal element in $\PP(V_v^\alpha)$.
This is the case by Proposition~\ref{prop:benoist-irred-Gamma}, given that $\rho_v^\alpha(W_S)$ acts irreducibly on $\PP(V_v^\alpha)$ by Proposition~\ref{prop:rho-irred}.\eqref{item:irred-reduced}.

\eqref{item:C-meets-P(Vv)} Let $\C$ be a nonempty $\rho(W_S)$-invariant closed properly convex subset of $\OhmVin$ and let $\overline{\C}$ be its closure in $\PP(V)$.
We claim that
\begin{equation} \label{eqn:Lambda-subset-C-bar-V-v}
\Lambda_{\rho(W_S)} \subset \overline{\C}\cap\PP(V_v).
\end{equation}
Indeed, if $\OhmVin \subset \PP(V)$ is properly convex, then \eqref{eqn:Lambda-subset-C-bar-V-v} follows from Proposition~\ref{prop:max-inv-conv}.\eqref{item:Lambda-prox-in-boundary} and \eqref{item:prox-lim-set-nonempty} above.
In general, $\OhmVin$ might not be properly convex, but we note that the relative interior of~$\C$ is a properly convex open subset of $\PP(V_{\C})$, where $V_{\C}$ is the span of $\C$ in~$V$.
Since $\C \subset \OhmVin$ and $\OhmVin \cap \PP(V_\alpha) = \emptyset$, the $\rho(W_S)$-invariant subspace $V_{\C}$ is not contained in~$V_\alpha$, hence it contains~$V_v$ by Proposition~\ref{prop:rho-irred}.\eqref{item:inv-subspace}.
By \eqref{item:prox-lim-set-nonempty} above, $\rho(W_S)$ contains elements that are proximal in $\PP(V_{\C})$, and $\Lambda_{\rho(W_S)} \subset \PP(V_v)$; applying Proposition~\ref{prop:max-inv-conv}.\eqref{item:Lambda-prox-in-boundary} to the relative interior of~$\C$, we obtain \eqref{eqn:Lambda-subset-C-bar-V-v}.
The inclusion \eqref{eqn:Lambda-subset-C-bar-V-v} implies that the convex hull of $\Lambda_{\rho(W_S)}$ in~$\OhmVin$ is contained in $\C\cap\PP(V_v)$.
This convex hull has nonempty interior in $\PP(V_v)$ because $\OhmVin \cap \PP(V_v) \neq \emptyset$ (Remark~\ref{rem:neg-type-interpret}.\eqref{item:neg-type-interpret}) and the projective span of $\Lambda_{\rho(W_S)}$ is the whole of $\PP(V_v)$, by Proposition~\ref{prop:rho-irred}.\eqref{item:inv-subspace}.

We now assume that $W_S = W_{S_1} \times \dots \times W_{S_r}$ is a product of $r$ irreducible factors.

\eqref{item:prox-lim-set-nonempty} The proximal limit set $\Lambda_{\rho(W_S)}$ of $\rho(W_S)$ is contained in $\PP(V_v)\subset\PP(V)$ by the same argument as in the irreducible case.
It is nonempty because it contains the proximal limit set $\Lambda_{\rho(W_{S_{\ell}})}$ of $\rho(W_{S_{\ell}})$ for each $1\leq\ell\leq r$, and the latter is nonempty by the irreducible case treated above.

\eqref{item:C-meets-P(Vv)} For any $1\leq\ell\leq r$, the restriction of $\rho$ to $W_{S_{\ell}}$ is a representation of $W_{S_{\ell}}$ as a reflection group.
Let $\Omega_{\scriptscriptstyle\mathrm{TV},{\ell}} \subset \PP(V)$ be the corresponding Tits--Vinberg domain, and $\Delta_{\ell}^\flat \subset \Omega_{\scriptscriptstyle\mathrm{TV},\ell}$ the corresponding fundamental polytope as in Section~\ref{subsec:Vinberg-constr}.
Let $\C$ be a nonempty $\rho(W_S)$-invariant closed properly convex subset of $\OhmVin$ and let $\overline{\C}$ be its closure in $\PP(V)$.
The irreducible case proved above implies that $\overline{\C} \cap \PP((V_v)_{S_{\ell}})$ has nonempty interior in $\PP((V_v)_{S_{\ell}})$ for all $1\leq\ell\leq r$.
In particular, the interior $U_{\ell}$ of $\overline{\C} \cap \PP((V_v)_{S_\ell}) \cap \Delta_\ell^\flat$ in $\PP((V_v)_{S_\ell})$ is nonempty.
Lift $U_1, \ldots, U_r$ to convex cones $\widetilde U_1, \ldots, \widetilde U_r$ of~$V$ contained in~$\widetilde{\Omega}_{\scriptscriptstyle\mathrm{TV}}$.
The sum $\widetilde U_1 + \cdots + \widetilde U_r$ is open in $V_v$.
It projects down to a convex open subset of $\PP(V_v)$ which is contained in $\Delta^\flat \cap \overline{\C}$, and in fact contained in~$\C$ since $\C$ is closed in $\OhmVin$. 
\end{proof}

\subsection{Block triangular decomposition}

We now assume that the Coxeter group $W_S$ is irreducible. Let $\rho : W_S \to \GL(V)$ be a representation of $W_S$ generated by weakly compatible reflections (Definition~\ref{def:Cartan-matrix}), associated to some $\alpha \in {V^*}^N$ and $v \in V^N$, with Cartan matrix~$\Cart$.
Choose a complementary subspace $U$ of $V_{\alpha} \cap V_v$ in $V_\alpha$, a complementary subspace $U'$ of $V_{\alpha} \cap V_v$ in $V_v$, and a complementary subspace $U''$ of $V_{\alpha}+V_v$ in~$V$:
$$
V= \lefteqn{\underbrace{\phantom{U \oplus (V_{\alpha} \cap V_v)}}_{V_\alpha}} U \oplus \overbrace{(V_{\alpha} \cap V_v) \oplus U'}^{V_v} \oplus U''.
$$
By the definition \eqref{eqn:rho-s-i} of the $\rho(s_i)$, in a basis adapted to this decomposition of $V$,
the elements of $\rho(W_S)$ are matrices of the form
\begin{equation} \label{eqn:ninezeros}
\rho(\gamma) =
\begin{tikzpicture}[baseline={([yshift=-.5ex]current bounding box.center)}]
\matrix [matrix of math nodes,left delimiter=(,right delimiter=)] (m)
        {
 \mathrm{Id} & 0 & 0 & 0 \\
 0 & \mathrm{Id} & * & * \\
 0 & 0 & \rho_v^\alpha(\gamma)  & *\\
  0 & 0 & 0 & \mathrm{Id} \\
};  
    \draw[color=blue] (m-3-2.south west) rectangle (m-2-4.north west);
    \draw[color=blue,implies-](m-2-4.north west) -- +(1.2,0.3)  node [pos=1.4] {\hspace{3cm}$\rho_v(\gamma)\:\:$ (dual-reduced)};
    \draw[color=red] (m-4-2.south east) rectangle (m-3-4.north east);
   \draw[color=red,implies-](m-3-4.east) -- +(1.0,0.2)  node [pos=1.5] {\hspace{2cm}$\rho^\alpha(\gamma)\:\:$ (reduced)};
    \end{tikzpicture}
\end{equation}
where $\rho_v,\rho^\alpha,\rho_v^\alpha$ are the induced representations from Remark~\ref{rem:notation_rhos}, and we implicitly identify $V^{\alpha}=V/V_\alpha$ with $U' \oplus U''$ and $V_v^{\alpha}=V_v/(V_\alpha\cap V_v)$ with~$U'$.
(The zeros in~\eqref{eqn:ninezeros} come from the definition of $V_\alpha$ in columns 1 and~2, and of $V_v$ in rows 1 and~4.)
Further, $\dim (U') = \mathrm{rank}(\Cart$) by \cite[Prop.\,15]{vin71} and $\rho_v^\alpha$ is irreducible by Proposition~\ref{prop:rho-irred}.\eqref{item:irred-reduced}.

\begin{lemma}\label{lemma:semisimple}
The representation $\rho$ is semisimple if and only if $V=V_{\alpha}\oplus V_v$, \ie\linebreak $(U,V_{\alpha}\cap V_v,U',U'') = (V_{\alpha},\{0\},V_v,\{0\})$, in which case $\rho$ has the form $\begin{pmatrix} \mathrm{Id} & 0\\ 0 & \rho_v^{\alpha}\end{pmatrix}$.
\end{lemma}

\begin{proof}
If $\rho$ is semisimple, then $V_{\alpha}$ (\resp $V_v$) admits a $\rho(W_S)$-invariant complementary subspace in~$V$, which must contain~$V_v$ (\resp be contained in $V_{\alpha}$) by Proposition~\ref{prop:rho-irred}.\eqref{item:inv-subspace}; we deduce $V=V_{\alpha}\oplus V_v$.
Conversely, if $V=V_{\alpha}\oplus V_v$, then $\rho$ has the given matrix form, hence $\rho$ is semisimple since $\rho_v^{\alpha}$ is irreducible.
\end{proof}

\begin{remark} \label{rem:JWST}
If $W_S$ is an affine irreducible Coxeter group in $N\geq 2$ generators as in \eqref{eqn:WS}, then the analogue of Remark~\ref{rem:open-closed-in-Hom-df}.\eqref{item:open-closed-in-Hom-df} fails: the set $\Hom^{\mathrm{ref}}(W_S, \GL(V))$ is not closed in $\Hom^{\mathrm{fd}}(W_S, \GL(V))$ when $\dim(V)>N$.
Indeed, consider the Tits geometric representation $\rho : W_S\to\GL(\RR^N)$ of Remark~\ref{rem:Tits-geom}, with Cartan matrix $\Cart = \mathrm{Cos}(W_S) = (-2\cos(\pi/m_{i,j}))_{1\leq i,j\leq N}$ (with the convention $\pi/\infty=0$), and associated $(\alpha_1,\dots,\alpha_N)$ (which is the canonical basis of $(\RR^N)^*$) and $(v_1,\dots,v_N)$ (such that $v_i \in \mathbb R^N$ is the $i$-th column of $\Cart$ in the canonical basis of~$\RR^N$).
By construction, $\rho$ is reduced (Definition~\ref{def:V-v-alpha}) and symmetrizable, and if $N=2$ then one readily sees that $\Cart$ is of zero type.
By Fact~\ref{fact:types}.\eqref{item:affine_case}, the $\rho(W_S)$-invariant subspace $V' := \mathrm{span}(v_1,\dots,v_N)$ of~$\RR^N$ is a hyperplane, and $\rho(W_S)$ acts properly discontinuously and cocompactly on the affine chart $\PP(\RR^N) \smallsetminus \PP(V')$, preserving some Euclidean metric (hence $\rho(W_S)$ is a crystallographic group).
Note that the restriction $\rho'$ of $\rho$ to $V'$ is not injective.
In fact, it has finite image since $\rho'$ is the linear part of the action of $\rho(W_S)$ on the affine chart $\PP(\RR^N) \smallsetminus \PP(V')$.
For $1\leq i\leq N$, let $\alpha'_i \in {V'}^*$ be the restriction of $\alpha_i$ to~$V'$.
For any $c,d\in\RR$, we define a representation $\rho^{c,d} : W_S \rightarrow \mathrm{GL}\underbrace{(\RR\oplus V' \oplus \RR)}_{=:V}$ by the data: 
\vspace{-3mm}
$$ \alpha_i^{c,d}:=(0,\alpha'_i,-d)\text{ and }v_i^{c,d}:=\begin{pmatrix}-c \\ v_i \\ 0 \end{pmatrix}, 
	~\text{ so that }
	\rho^{c,d}(s_i)= 
	\begin{tikzpicture}[baseline={([yshift=-0.5ex]current bounding box.center)}]
		\matrix [matrix of math nodes,left delimiter=(,right delimiter=)] (m)
		{1 & c \alpha'_i & -cd  \\  0 & \mathrm{Id}-v_i \alpha'_i & d v_i  \\ 0 & 0 & 1 \\ };  
		\draw[color=red] (m-3-1.south east) rectangle (m-2-3.north east);
		\draw[color=blue] (m-1-1.north west) rectangle (m-2-2.south east);
	\end{tikzpicture}$$
for all $1\leq i\leq N$.
For $d\neq 0$, the representation in the bottom (red) box identifies with $(\rho^{c,d})^\alpha$, and is conjugate to~$\rho$.
For $c \neq 0$, the representation in the upper (blue) box identifies with $(\rho^{c,d})_v$ and is conjugate to the dual representation $\rho^*$. 
Define $V_\alpha^{c,d} := \bigcap_{1\leq i\leq N} \mathrm{Ker}(\alpha_i^{c,d})$ and $V_v^{c,d} := \mathrm{span}_{1\leq i\leq N}(v_i^{c,d})$. Then by Fact~\ref{fact:types}.\eqref{item:affine_case_zero_type}, 
$$\begin{array}{lcccl} 
		\rho^{c,d} \in \Hom^{\mathrm{ref}}(W_S, \GL(V)) & \Longleftrightarrow & V_\alpha^{c,d}+V_v^{c,d} \neq V & \Longleftrightarrow & d\neq 0.
		\end{array} $$
Applying Fact~\ref{fact:types}.\eqref{item:affine_case_zero_type} to the dual representation $\left(\rho^{c,d}\right)^*$, we have that
$$\begin{array}{lcccl} 
		\left(\rho^{c,d}\right)^* \in \Hom^{\mathrm{ref}}(W_S, \GL(V)) &  \Longleftrightarrow & 
		V_\alpha^{c,d}\cap V_v^{c,d} \neq \{0\} & \Longleftrightarrow & c \neq 0.
\end{array} $$
Observing that $\rho^{0,0}$ is not injective, we have that 
$$\begin{array}{lcc} 
		\rho^{c,d} \in \Hom^{\mathrm{fd}}(W_S, \GL(V)) & \Longleftrightarrow & (c,d)\neq (0,0).
\end{array} $$
Thus, for $d\neq 0$ we have $\rho^{1,d} \in \Hom^{\mathrm{ref}}(W_S, \GL(V))$, and as $d\to 0$ the limit lies in $\Hom^{\mathrm{fd}}(W_S, \GL(V)) \smallsetminus \Hom^{\mathrm{ref}}(W_S, \GL(V))$, which proves the result.
	\\ \indent
When $W_S$ is of type~$\widetilde{A}_1$, the representation $\rho^{1,1}$ corresponds to Example~\ref{ex:N=2}.(iii); in the affine chart of Figure~\ref{fig:Z2xZ2}.(iii), the affine action of $\rho^{1,d}(W_S)$ is a scaling, by a factor~$d$, of the affine action of $\rho^{1,1}(W_S)$.
	\\ \indent
In general, $\rho^{1,1}(W_S)$ can also be realized by letting $W_S$ act properly discontinuously and cocompactly by reflections on a horosphere of the hyperbolic space~$\mathbb{H}^N$, and lifting to reflections of the Minkowski space~$\RR^{N,1}$.
\end{remark}

\subsection{Parametrizing characters by equivalence classes of Cartan matrices} \label{sec:decomp_rep}

Suppose our Coxeter group $W_S$ in $N$ generators is infinite and irreducible.
Let $\Cart$ be an $N \times N$ real matrix which is weakly compatible with~$W_S$ (Definition~\ref{def:weak-compat}).

If $\mathrm{rank}(\Cart) = \dim(V)$, then we may directly construct $\alpha = (\alpha_1,\dots,\alpha_N) \in {V^*}^N$ and $v = (v_1,\dots,v_N) \in V^N$ with $ (\alpha_i(v_j))_{i,j=1}^N = \Cart$ as follows.
Choose $\mathrm{rank}(\Cart)$ linearly independent rows of $\Cart$ giving a basis of the row space of $\Cart$; the subset of $\alpha$ corresponding to those rows must form a basis $\mathcal B^*$ of~$V^*$.
Each remaining row is a linear combination of the chosen row basis; the corresponding element of $\alpha$ must be the corresponding linear combination of $\mathcal B^*$.
This uniquely determines $\alpha$, up to the choice of the basis $\mathcal B^*$.
The vectors of~$v$ are then determined by their coordinates with respect to $\mathcal B^*$ which are the entries of $\Cart$.
Hence, up to the action of $\GL(V)$, the Cartan matrix $\Cart$ uniquely determines $\alpha$ and~$v$.
The corresponding representation $\rho: W_S \to \mathrm{GL}(V)$ generated by weakly compatible reflections is uniquely determined, up to conjugacy, by the equivalence class of~$\Cart$; it is irreducible by Proposition~\ref{prop:rho-irred}.\eqref{item:irred-reduced}.

If $\mathrm{rank}(\Cart) < \dim(V)$, then splitting $V$ into a trivial summand plus a subspace of dimension $\mathrm{rank}(\Cart)$ and applying the same process yields a unique conjugacy class of semisimple representations. 
Hence we get the following fact.

\begin{fact} \label{fact:semisimplification}
For any $N \times N$ real matrix $\Cart$ of rank $\leq\dim(V)$ which is weakly compatible with~$W_S$ (Definition~\ref{def:weak-compat}), there is a unique conjugacy class of semisimple representations $\rho : W_S\to\GL(V)$ with Cartan matrix~$\Cart$.
\end{fact}

Further, the constructive correspondence described above implies the following para\-metrization statement, which will be used in Section~\ref{sec:deformation}.

\begin{fact} \label{fact:parameterization}
The map assigning to a conjugacy class of semisimple representations the equivalence class of its Cartan matrices is a homeomorphism from the open and closed subset of $\chi(W_S, \GL(V))$ consisting of representations generated by weakly compatible reflections to the space of equivalence classes of $N \times N$ matrices of rank $\leq\dim(V)$ that are weakly compatible with~$W_S$. 
\end{fact}

Note that if $\mathrm{rank}(\Cart) < \dim(V)$, then there are also non-semisimple conjugacy classes associated to $\Cart$.
Invariants for these are more subtle to describe, see \cite[Prop.\,15]{vin71}.

\subsection{Proof of Fact~\ref{fact:types}} \label{subsec:proof-fact-types}

We first recall a useful fact.

\begin{fact}[{\cite[Prop.\,21--22--23]{vin71}}] \label{fact:Cosine-matrix}
Let $W_S$ be an irreducible Coxeter group with presentation \eqref{eqn:WS}.
Then the $N\times N$ real matrix $\mathrm{Cos}(W_S) := (-2\cos(\pi/m_{i,j}))_{1\leq i,j\leq N}$ (with the convention $\pi/\infty=0$) is a Cartan matrix compatible with~$W_S$ (Definition~\ref{def:compatible}), and $W_S$ is spherical (\resp affine, \resp large) if and only if $\mathrm{Cos}(W_S)$ is of positive (\resp zero, \resp negative) type.
\end{fact}

Recall the Tits geometric representation of Remark~\ref{rem:Tits-geom} has Cartan matrix $\mathrm{Cos}(W_S)$.

\begin{proof}[Proof of Fact~\ref{fact:types}]
\eqref{item:spherical_case} By \cite[Prop.\,22]{vin71}, the group $W_S$ is spherical if and only if $\Cart$ is of positive type; in this case $\Cart$ is symmetrizable by \cite[Lem.\,13]{vin71}.

\eqref{item:affine_case_zero_type} Suppose $W_S$ is affine and $\Cart$ is of zero type.
Then $\Cart$ is symmetrizable by \cite[Lem.\,13]{vin71}.
Since condition \eqref{eqn:not_empt_int} holds for $\widetilde{\Delta}$ in~$V$, it also holds for its image in $V^{\alpha} = V/V_{\alpha}$, and so $\rho^{\alpha}$ is still a representation of $W_S$ as a reflection group with Cartan matrix $\Cart$, by Remark~\ref{rem:notation_rhos}.
The representation $\rho^{\alpha}$ is now reduced.
By \cite[Prop.\,23]{vin71}, there is a hyperplane $V'$ of~$V^{\alpha}$ such that the group $\rho^{\alpha}(W_S)$ acts properly discontinuously and cocompactly on the affine chart $\PP(V^{\alpha})\smallsetminus\PP(V')$ of $\PP(V^{\alpha})$, preserving some Euclidean metric.
In particular, $\rho^{\alpha}$ is non-semisimple.
This implies that $\rho$ is non-semisimple (see \eqref{eqn:ninezeros} and Lemma~\ref{lemma:semisimple}).
Note that the hyperplane $V'$ of~$V^{\alpha}$ must be $\rho^{\alpha}(W_S)$-invariant; therefore it contains $V_v^{\alpha}$ by Proposition~\ref{prop:rho-irred}.\eqref{item:inv-subspace}.
We must in fact have $V'=V_v^{\alpha}$ because the dimension of~$V_v^{\alpha}$ is at least the rank of the matrix $\Cart$, which is $(\dim V^{\alpha})-1$ by \cite[Prop.\,23]{vin71}.

\eqref{item:affine_case_negative_type} Suppose $W_S$ is affine and $\Cart$ is of negative type.
The case $N=2$ is treated in Example~\ref{ex:N=2bis}.
We now assume $N \geq 3$.
If $W_S$ were not of type~$\widetilde{A}_{N-1}$, then by \cite[Prop.\,21.(2)]{vin71} the Cartan matrix $\Cart$ would be equivalent to $\mathrm{Cos}(W_S)$, hence would be of zero type by Fact~\ref{fact:Cosine-matrix}: impossible.
Therefore $W_S$ is of type~$\widetilde{A}_{N-1}$, \ie $m_{i,i+1}=3$ for all $1\leq i\leq N-1$ and $m_{1,N}=3$ and $m_{i,j}=0$ for all other pairs $(i,j)$ with $i\neq j$.
Conjugating by appropriate positive diagonal matrices, we see that there exists $a>0$ such that the compatible matrix $\Cart$ belongs to the same equivalence class (Definition~\ref{def:equivalent}) as the matrices  
\begin{equation} \label{eqn:Cart-affine-neg-type}
\Cart' = \begin{pmatrix} 
2     & -1    &       &     & -a    \\
 -1     & 2     &\ddots    &  ~0  &     \\
    &\ddots &\ddots &\ddots &   \\
    & 0~  &\ddots &\ddots &-1   \\
-a^{-1} &   &   & -1  & 2   
 \end{pmatrix}
 \quad \textrm{and} \quad
\Cart'' = \begin{pmatrix} 
2     & -b^{-1}    &       &     & -b    \\
 -b     & 2     &\ddots    &  ~0  &     \\
    &\ddots &\ddots &\ddots &   \\
    & 0~  &\ddots &\ddots &-b^{-1}   \\
-b^{-1} &   &   & -b  & 2   
 \end{pmatrix},
\end{equation}
where $b := \sqrt[N]{a}$ (see also \cite[Prop.\,16]{vin71}).
We have $\det\Cart = \det\Cart' = \det\Cart'' = 2-a-a^{-1}$.
Note that if $\tau \in \RR^N$ has all its entries $1$, then $\Cart'' \tau$ has all its entries $2-b-b^{-1}$.
Since $\Cart''$ is of negative type, we must have $2-b-b^{-1}\neq 0$ by Fact~\ref{fact:type-sign-entries}.\eqref{item:zero-type-conclusion}, hence $b\neq 1$, hence $\det\Cart \neq 0$.

\eqref{item:large_case} We argue by contraposition.
Suppose that $\Cart$ is not of negative type.
By \cite[Prop.\,21.(1)]{vin71}, the Cartan matrix $\Cart$ is equivalent to $\mathrm{Cos}(W_S)$.
Fact~\ref{fact:Cosine-matrix} then implies that $W_S$ is spherical or affine, \ie not large.
\end{proof}

\section{Maximality of the Tits--Vinberg domain} \label{sec:Omega-Vin-max}

Recall the notions of reduced and dual-reduced representations from Definition~\ref{def:V-v-alpha}.
The goal of this section is to establish the following maximality properties of the Tits--Vinberg domain for irreducible Coxeter groups, as well as some analogues for nonirreducible Coxeter groups (Propositions \ref{prop:maximal-non-irred} and~\ref{prop:nonirred_ohmweird}).

\begin{proposition} \label{prop:maximal}
Let $W_S$ be an irreducible Coxeter group in $N\geq 2$ generators and $\rho: W_S \to \GL(V)$ a representation of $W_S$ as a reflection group in~$V$ (Definition~\ref{def:refl-group}) with a Cartan matrix $\Cart=(\alpha_i(v_j))_{1\leq i,j\leq N}$ of negative type (Definition~\ref{def:type}).
Then
\begin{enumerate}
  \item\label{item:max-1} if $\rho(W_S)$ is reduced, then $\OhmVin$ is properly convex and it is maximal with respect to inclusion among all $\rho(W_S)$-invariant convex open subsets of $\PP(V)$;
   \item\label{item:max-2} if $\rho(W_S)$ is reduced and dual-reduced and $N\geq 3$, then $\OhmVin$ is the unique maximal nonempty $\rho(W_S)$-invariant convex open subset of $\PP(V)$;
  \item\label{item:max-3} in general, if $N\geq 3$, then any $\rho(W_S)$-invariant properly convex open subset of $\PP(V)$ is contained in the Tits--Vinberg domain $\OhmVin$.
\end{enumerate}
\end{proposition}

\begin{remark} \label{rem:2_gen}
In the setting of Proposition~\ref{prop:maximal}, if $N=2$ and $W_S$ is infinite (\ie $W_S$ is of type $\widetilde{A}_1$ as in Example~\ref{ex:N=2}.(i)), and if $\rho(W_S)$ is reduced and dual-reduced, then there are exactly two nonempty $\rho(W_S)$-invariant properly convex open subsets of $\PP(V)$, namely $\OhmVin$ (which contains $\PP(\mathrm{Ker}(\alpha_1))$ and $\PP(\mathrm{Ker}(\alpha_2)$)) and $\PP(V)\smallsetminus\overline{\OhmVin}$ (which contains $[v_1]$ and~$[v_2]$).
\end{remark}

\subsection{The reduced case} \label{subsec:Vinberg-max-reduced}

We first prove Proposition~\ref{prop:maximal}.\eqref{item:max-1}--\eqref{item:max-2}.
If $\rho$ is defined by $\alpha = (\alpha_1,\dots,\alpha_N) \in {V^*}^N$ and $v = (v_1,\dots,v_N) \in V^N$, we denote by $V_v$ the span of $v_1,\dots,v_N$ and by $V_{\alpha}$ the intersection of the kernels of $\alpha_1,\dots,\alpha_N$, as in Definition~\ref{def:V-v-alpha}.

\begin{proof}[Proof of Proposition~\ref{prop:maximal}.\eqref{item:max-1}]
Let $\Omega$ be a $\rho(W_S)$-invariant convex open subset of $\PP(V)$ containing~$\OhmVin$.

We first show that if $\rho(W_S)$ is reduced (\ie $V_\alpha=\{0\}$), then $\Omega$ is properly convex.
We will actually prove the contrapositive.
Assume $\Omega$ is not properly convex. 
The largest linear subspace $V'$ of~$V$ contained in the closure of a convex lift of $\Omega$ to $V$ defines a nontrivial invariant projective subspace $\PP(V')$ which is disjoint from $\Omega$ (and from any affine chart containing~$\Omega$).
Proposition~\ref{prop:rho-irred}.\eqref{item:inv-subspace} implies that either $V_\alpha \supset V'$ or $V_v \subset V'$.
If $V_v \subset V'$, then $\PP(V_v) \subset \PP(V) \smallsetminus \Omega \subset \PP(V) \smallsetminus \OhmVin$, which contradicts the fact $\PP(V_v) \cap \OhmVin \neq \emptyset$ (Remark~\ref{rem:neg-type-interpret}.\eqref{item:neg-type-interpret}).
So $V_\alpha \supset V' \neq \{0\}$, showing that $\rho(W_S)$ is not reduced. 

\begin{figure}
\definecolor{qqwuqq}{rgb}{0,0.39,0}
\definecolor{uququq}{rgb}{0.25,0.25,0.25}
\begin{tikzpicture}[scale=1,line cap=round,line join=round,>=triangle 45,x=1.0cm,y=1.0cm]
\clip(-4.35,-3.97) rectangle (4.62,3.83);
\draw [rotate around={0:(0,0)}] (0,0) ellipse (4cm and 3.46cm);
\draw [shift={(-0.07,-9.6)},dash pattern=on 1pt off 1pt]  plot[domain=1.27:1.86,variable=\t]({1*11.86*cos(\t r)+0*11.86*sin(\t r)},{0*11.86*cos(\t r)+1*11.86*sin(\t r)});
\draw [shift={(-0.01,9.59)},dash pattern=on 1pt off 1pt]  plot[domain=4.42:5.01,variable=\t]({1*11.89*cos(\t r)+0*11.89*sin(\t r)},{0*11.89*cos(\t r)+1*11.89*sin(\t r)});
\draw (4,0)-- (-1.41,1.63);
\draw (-1.41,1.63)-- (-0.54,-0.35);
\draw (-0.54,-0.35)-- (4,0);
\draw (0,0)-- (0,3.25);
\draw [dash pattern=on 1pt off 1pt on 2pt off 4pt] (0,1.21)-- (-1.08,0.88);
\draw [dash pattern=on 1pt off 1pt on 2pt off 4pt] (-1.08,0.88)-- (0.95,0.92);
\draw [dash pattern=on 1pt off 1pt on 2pt off 4pt] (0.95,0.92)-- (1.62,-0.19);
\draw [dash pattern=on 1pt off 1pt on 2pt off 4pt] (1.62,-0.19)-- (2.19,0.55);
\draw [dash pattern=on 1pt off 1pt on 2pt off 4pt] (2.19,0.55)-- (2.86,-0.09);
\draw [dash pattern=on 1pt off 1pt on 2pt off 4pt] (2.86,-0.09)-- (2.99,0.3);
\draw [dash pattern=on 1pt off 1pt on 2pt off 4pt] (3.37,-0.05)-- (2.99,0.3);
\draw [dash pattern=on 1pt off 1pt on 2pt off 4pt] (3.37,-0.05)-- (3.56,0.13);
\draw [dash pattern=on 1pt off 1pt on 2pt off 4pt] (3.56,0.13)-- (3.72,-0.02);
\draw (1.7,2.77) node[anchor=north west] {$\Omega$};
\draw (1.7,1.94) node[anchor=north west] {$\OhmVin$};
\draw (0.2,0.85) node[anchor=north west] {$\Delta^{\flat}$};
\draw [dash pattern=on 1pt off 1pt on 2pt off 4pt] (3.72,-0.02)-- (3.8,0.06);
\draw [dash pattern=on 1pt off 1pt on 2pt off 4pt] (3.8,0.06)-- (3.86,-0.01);
\draw [dash pattern=on 1pt off 1pt on 2pt off 4pt] (3.86,-0.01)-- (3.91,0.03);
\draw [dash pattern=on 1pt off 1pt on 2pt off 4pt] (3.91,0.03)-- (3.93,-0.01);
\draw [dash pattern=on 1pt off 1pt on 2pt off 4pt] (3.93,-0.01)-- (3.97,0.01);
\draw [dash pattern=on 1pt off 1pt on 2pt off 4pt] (3.97,0.01)-- (3.98,0);
\draw (0.02,0.2) node[anchor=north west] {$z$};
\draw (4.09,0.3) node[anchor=north west] {$y''$};
\draw (0.1,2.7) node[anchor=north west] {$y'$};
\draw (0.07,3.4) node[anchor=north west] {$y$};
\begin{scriptsize}
\fill [color=black] (0,3.25) circle (1.5pt);
\draw [color=uququq] (0,2.27)-- ++(-2.0pt,-2.0pt) -- ++(4.0pt,4.0pt) ++(-4.0pt,0) -- ++(4.0pt,-4.0pt);
\fill [color=uququq] (0,0) circle (1.5pt);
\draw [color=uququq] (4,0)-- ++(-2.0pt,-2.0pt) -- ++(4.0pt,4.0pt) ++(-4.0pt,0) -- ++(4.0pt,-4.0pt);
\end{scriptsize}
\end{tikzpicture}
\caption{Illustration for the proof of Proposition~\ref{prop:maximal}.\eqref{item:max-1}}
\label{fig:omega_max_is_omega_vin}
\end{figure}

Assume then that $\rho(W_S)$ is reduced, so that $\Omega$ is properly convex. We now check that $\Omega = \OhmVin$.
Suppose by contradiction that there is a point $y \in \Omega \smallsetminus \OhmVin$; Figure~\ref{fig:omega_max_is_omega_vin} illustrates the proof.
Let $z$ lie in the interior of the fundamental domain $\Delta^{\flat} = \Delta \cap \OhmVin$ for the action of $\rho(W_S)$ on $\OhmVin$.
Recall (Fact~\ref{fact:Delta-flat}) that $\Delta^{\flat}$ is equal to $\Delta$ minus the union of all faces with infinite stabilizer and the quotient $\rho(W_S) \backslash \OhmVin$ is an orbifold homeomorphic to $\Delta^{\flat}$, with mirrors on the reflection walls.
The intersection of the segment $[z,y]$ with $\OhmVin$ is a half-open segment $[z,y')$ where $y' \in \partial \OhmVin \cap \Omega$.
It has finite length in the Hilbert metric~$d_{\Omega}$ (see Section~\ref{subsec:prop-conv-proj}).
The image of $[z,y')$ in the quotient may be viewed as a billiard path $q: [z,y') \to \Delta^{\flat}$ (where the reflection associated to a mirror wall is used to determine how the trajectory reflects off that wall). 
Any point of $\Delta^{\flat}$ is the center of a $d_{\Omega}$-ball contained in $\OhmVin$, hence disjoint from $\rho(W_S)\cdot y'$; it follows that $q(t)$ can only accumulate on points of $\Delta\smallsetminus \Delta^{\flat}$ as $t\rightarrow y'$.
Any such accumulation point $y''$ belongs to a face of $\Delta$ with infinite stabilizer, hence $y''\in \partial \Omega$ since $\rho(W_S)$ acts properly discontinuously on $\Omega$.
It follows that $d_{\Omega}(z,q(t))$ goes to infinity as $t \to y'$.
By the triangle inequality, we have $d_{\Omega}(z,y') > d_{\Omega}(z,q(t))$, contradicting the fact that $d_{\Omega}(z,y')$ is finite.
\end{proof}

\begin{proof}[Proof of Proposition~\ref{prop:maximal}.\eqref{item:max-2}]
Suppose $\rho(W_S)$ is reduced and dual-reduced and $N\geq 3$.
By Proposition~\ref{prop:rho-irred}.\eqref{item:irred-reduced}, the representation $\rho$ is irreducible, hence any $\rho(W_S)$-invariant convex open subset $\Omega$ of $\PP(V)$ has to be properly convex.
Indeed, the largest linear subspace of~$V$ contained in the closure of a convex lift of $\Omega$ to $V$ is $\rho(W_S)$-invariant.

Suppose $W_S$ is affine.
Since $\rho$ is irreducible (hence semisimple), Fact~\ref{fact:types} implies that $W_S$ is of type $\widetilde{A}_{N-1}$ with $N=\dim(V)$.
Since $N\geq 3$, Lemma~\ref{lem:out_of_vinberg} applies: the set $\OhmVin$ is the unique $\rho(W_S)$-invariant convex open subset of $\PP(V)$.

Suppose $W_S$ is large.
By Proposition~\ref{prop:rho-irred}.\eqref{item:irred-large}, the representation $\rho$ is strongly irreducible.
By Proposition~\ref{prop:benoist-irred-Gamma}, there is a unique maximal $\rho(W_S)$-invariant properly convex open set $\Omega_{\max}$ containing all other invariant properly convex open sets; this must be the Tits--Vinberg domain $\Omega_{\max} = \OhmVin$ by Proposition~\ref{prop:maximal}.\eqref{item:max-1}.
\end{proof}

\begin{remark} \label{rem:Vinberg-not-reduced}
Suppose we are in the setting of Proposition~\ref{prop:maximal}, namely $W_S$ is irreducible and the Cartan matrix $\Cart$ is of negative type.
\begin{enumerate}
  \item\label{item:OhmVin-prop-conv-or-not} If $\rho(W_S)$ is reduced, then $\OhmVin$ is properly convex by Proposition~\ref{prop:maximal}.\eqref{item:max-1}.
  On the other hand, if $\rho(W_S)$ is not reduced, then $\OhmVin$ is not properly convex, for it contains the projective subspace $\PP(V_{\alpha})$ in its boundary.
  \item\label{item:inv-conv-not-in-OhmVin} If $\rho(W_S)$ is not dual-reduced, then there is a $\rho(W_S)$-invariant convex open subset of $\PP(V)$ which is not contained in $\OhmVin$, namely the complement $\mathcal{U}$ in $\PP(V)$ of any projective hyperplane containing $\PP(V_v)$.
  Indeed, $\mathcal{U}$ is $\rho(W_S)$-invariant by Proposition~\ref{prop:rho-irred}.\eqref{item:inv-subspace}, and $\mathcal{U}\not\subset\OhmVin$ because $\mathcal{U}$ is an affine chart, $\OhmVin$ is convex, and $\OhmVin \cap \PP(V_v) \neq \emptyset$ by Remark~\ref{rem:neg-type-interpret}.
  See Figure~\ref{fig:reduced}.
\end{enumerate}
\end{remark}

\begin{figure}[h]
\centering
\labellist
\small\hair 2pt
\pinlabel {$\Delta$} [u] at 105 115
\pinlabel {$\Delta$} [u] at 105 50
\pinlabel {$\PP(V_v)$} [u] at 230 178
\pinlabel {${}^{\Delta \cap \PP(V_v)}$} [u] at 105 180
\endlabellist
\includegraphics[scale=.85]{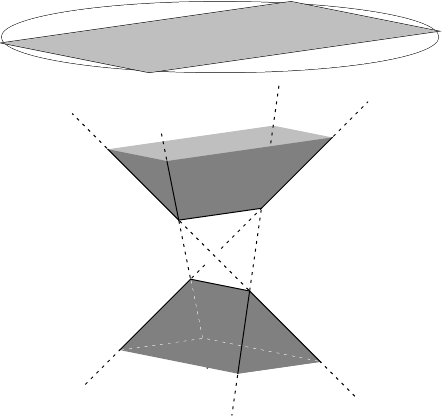}
\caption{Illustration of Remark~\ref{rem:Vinberg-not-reduced}.\eqref{item:inv-conv-not-in-OhmVin} for $N=4$, where $V_v$ is a hyperplane in $V=\RR^4$ and the affine reflections preserving the chart $\PP(V)\smallsetminus \PP(V_v) \simeq \RR^3$ are chosen with linear parts in $\OO(2,1)$. As above, $\Delta$ is the fundamental polytope for~$\rho$, and $\OhmVin$ is the interior of $\bigcup_{\gamma\in W_S} \rho(\gamma)\cdot \Delta$.
Here $\OhmVin$ intersects the chart in two connected components, each of which is a \emph{domain of dependence} as in \cite{mes90} (see also \cite[Ex.\,11.13]{dgk-proj-cc}).}
\label{fig:reduced}
\end{figure}

\subsection{The general case in Proposition~\ref{prop:maximal}} \label{subsec:Vinberg-max}

For any representation $\rho' : W_S\to\GL(V')$ of our Coxeter group $W_S$, we write $\Lambda_{\rho'}$ for the proximal limit set $\Lambda_{\rho'(W_S)}$ of $\rho'(W_S)$ in $\PP(V')$, as in Definition~\ref{def:prox-lim-set}, and $\Lambda_{\rho'}^*:=\Lambda_{{\rho'}^*(W_S)}$ for the proximal limit set of ${\rho'}^*(W_S)$ in~$\PP({V'}^*)$, where ${\rho'}^* : W_S\to\GL({V'}^*)$ is the dual representation.

\begin{proof}[Proof of Proposition~\ref{prop:maximal}.\eqref{item:max-3}]
By Corollary~\ref{cor:C-meets-P(Vv)}.\eqref{item:prox-lim-set-nonempty}, there exists $\gamma\in W_S$ such that $\rho(\gamma)$ is proximal in $\PP(V)$; then $\rho^{*}(\gamma^{-1})$ is proximal in $\PP(V^*)$, and so $\Lambda_{\rho}^*\neq\emptyset$.
By Proposition~\ref{prop:max-inv-conv}.\eqref{item:Omega-max}, it is sufficient to prove the following two claims:
\begin{enumerate}[label=(\roman*)]
  \item\label{item:Vinberg-conn-comp} the Tits--Vinberg domain $\OhmVin$ is a full connected component of
$$\mathcal{U} := \PP(V) \smallsetminus \bigcup_{[\varphi]\in\Lambda_{\rho}^*} \PP(\mathrm{Ker}(\varphi)) ;$$
  \item\label{item:only-one-conn-comp} the set $\mathcal{U}$ admits only one $\rho(W_S)$-invariant connected component.
\end{enumerate}

We first prove \ref{item:Vinberg-conn-comp}.
Consider $V^{\alpha} = V/V_{\alpha}$ and the representation $\rho^{\alpha} : W_S\to\GL(V^{\alpha})$ induced by~$\rho$, as in Remark~\ref{rem:notation_rhos}: since $\Cart$ is assumed of negative type, $\rho^{\alpha}$ is a representation of $W_S$ as a reflection group (Remark~\ref{rem:rho-v-alpha-refl-group}).
By Propositions~\ref{prop:max-inv-conv}.\eqref{item:Omega-max} and~\ref{prop:maximal}.\eqref{item:max-1}, the corresponding Tits--Vinberg domain $\OhmVin^\alpha \subset \PP(V^\alpha)$ is a full connected component of
\begin{equation} \label{eqn:A'}
\mathcal{U}^\alpha := \PP(V^\alpha) \smallsetminus \bigcup_{[\varphi]\in \Lambda_{\rho^{\alpha}}^*} \PP(\mathrm{Ker}(\varphi)).
\end{equation}
Let $\pi : \PP(V) \smallsetminus \PP(V_{\alpha}) \to \PP(V^\alpha)$ be the natural projection.
On the one hand, the fundamental reflection polytope $\Delta^\alpha \subset \PP(V^\alpha)$ for $\rho^{\alpha}(W_S)$ satisfies $\pi^{-1}(\Delta^\alpha) = \Delta \smallsetminus \PP(V_\alpha)$, hence $\OhmVin = \pi^{-1}(\OhmVin^\alpha)$.
On the other hand, the connected components of~$\mathcal{U}$ are exactly the preimages under~$\pi$ of the connected components of~$\mathcal{U}^\alpha$.
Indeed, the dual $(V^{\alpha})^*$ of~$V^\alpha$ identifies with the annihilator $(V^*)_{\alpha} \subset V^*$ of $V_{\alpha}$, \ie the set of linear forms $\varphi\in V^*$ that vanish on~$V_{\alpha}$, which is also the span of the $\alpha_i$ in~$V^*$.
The dual action of $\rho(W_S)$ on $V^*$ exchanges the roles of the $v_i$ and the~$\alpha_i$: namely, $\rho(s_i)$ acts on~$V^*$ as a reflection in the hyperplane defined by~$v_i$, with $(-1)$-eigenvector~$\alpha_i$.
Therefore the proximal limit set $\Lambda_{\rho}^*$ of $\rho^*(W_S)$ in $\PP(V^*)$ is contained in the subspace $\PP((V^*)_{\alpha}) = \PP((V^{\alpha})^*)$, and equal to the proximal limit set $\Lambda^*_{\rho^\alpha}$ of $(\rho^{\alpha})^*(W_S)$ in that subspace.
This implies that the connected components of~$\mathcal{U}$ are exactly the preimages under~$\pi$ of the connected components of~$\mathcal{U}^\alpha$, and completes the proof of~\ref{item:Vinberg-conn-comp}.

We now prove \ref{item:only-one-conn-comp}, assuming $N\geq 3$.
By the proof of \ref{item:Vinberg-conn-comp}, it is sufficient to check that the set $\mathcal{U}^\alpha$ of \eqref{eqn:A'} has a unique $\rho(W_S)$-invariant connected component, for the projection $\pi : \PP(V) \smallsetminus\nolinebreak \PP(V_{\alpha}) \to \PP(V^\alpha)$ is $\rho(W_S)$-equivariant.
Assume by contradiction that there exists a $\rho(W_S)$-invariant connected component $\Omega'$ of $\mathcal{U}^\alpha$ different from the Tits--Vinberg domain $\OhmVin^\alpha$.
Intersecting $\OhmVin^\alpha$ with $\PP(V_v^\alpha)$ yields the Tits--Vinberg domain for $\rho_v^{\alpha}(W_S)$ in $\PP(V_v^\alpha)$.
Since $\rho_v^\alpha(W_S)$ is both reduced and dual reduced, $\OhmVin^\alpha \cap \PP(V_v^\alpha)$ is the unique maximal nonempty $\rho_v^\alpha(W_S)$-invariant convex open subset of $\PP(V_v^\alpha)$ by Proposition~\ref{prop:maximal}.\eqref{item:max-2}.
Now, it may be that $\Omega'$ does not intersect $\PP(V_v^\alpha)$.
However, $\overline{\Omega'}$ contains the proximal limit set $\Lambda_{\rho^\alpha}$, hence $\overline{\Omega'} \cap \PP(V_v^\alpha)$ is a closed $\rho_v^{\alpha}(W_S)$-invariant convex subset of $\PP(V_v^\alpha)$ whose interior, necessarily nonempty by Corollary~\ref{cor:C-meets-P(Vv)}.\eqref{item:C-meets-P(Vv)}, is contained in $\OhmVin^\alpha \cap \PP(V_v^\alpha)$.
Since $\OhmVin^\alpha$ is open, we find that $\OhmVin^\alpha$ and $\Omega'$ must overlap, hence they are equal.
\end{proof}

\subsection{Invariant cones for irreducible~$W_S$}

Whereas Proposition~\ref{prop:maximal} assumed a Cartan matrix of negative type and concerned invariant properly convex sets in the projective space $\PP(V)$, the following does not make any assumptions about the Cartan matrix, and concerns invariant properly convex cones in the vector space~$V$.

\begin{lemma}\label{lem:maximal-W-product}
Let $W_S$ be an irreducible Coxeter group and $\rho: W_S \to \GL(V)$ a representation of $W_S$ as a reflection group in~$V$ (Definition~\ref{def:refl-group}), associated to some $\alpha = (\alpha_1,\dots,\alpha_N) \in {V^*}^N$ and $v = (v_1,\dots,v_N) \in V^N$.
Then any $\rho(W_S)$-invariant properly convex open cone $\widetilde{\Omega}$ of~$V$ is contained in the Tits--Vinberg cone $\widetilde{\Omega}_{\scriptscriptstyle\mathrm{TV}}$ or its opposite $-\widetilde{\Omega}_{\scriptscriptstyle\mathrm{TV}}$.
\end{lemma}

(Recall from Section~\ref{subsec:Vinberg-constr} that $\widetilde{\Omega}_{\scriptscriptstyle\mathrm{TV}}$ and $-\widetilde{\Omega}_{\scriptscriptstyle\mathrm{TV}}$ are both $\rho(W_S)$-invariant convex open cones of~$V$, on which $\rho(W_S)$ acts properly discontinuously.)

Note that when $N=2$ and $W_S$ is infinite, there may exist a $\rho(W_S)$-invariant properly convex open subset of $\PP(V)$ which is not contained in~$\OhmVin$ (see Remark~\ref{rem:2_gen}); however, its preimage in~$V$ is the union of two properly convex cones of~$V$ which are \emph{not} $\rho(W_S)$-invariant (the two cones are switched by $\rho(W_S)$).

\begin{proof}
If $W_S$ is finite, then $\widetilde{\Omega}_{\scriptscriptstyle\mathrm{TV}} = V$ by \cite[Lem.\,10]{vin71}, and so the lemma is obviously true.
We now assume that $W_S$ is infinite.
Then the Cartan matrix $\Cart=(\alpha_i(v_j))_{1\leq i,j\leq N}$ is of negative or zero type (Fact~\ref{fact:types}).

If $N=2$, then we are in the setting of Example~\ref{ex:N=2} and $\widetilde \Omega$ is contained in $\widetilde{\Omega}_{\scriptscriptstyle\mathrm{TV}}$ or its opposite $-\widetilde{\Omega}_{\scriptscriptstyle\mathrm{TV}}$.
So we now assume $N \geq 3$.

Suppose $\Cart$ is of negative type.
The projection $\Omega \subset \PP(V)$ of $\widetilde{\Omega} \subset V$ is a $\rho(W_S)$-invariant properly convex open subset of $\PP(V)$.
By Proposition~\ref{prop:maximal}.\eqref{item:max-3} we have $\Omega \subset \OhmVin$, and so $\widetilde{\Omega}$ is contained in one of the two $\rho(W_S)$-invariant cones $\widetilde{\Omega}_{\scriptscriptstyle\mathrm{TV}}$ or $-\widetilde{\Omega}_{\scriptscriptstyle\mathrm{TV}}$.

Suppose $\Cart$ is of zero type.
First, observe that $\widetilde{\Omega} \cap V_\alpha =\emptyset$, since the action of $\rho(W_S)$ on $\widetilde{\Omega}$ is properly discontinuous and the action on $V_\alpha$ is trivial.
Hence, $\widetilde{\Omega}$ projects to a (not necessarily properly) convex open cone ${\widetilde{\Omega}}^{\alpha}$ in $V^\alpha = V/V_\alpha$, which does not contain~$0$.
The induced representation $\rho^{\alpha} : W_S\to V^{\alpha}$ is still a representation as a reflection group (Fact~\ref{fact:types}.\eqref{item:affine_case_zero_type}), with a corresponding Tits--Vinberg cone $\widetilde{\Omega}_{\scriptscriptstyle\mathrm{TV}}^{\alpha} \subset V^{\alpha}$.
Note that $\Delta + V_{\alpha} = \Delta$; therefore $\widetilde{\Omega}_{\scriptscriptstyle\mathrm{TV}}$ is the full preimage of $\widetilde{\Omega}_{\scriptscriptstyle\mathrm{TV}}^{\alpha}$ under the natural projection $V\to V^{\alpha}$.
Therefore it is sufficient to check that ${\widetilde{\Omega}}^{\alpha}$ is equal to $\widetilde{\Omega}_{\scriptscriptstyle\mathrm{TV}}^{\alpha}$ or $-\widetilde{\Omega}_{\scriptscriptstyle\mathrm{TV}}^{\alpha}$.
By Fact~\ref{fact:types}.\eqref{item:affine_case_zero_type}, the set $V^\alpha_v$ is a $\rho^\alpha(W_S)$-invariant hyperplane in $V^\alpha$ and $\widetilde{\Omega}_{\scriptscriptstyle\mathrm{TV}}^{\alpha}$ is a connected component of $V^{\alpha} \smallsetminus V_v^{\alpha}$.
Any affine translate $w + V^\alpha_v$, for $w \in V^\alpha \smallsetminus V^\alpha_v$, is invariant under $\rho^\alpha(W_S)$, with a compact fundamental domain.
Since ${\widetilde{\Omega}}^{\alpha}$ is open, it contains some $w \in V^\alpha \smallsetminus V^\alpha_v$.
Since the action on $w + V_v^\alpha$ is cocompact, the $\rho^\alpha(W_S)$-invariant convex cone ${\widetilde{\Omega}}^{\alpha}$ contains all of $w + V^\alpha_v$, hence it contains the entire halfspace on the $w$ side of $V^\alpha_v$, which is $\widetilde{\Omega}_{\scriptscriptstyle\mathrm{TV}}^{\alpha}$ or $-\widetilde{\Omega}_{\scriptscriptstyle\mathrm{TV}}^{\alpha}$.
In fact ${\widetilde{\Omega}}^{\alpha}$ is equal to this halfspace, for if it contained a point $w \in V^\alpha \smallsetminus V^\alpha_v$ in the other halfspace, the same argument would show that it contains this full other halfspace, hence the whole of $V^\alpha$ by convexity: contradiction.
Thus ${\widetilde{\Omega}}^{\alpha}$ is equal to $\widetilde{\Omega}_{\scriptscriptstyle\mathrm{TV}}^{\alpha}$ or $-\widetilde{\Omega}_{\scriptscriptstyle\mathrm{TV}}^{\alpha}$, and so $\widetilde{\Omega}$ is contained in $\widetilde{\Omega}_{\scriptscriptstyle\mathrm{TV}}$ or $-\widetilde{\Omega}_{\scriptscriptstyle\mathrm{TV}}$.
\end{proof}

\subsection{Tits--Vinberg domains for nonirreducible Coxeter groups}

Let $W_S = W_{S_1} \times \cdots \times W_{S_r}$ be a Coxeter group which is the product of $r \geq 1$ irreducible factors.
Let $\alpha = (\alpha_1,\dots,\alpha_N) \in {V^*}^N$ and $v = (v_1,\dots,v_N) \in V^N$.
Then the matrix $\Cart = (\alpha_i(v_j))_{1\leq i,j\leq N}$ is compatible with $W_S$ if and only if the Cartan submatrix $\Cart_{S_{\ell}} := (\alpha_i(v_j))_{s_i,s_j\in S_{\ell}}$ is compatible with $W_{S_{\ell}}$ for all $1\leq\ell\leq r$.
The fundamental polyhedral cone $\widetilde \Delta$ for $W_S$, defined as the intersection of the nonpositive halfspaces for $\alpha_1, \ldots, \alpha_N$, is equal to the intersection $\widetilde \Delta = \bigcap_{\ell=1}^r \widetilde \Delta_{\ell}$, where for each $1 \leq \ell \leq r$, the set $\widetilde \Delta_{\ell}$ is the fundamental polyhedral cone for the factor $W_{S_{\ell}}$, defined as the intersection of the nonpositive halfspaces for $\{\alpha_i \mid s_i \in S_\ell \}$. 
Note that for each $\ell \neq \ell'$, the cone $\widetilde \Delta_{\ell}$ is preserved by $\rho(W_{S_{\ell'}})$.
This implies the following.

\begin{fact} \label{fact:Vinberg-cones}
Let $W_S = W_{S_1} \times \cdots \times W_{S_r}$ be a Coxeter group in $N$ generators which is the product of $r\geq 1$ irreducible factors, and let $\rho : W_S\to\GL(V)$ be a representation of $W_S$ as a reflection group, associated to some $\alpha = (\alpha_1,\dots,\alpha_N) \in {V^*}^N$ and $v = (v_1,\dots,v_N) \in\nolinebreak V^N$.
Then for each irreducible factor $W_{S_\ell}$, the restriction $\rho|_{W_{S_{\ell}}} : W_{S_{\ell}}\to\GL(V)$ is a representation of $W_{S_{\ell}}$ as a reflection group, associated to $(\alpha_i)_{s_i\in S_{\ell}}$ and $(v_i)_{s_i\in S_{\ell}}$.
The Tits--Vinberg cone $\widetilde{\Omega}_{\scriptscriptstyle\mathrm{TV}} = \mathrm{Int}\Big( \bigcup_{\gamma \in W_{S}} \rho(\gamma)\cdot\widetilde \Delta \Big)$ for $\rho(W_S)$ is the intersection of the Tits--Vinberg cones $\widetilde{\Omega}_{\scriptscriptstyle\mathrm{TV},\ell} = \mathrm{Int}\left( \bigcup_{\gamma \in W_{S_{\ell}}} \rho(\gamma)\cdot\widetilde \Delta_{\ell} \right)$ for the $\rho(W_{S_{\ell}})$:
$$\widetilde{\Omega}_{\scriptscriptstyle\mathrm{TV}} = \bigcap_{\ell=1}^r \widetilde{\Omega}_{\scriptscriptstyle\mathrm{TV},\ell}.$$
\end{fact}

Here is an analogue of Proposition~\ref{prop:maximal}.\eqref{item:max-1} when $W_S$ is not necessarily irreducible.

\begin{proposition} \label{prop:maximal-non-irred}
Let $W_S = W_{S_1} \times \cdots \times W_{S_r}$ be an infinite Coxeter group in $N$ generators which is the product of $r \geq 1$ irreducible factors, and let $\rho : W_S\to\GL(V)$ be a representation of~$W_S$ as a reflection group, associated to some $\alpha = (\alpha_1,\dots,\alpha_N) \in {V^*}^N$ and $v = (v_1,\dots,v_N) \in V^N$.
Suppose that the Cartan submatrix $\Cart_{S_{\ell}} = (\alpha_i(v_j))_{s_i,s_j\in S_{\ell}}$ is of negative type for each irreducible factor $W_{S_{\ell}}$.
If $\rho$ is reduced, then the Tits--Vinberg cone $\widetilde{\Omega}_{{\scriptscriptstyle\mathrm{TV}}}$, and hence the Tits--Vinberg domain $\OhmVin = \PP(\widetilde{\Omega}_{\scriptscriptstyle\mathrm{TV}})$, is properly convex.
\end{proposition}

\begin{proof}
By contraposition, suppose that $\widetilde{\Omega}_{\scriptscriptstyle\mathrm{TV}}$ is not properly convex.
By Fact~\ref{fact:Vinberg-cones}, it is contained in the Tits--Vinberg cone $\widetilde{\Omega}_{\scriptscriptstyle\mathrm{TV},\ell}$ for $\rho(W_{S_{\ell}})$ for all~$\ell$.
Therefore the maximal linear subspace $V'$ of~$V$ included in the closure of $\widetilde{\Omega}_{\scriptscriptstyle\mathrm{TV}}$ is contained in the maximal linear subspace $V'_{\ell}$ of~$V$ included in the closure of $\widetilde{\Omega}_{\scriptscriptstyle\mathrm{TV},\ell}$ for all~$\ell$.
The proof of Proposition~\ref{prop:maximal}.\eqref{item:max-1} implies that $V'_{\ell} \subset (V_{\alpha})_{S_{\ell}} := \bigcap_{s_j \in S_\ell} \mathrm{Ker}(\alpha_j)$ for all~$\ell$.
Hence $\{0\} \neq V' \subset \bigcap_{\ell=1}^r  (V_{\alpha})_{S_{\ell}} = V_\alpha$, and so $\rho$ is not reduced.
\end{proof}

As discussed in Section~\ref{subsec:refl-basics}, for any $1\leq i\leq N$, the pair $(\alpha_i,v_i)\in V^*\times V$ defining the reflection $\rho(s_i)$ in \eqref{eqn:rho-s-i} is not uniquely determined by $\rho(s_i)$.
Indeed, for any $\lambda_i \neq 0$ the pair $(\lambda_i \alpha_i, \lambda_i^{-1} v_i)$ yields the same reflection.
Changing $(\alpha_i,v_i)$ into $(\lambda_i \alpha_i, \lambda_i^{-1} v_i)$ changes the Cartan matrix $\Cart$ into its conjugate $D \Cart D^{-1}$ where $D = \mathrm{Diag}(\lambda_1, \ldots, \lambda_N)$ is a diagonal matrix with nonzero entries. 

If $W_S$ is irreducible, then it is easy to see that $D \Cart D^{-1}$ is still a compatible matrix if and only if $\lambda_1, \ldots, \lambda_N$ all have the same sign.
Assume that \eqref{eqn:not_empt_int} holds, \ie the fundamental polyhedral cone $\widetilde \Delta$ has nonempty interior. 
Then choosing $\lambda_1, \ldots, \lambda_N$  all positive does not change the definition of $\widetilde \Delta$, and the Tits--Vinberg cone $\widetilde{\Omega}_{\scriptscriptstyle\mathrm{TV}}$ remains unchanged.
If $\lambda_1, \ldots, \lambda_N$ are all negative, then $\widetilde \Delta$ becomes $- \widetilde \Delta$, and the Tits--Vinberg cone $\widetilde{\Omega}_{\scriptscriptstyle\mathrm{TV}}$ becomes $- \widetilde{\Omega}_{\scriptscriptstyle\mathrm{TV}}$.
In either case, the projections $\Delta$ and $\OhmVin$ to $\PP(V)$ are unchanged.
Hence it makes sense to refer to \emph{the} Tits--Vinberg domain $\OhmVin$ of $\rho(W_S)$ in $\PP(V)$.

Assume now that $W_S = W_{S_1}\times \dotsm \times W_{S_{r}}$ is the product of $r \geq 2$ irreducible factors. 
In order to preserve compatibility of the Cartan matrix, we must choose the signs of the $\lambda_i$ corresponding to any one of the factors $W_{S_{\ell}}$ to be consistent, but the choice of sign may be different for different factors.
A nontrivial such blockwise sign change will not change the representation $\rho$, but will nontrivially change the definition of $\widetilde \Delta$; we note that $\widetilde{\Delta}$ could actually become empty for some sign choices (see Remark~\ref{rem:neg-type-interpret-non-irred} for a case where all $2^r$ sign choices work).
The result is up to $2^r$ different fundamental polyhedral cones $\widetilde \Delta$, hence up to $2^r$ different Tits--Vinberg cones $\widetilde{\Omega}_{\scriptscriptstyle\mathrm{TV}}$.
Projecting to $\PP(V)$, this gives up to $2^{r-1}$ different fundamental polyhedra $\Delta$ and up to $2^{r-1}$ different Tits--Vinberg domains $\OhmVin$ for the representation $\rho$ of $W_S$ as a reflection group, depending on the choice of $\alpha = (\alpha_1,\dots,\alpha_N) \in {V^*}^N$ and $v = (v_1,\dots,v_N) \in V^N$ defining~$\rho$.

The following is an immediate consequence of Lemma~\ref{lem:maximal-W-product} and Fact~\ref{fact:Vinberg-cones}.

\begin{proposition} \label{prop:maximal-W-product}
Let $W_S = W_{S_1} \times \cdots \times W_{S_r}$ be a Coxeter group which is the product of $r\geq 1$ irreducible factors and let $\rho: W_S \to \GL(V)$ be a representation of $W_S$ as a reflection group.
Then any $\rho(W_S)$-invariant properly convex open cone $\widetilde \Omega$ of~$V$ is contained in one of the (at most) $2^r$ Tits--Vinberg cones corresponding to the (at most) $2^r$ possible choices of $\alpha = (\alpha_1,\dots,\alpha_N) \in {V^*}^N$ and $v = (v_1,\dots,v_N) \in V^N$ defining~$\rho$.
\end{proposition}

On the other hand, there can exist an invariant properly convex open subset of $\PP(V)$ that lifts to a convex cone of~$V$ which is not invariant, but the following proposition (analogous to Proposition~\ref{prop:maximal}.\eqref{item:max-3}) shows that this only happens in a degenerate situation.

\begin{proposition} \label{prop:nonirred_ohmweird}
Let $W_S = W_{S_1} \times \cdots \times W_{S_r}$ be a Coxeter group which is the product of $r\geq 1$ irreducible factors and let $\rho: W_S \to \GL(V)$ be a representation of $W_S$ as a reflection group.
Suppose there exists a $\rho(W_S)$-invariant properly convex open subset $\Omega$ of $\PP(V)$ which is not contained in one of the (at most) $2^{r-1}$ Tits--Vinberg domains corresponding to the (at most) $2^{r-1}$ possible choices of $\alpha = (\alpha_1,\dots,\alpha_N) \in {V^*}^N$ and $v = (v_1,\dots,v_N) \in V^N$ defining~$\rho$.
Then, up to renumbering, $W_{S_1}$ is of type $\widetilde{A}_1$, the Cartan submatrix $\Cart_{S_1} = (\alpha_i(v_j))_{s_i,s_j\in S_1}$ is of negative type, and $W_{S_{\ell}}$ is spherical for all $\ell\geq 2$.
\end{proposition}

\subsection{Proof of Proposition~\ref{prop:nonirred_ohmweird}}

We first establish the following.

\begin{lemma}\label{lem:nonirred_ohmweird}
Let $W_S$ be a Coxeter group and $\rho: W_S \to \GL(V)$ a representation of $W_S$ as a reflection group.
Let $\Omega$ be a nonempty $\rho(W_S)$-invariant convex open subset of $\PP(V)$.
Then
\begin{enumerate}
  \item\label{item:vi-or-keralphai} for every $s_i \in S$, either $\mathrm{Ker}(\alpha_i) \cap \Omega \neq \varnothing$ or $[v_i] \in \Omega$, and both cannot simultaneously happen;
  \item\label{item:nonirred_ohmweird_inside} if $\Omega$ is properly convex, then, setting $S' :=\{ s_i \in S \, |\, [v_i] \in \Omega \}$, we have
  \begin{enumerate}
    \item\label{item:nonirred_ohmweird_inside-1} $m_{i,j} = \infty$ and $\Cart_{i,j} \Cart_{j,i} > 4$ for all $s_i\neq s_j$ in~$S'$,
    \item\label{item:nonirred_ohmweird_inside-2} $m_{i,j} = 2$ for all $s_i \in S'$ and $s_j \in S\smallsetminus S'$.
\end{enumerate}
\end{enumerate}
\end{lemma}

\begin{proof}
\eqref{item:vi-or-keralphai} Fix $i \in \{1, \ldots, N\}$.
Consider any point $x \in \Omega$ with $x \neq [v_i]$.
The projective line $\ell$ through $x$ and $[v_i]$ is $\rho(s_i)$-invariant and $\Omega \cap \ell$ is a $\rho(s_i)$-invariant properly convex subset of~$\ell$.
The action of $\rho(s_i)$ fixes the points $\mathrm{Ker}(\alpha_i) \cap \ell$ and $[v_i]$ and switches the two intervals in between them.
It follows that $\Omega \cap \ell$ contains exactly one of $\mathrm{Ker}(\alpha_i) \cap \ell$ and~$[v_i]$.

\eqref{item:nonirred_ohmweird_inside} Suppose $\Omega$ is properly convex.
If $N=1$ (\ie $\# W_S=2$), then there is nothing to prove.
So we assume $N \geq 2$.
Consider $s_i \in S'$ and some other $s_j \in S$.
Let $W_{i,j}$ be the subgroup of $W_S$ generated by $s_i$ and $s_j$, and $V_{i,j}$ be the linear subspace of~$V$ spanned by $v_i$ and $v_j$.

We first observe that $\dim V_{i,j}=2$.
Indeed, suppose by contradiction that $[v_i] = [v_j]$.
Any projective line $\ell$ of $\PP(V)$ passing through $[v_i] = [v_j]$ is globally preserved by $\rho(W_{i,j})$.
The intersection $\ell \cap \Omega$ of such a line with $\Omega$ is a properly convex subset of $\ell$, \ie an interval, which is invariant under $\rho(W_{i,j})$.
The element $\rho(s_i s_j)$ fixes each endpoint of this interval, and it also fixes $[v_i]$ (which is not an endpoint since it belongs to~$\Omega$ by assumption); therefore $\rho(s_i s_j)$ fixes pointwise the whole line~$\ell$.
This holds for any projective line $\ell$ containing $[v_i]$, hence $\rho(s_i s_j)$ acts trivially on $\PP(V)$.
But $\rho$ is faithful and the subgroup generated by $s_i s_j$ has index two in $W_{i,j}$, hence $W_{i,j}$ has order at most two: contradiction.
Thus $[v_i] \neq [v_j]$ and $\dim V_{i,j} = 2$.

Consider the action of $\rho(W_{i,j})$ on the projective line $\PP(V_{i,j})$.
Again, the intersection $\Omega \cap \PP(V_{i,j})$ is a nontrivial open interval and both generators $\rho(s_i)$ and $\rho(s_j)$ act nontrivially on this interval, switching its endpoints.
These two endpoints are fixed by $\rho(s_i s_j)$, hence correspond to two eigenlines of $\rho(s_i s_j)$, associated to eigenvalues $\lambda$ and~$ \lambda^{-1}$ for some $\lambda \neq 0$.
If\ $\lambda^2 = 1$, then $\rho((s_i s_j)^2)$ acts trivially on $V_{i,j}$, hence $\alpha_j(v_i) = \alpha_i(v_j)= 0$ and $m_{i,j} = 2$; in this case $\mathrm{Ker}(\alpha_j) \cap \Omega \neq \varnothing$, and so $[v_j] \notin \Omega$ by~\eqref{item:vi-or-keralphai}, which means $s_j \notin S'$. 
If $\lambda^2 \neq 1$, then $m_{i,j} = \infty$; in this case $\Cart_{i,j} = \alpha_i(v_j) \, \alpha_j(v_i) > 4$ and $[v_j] \in \Omega \cap \PP(V_{i,j})$ by Remark~\ref{rem:2_gen}, hence $s_j \in S'$.
\end{proof}

\begin{proof}[Proof of Proposition~\ref{prop:nonirred_ohmweird}]
Let $\Omega$ be a $\rho(W_S)$-invariant properly convex open subset of $\PP(V)$. 
Let  $S' := \{ s_i \in S \, |\, [v_i] \in \Omega \}$.
We will use the following simple observation: let $\widetilde{\Omega}$ be a convex cone of~$V$ lifting $\Omega$.
Then any element of $\rho(W_S)$ either preserves $\widetilde{\Omega}$, or switches $\widetilde{\Omega}$ and $-\widetilde{\Omega}$.
More precisely, for a generator $s_i \in S$, the cone $\widetilde{\Omega}$ is preserved by $\rho(s_i)$ if and only if $s_i \notin S'$.
If $S' = \emptyset$, then $\widetilde{\Omega}$ is $\rho(W_S)$-invariant and so by Proposition~\ref{prop:maximal-W-product}, the set $\Omega$ is contained in one of the (at most) $2^{r-1}$ Tits--Vinberg domains corresponding to the (at most) $2^{r-1}$ possible choices of $\alpha = (\alpha_1,\dots,\alpha_N) \in {V^*}^N$ and $v = (v_1,\dots,v_N) \in V^N$ defining~$\rho$. 

Assume then that $\Omega$ is not contained in any of these (at most) $2^{r-1}$ Tits--Vinberg domains for $\rho(W_S)$, and hence that $S' \neq \emptyset$. 

By Lemma~\ref{lem:nonirred_ohmweird}.\eqref{item:nonirred_ohmweird_inside-1}, the Coxeter group $W_{S'}$ is irreducible.
In particular, the Tits--Vinberg domain $\Omega_{\scriptscriptstyle\mathrm{TV},S'}$ for $\rho(W_{S'})$ is uniquely defined, independently of the choice of $\alpha = (\alpha_1,\dots,\alpha_N) \in {V^*}^N$ and $v = (v_1,\dots,v_N) \in V^N$.
Suppose by contradiction that $\# S' \geqslant 3$.
Then the fact that $m_{i,j}=\infty$ for all $s_i\neq s_j$ in~$S'$ implies that $W_{S'}$ is large, hence the Cartan matrix $\Cart=(\alpha_i(v_j))_{1\leq i,j\leq N}$ is of negative type by Fact~\ref{fact:types} and $\Omega_{\scriptscriptstyle\mathrm{TV},S'}$ is convex.
By Proposition~\ref{prop:maximal}.\eqref{item:max-3} the $\rho(W_{S'})$-invariant properly convex open set $\Omega$ is contained in $\Omega_{\scriptscriptstyle\mathrm{TV},S'}$, hence $[v_i] \in \Omega_{\scriptscriptstyle\mathrm{TV},S'}$ for all $s_i\in S'$ by definition of~$S'$.
On the other hand, we have $\mathrm{Ker}(\alpha_i) \cap \Omega_{\scriptscriptstyle\mathrm{TV},S'} \neq \emptyset$ for all $s_i\in S'$ by definition of $\Omega_{\scriptscriptstyle\mathrm{TV},S'} = \mathrm{Int}(\rho(W_{S'})\cdot\Delta_{S'})$: contradiction with Lemma~\ref{lem:nonirred_ohmweird}.\eqref{item:vi-or-keralphai}.
Thus $\# S'\leq 2$.

By Lemma~\ref{lem:nonirred_ohmweird}.\eqref{item:nonirred_ohmweird_inside-2}, we have $W_S  = W_{S'} \times W_{S''}$ where $S'' := S\smallsetminus S'$, and for every $s_i \in S'$, the point $[v_i] \in \Omega$ is fixed by $\rho(W_{S''})$.
Since $\Omega$ is properly convex, the action of $\rho(W_{S''})$ on $\Omega$ is properly discontinuous (see Section~\ref{subsec:prop-conv-proj}), the group $W_{S''}$ is finite (possibly trivial).
If $\# S' = 1$, then $W_S = W_{S'} \times W_{S''}$ is finite, hence by \cite[Lem.\,10]{vin71} there is only one possible Tits--Vinberg domain for $\rho(W_S)$, namely $\OhmVin = \PP(V)$: this is not possible since we have assumed that $\Omega$ is not contained in $\OhmVin$.
Therefore $\# S' = 2$.
By Lemma~\ref{lem:nonirred_ohmweird}.\eqref{item:nonirred_ohmweird_inside-1}, the Coxeter group $W_{S'}$ is of type $\widetilde{A}_1$ and the Cartan submatrix $\Cart_{S'}$ is of negative type.
\end{proof}

\section{The minimal invariant convex subset of the Tits--Vinberg domain} \label{sec:Omega-min}

In this section we consider an irreducible Coxeter group $W_S$ and a representation $\rho : W_S\to\GL(V)$ of~$W_S$ as a reflection group, associated to some $\alpha = (\alpha_1,\dots,\alpha_N) \in {V^*}^N$ and $v = (v_1,\dots,v_N) \in V^N$, such that the Cartan matrix $\Cart=(\alpha_i(v_j))_{1\leq i,j\leq N}$ is of negative type (Definition~\ref{def:type}) and $\rho(W_S)$ is reduced and dual-reduced (Definition~\ref{def:V-v-alpha}).
By Proposition~\ref{prop:rho-irred}, such a representation $\rho$ is irreducible.
By Proposition~\ref{prop:maximal}, the Tits--Vinberg domain $\OhmVin$ is properly convex and contains all other $\rho(W_S)$-invariant convex open subsets of $\PP(V)$.
By Proposition~\ref{prop:benoist-irred-Gamma}, there is a unique smallest nonempty $\Gamma$-invariant convex open subset $\Omega_{\min}$ of $\OhmVin$.
The goal of this section is to describe a fundamental domain for the action of $\rho(W_S)$ on $\Omega_{\min}$, in terms of $(\alpha,v)$ (Lemma~\ref{lem:OmegaMin} and Theorem~\ref{thm:minimal_convex}).

\subsection{Reflections in the dual projective space} \label{subsec:refl-dual}

Recall that the dual action $\rho^*$ of $W_S$ on $V^*$ exchanges the roles of the $v_i$ and the~$\alpha_i$: namely, $\rho^*(s_i)$ is a reflection in the hyperplane defined by~$v_i$, with $(-1)$-eigenvector~$\alpha_i$.
The Cartan matrix for $\rho^*(W_S)$ is the transpose of the Cartan matrix $\Cart = (\alpha_i(v_j))_{1\leq i,j\leq N}$ for $\rho(W_S)$.
Similarly to Section~\ref{subsec:Vinberg-constr}, we define a closed fundamental polyhedral cone $\widetilde{\mathcal{D}} \subset V^*$ for $\rho^*(W_S)$, cut out by the kernels of $v_1, \ldots, v_N$ seen as linear forms on~$V^*$.
The group $\rho^*(W_{S})$ acts properly discontinuously on a nonempty convex open cone $\widetilde{\mathcal{O}}_{\scriptscriptstyle\mathrm{TV}}$ of~$V^*$, namely the interior of the union of all $\rho^*(W_{S})$-translates of~$\widetilde{\mathcal{D}}$.
The image $\mathcal{O}_{\scriptscriptstyle\mathrm{TV}}\subset\PP(V^*)$ of $\widetilde{\mathcal{O}}_{\scriptscriptstyle\mathrm{TV}}$ is the Tits--Vinberg domain of~$\rho^*$; by Proposition~\ref{prop:maximal}, it is again properly convex and contains all other $\rho^*(W_S)$-invariant convex subsets of $\PP(V^*)$.
Since duality between properly convex sets reverses inclusion, the minimal $\rho(W_S)$-invariant properly convex subset $\Omega_{\min}$ of $\PP(V)$ is the dual of the maximal $\rho^*(W_S)$-invariant properly convex subset $\mathcal{O}_{\scriptscriptstyle\mathrm{TV}}$ of $\PP(V^*)$, \ie $\Omega_{\min} = \mathcal{O}_{\scriptscriptstyle\mathrm{TV}}^*$.

We define a dualization operation $\overline{*}$ for closed sets: if $\mathcal{F}$ is a closed convex cone in~$V$, then $\mathcal{F}^{\overline{*}}$ is the closed convex cone in~$V^*$ defined by
$$\mathcal{F}^{\overline{*}} := \big \{ \varphi \in V^*~|~ \varphi(x) \leq 0 \quad \forall x\in\mathcal{F} \big \}.$$
Note that $\mathcal{F}^{\overline{*}\,\overline{*}} = \mathcal{F}$.
For $\widetilde{\mathcal{D}}, \widetilde{\mathcal{O}}_{\scriptscriptstyle\mathrm{TV}} \subset V^*$ as above, we have
\begin{equation} \label{eqn:dual-corners}
\widetilde{\mathcal{D}}^{\overline{*}} =  \sum_{j=1}^N \RR_{\geq 0} \, v_j \subset V
\end{equation}
and ${\overline{\widetilde{\mathcal{O}}_{\scriptscriptstyle\mathrm{TV}}}}^{\overline{*}} \subset V$ is a lift $\overline{\widetilde{\Omega}_{\min}}$ of $\overline{\Omega_{\min}}$.

\begin{lemma} \label{lem:OmegaMin}
If $W_S$ is an irreducible Coxeter group and $\rho : W_S\to\GL(V)$ a representation of~$W_S$ as a reflection group with a Cartan matrix of negative type, such that $\rho(W_S)$ is reduced and dual-reduced, then
$$\overline{\Omega_{\min}} = \! \bigcap_{\gamma \in W_S} \! \rho(\gamma) \cdot \mathcal{D}^{\overline{*}},$$
where $\mathcal{D}^{\overline{*}} \subset \PP(V)$ is the image of $\widetilde{\mathcal{D}}^{\overline{*}}$.
\end{lemma}

\begin{proof}
It is a well-known fact (see \eg \cite[Th.\,2]{san54}) that if $\{\widetilde{C}_i\}_{i \in I}$ is a family of closed convex cones of~$V$, then $( \bigcap_{i \in I} \widetilde{C}_i)^{\overline{*}}$ is the closed convex hull $\overline{\mathrm{Conv}}\,( \bigcup_{i\in I} \widetilde{C}_i^{\overline{*}})$ of 
the union of their duals.
This implies
$$\overline{\widetilde{\Omega}_{\min}}^{\overline{*}} \,=\, \overline{\widetilde{\mathcal{O}}_{\scriptscriptstyle\mathrm{TV}}} \,=\, \overline{\mathrm{Conv}}\, \Big( \!\bigcup_{\gamma \in W_S} \! \rho(\gamma) \cdot \widetilde{\mathcal{D}} \Big) \,=\, \Big( \!\bigcap_{\gamma \in W_S} \! \rho(\gamma) \cdot \widetilde{\mathcal{D}}^{\overline{*}} \Big)^{\overline{*}},$$
where for the second equality we use the fact that $\overline{\widetilde{\mathcal{O}}_{\scriptscriptstyle\mathrm{TV}}} = \overline{\bigcup_{\gamma \in W_S} \! \rho(\gamma) \cdot \widetilde{\mathcal{D}}}$ is convex.
We then apply $\mathcal{F}^{\overline{*}\,\overline{*}} = \mathcal{F}$ to the closed convex cones $\mathcal{F} = \overline{\widetilde{\Omega}_{\min}}$ and $\mathcal{F} = \bigcap_{\gamma \in W_S} \! \rho(\gamma) \cdot \widetilde{\mathcal{D}}^{\overline{*}}$.
\end{proof}

\subsection{Pruning the fundamental polytope}

Consider the following subsets of the fundamental polyhedral cone $\widetilde \Delta$:
\begin{equation} \label{eqn:Sigma-tilde}
\widetilde{\Sigma} := \widetilde{\Delta} \cap \widetilde{\mathcal{D}}^{\overline{*}} = \Bigg\{ x = \sum_{j=1}^N t_j v_j \in V ~\Big|~ \alpha_i(x) \leq 0 \ \text{ and }\ t_j \geq 0 \quad\forall 1\leq i,j \leq N \Bigg\} \subset \overline{\widetilde{\Omega}_{\scriptscriptstyle\mathrm{TV}} }
\end{equation}
and
\begin{equation} \label{eqn:Sigma-tilde-flat}
\widetilde{\Sigma}^{\flat} := \widetilde{\Delta} \cap \operatorname{Int} \big( \widetilde{\mathcal{D}}^{\overline{*}} \big) = \Bigg\{ x = \sum_{j=1}^N t_j v_j \in V ~\Big|~ \alpha_i(x) \leq 0 \ \text{ and }\ t_j > 0 \quad\forall 1\leq i,j \leq N \Bigg\} \subset \widetilde{\Sigma}.
\end{equation}
Note that $\widetilde{\Sigma}^{\flat}$ (hence $\widetilde{\Sigma}$) has nonempty interior by Remark~\ref{rem:neg-type-interpret}.
Let $\Sigma$ and~$\Sigma^{\flat}$ be the respective images of $\widetilde{\Sigma}$ and $\widetilde{\Sigma}^{\flat}$ in $\PP(V)$.
See Figures~\ref{fig:Truncate} and~\ref{fig:Core} for an illustration in some particular cases.
Here is the main result of this section.
\begin{figure}[h]
\centering
\labellist
\small\hair 2pt
\pinlabel {${}_{t_1=0}$} [u] at 35 151
\pinlabel {\rotatebox{90}{${}_{t_3=0}$}} [u] at 15 130
\pinlabel {$\Sigma$} [u] at 60 100
\pinlabel {$\OhmVin$} [u] at 120 55
\pinlabel {${}^\Delta$} [u] at 79 153
\pinlabel {$\mathrm{Ker}(\alpha_1)$} [u] at 50 78
\pinlabel {\rotatebox{-90}{$\mathrm{Ker}(\alpha_3)$}} [u] at 88 119
\pinlabel {\rotatebox{45}{$\mathrm{Ker}(\alpha_2)$}} [u] at 40 129
\endlabellist
\includegraphics[scale=1.1]{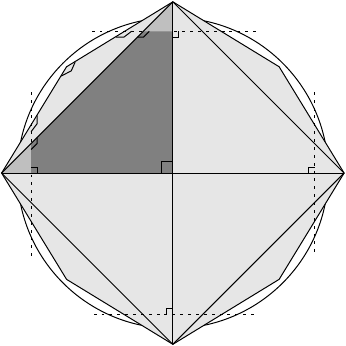}
\caption{The sets $\Delta$, $\Sigma$, and $\OhmVin$ for $W_S:=\langle s_1, s_2, s_3 \,|\, s_i^2=1=(s_1 s_3)^2\rangle$ acting on $V=\RR^3$ as a reflection group, preserving a copy of~$\HH^2$ in $\PP(V)$.
Here $(t_1,t_2,t_3)\mapsto t_1v_1+t_2v_2+t_3v_3$ gives coordinates on~$V$.
The set $\Delta$ (light gray) is a triangle bounded by the hyperplanes $\mathrm{Ker}(\alpha_i)$, and $\Sigma$ (dark gray) is the truncation of $\Delta$ by the hyperplanes $\{t_i=0\}$ (note that $\{t_2=0\}$ is at infinity).
The set $\OhmVin$ is the interior of the $W_S$-orbit of~$\Delta$ (here approximated by 8 iterates), and contains~$\Sigma$.}
\label{fig:Truncate}
\end{figure}
\begin{figure}[h]
\centering
\labellist
\small\hair 2pt
\pinlabel {${}^1$} [u] at 140 100 
\pinlabel {${}^2$} [u] at 305 170 
\pinlabel {${}^3$} [u] at 420 60 
\pinlabel {${}^4$} [u] at 690 145 
\pinlabel {${}^5$} [u] at 840 60 
\endlabellist
\hspace*{-4.0cm} \raisebox{2.0cm}{\includegraphics[width=5cm]{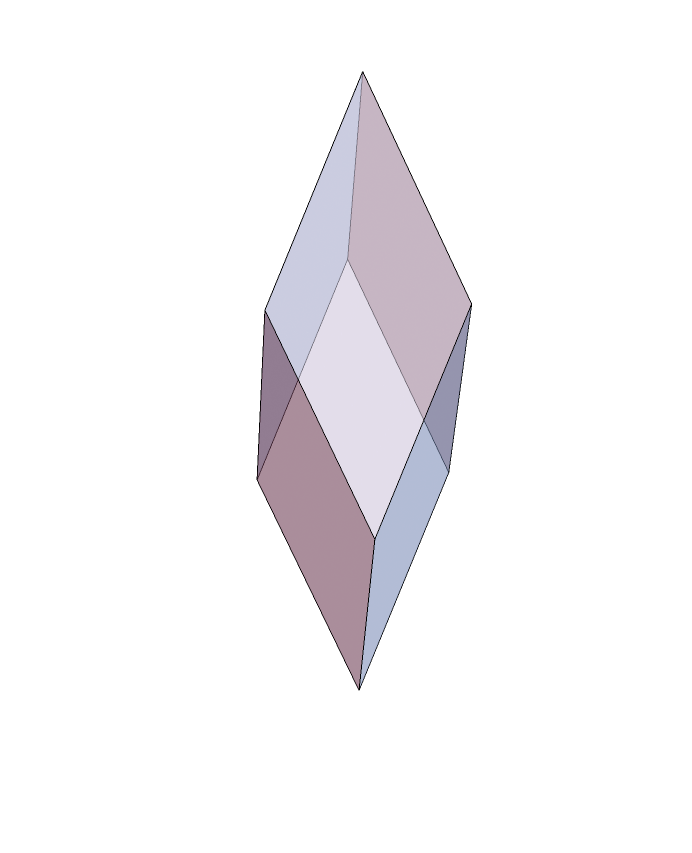}}
\hspace*{-3.5cm} \raisebox{1.5cm}{\includegraphics[width=5.5cm]{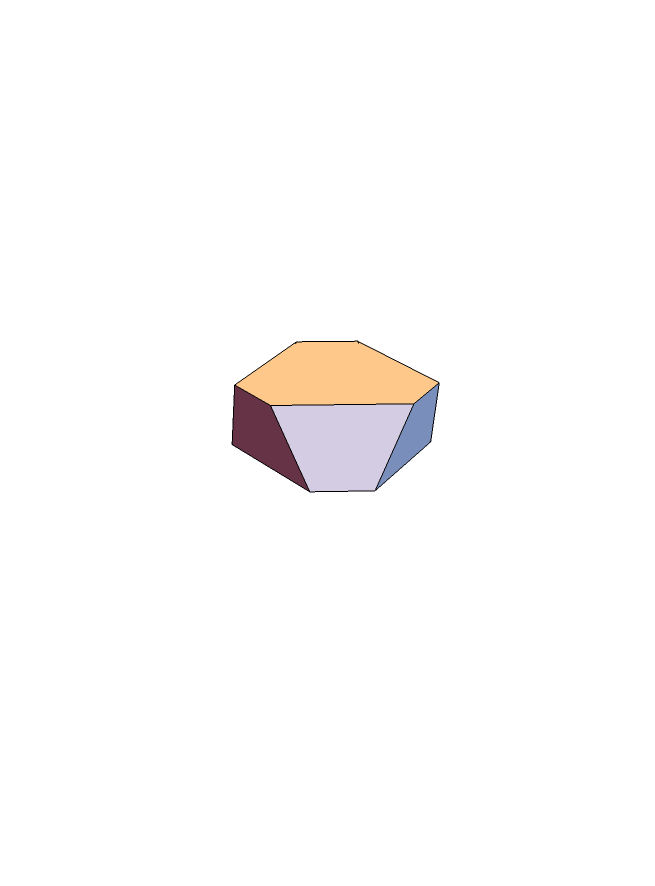}}
\hspace*{-3.9cm} \raisebox{0.6cm}{\includegraphics[width=7cm]{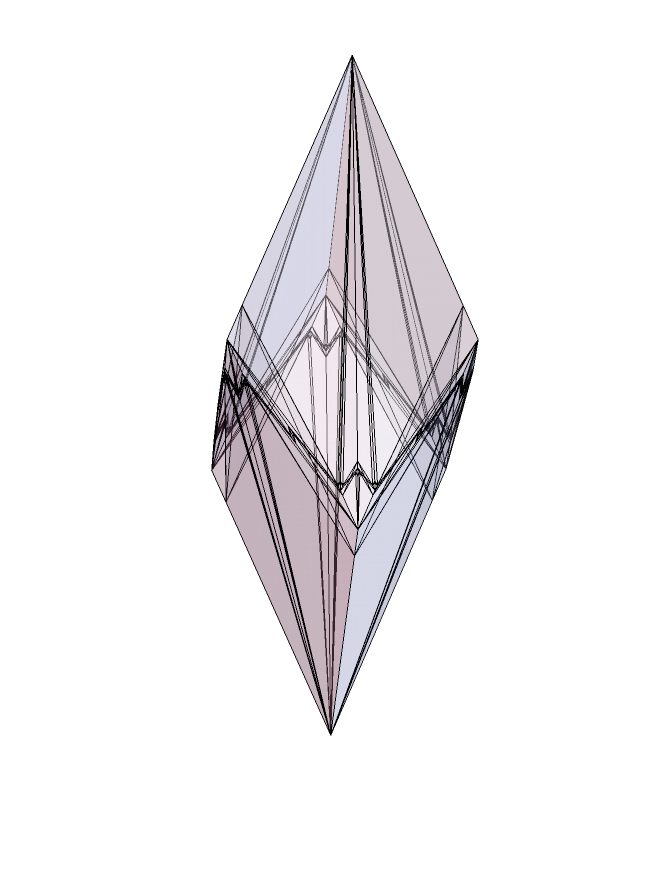}}
\hspace*{-5.4cm} \raisebox{0.1cm}{\includegraphics[width=9.5cm]{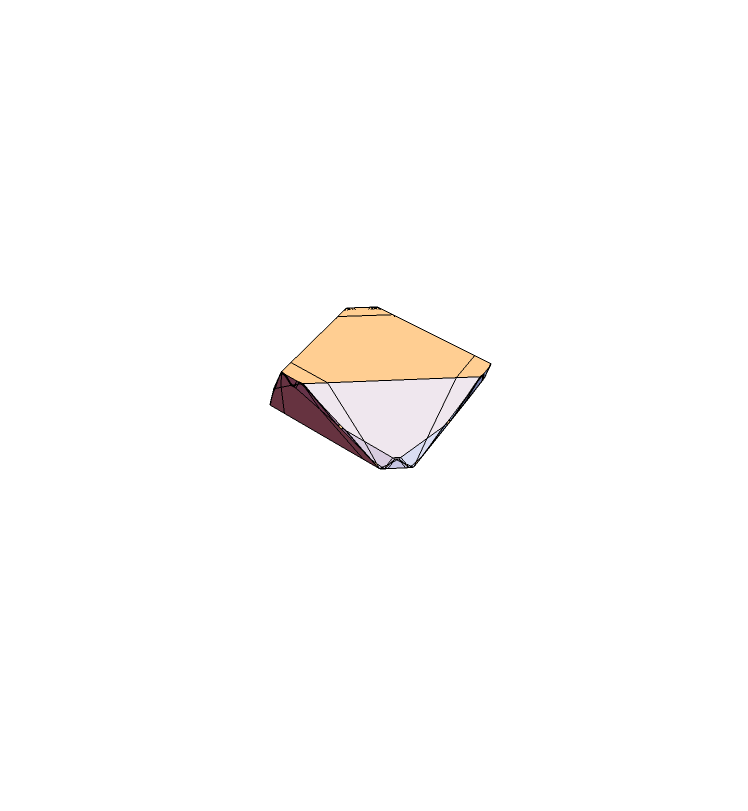}}
\hspace*{-6.6cm} \raisebox{0.0cm}{\includegraphics[width=9.8cm]{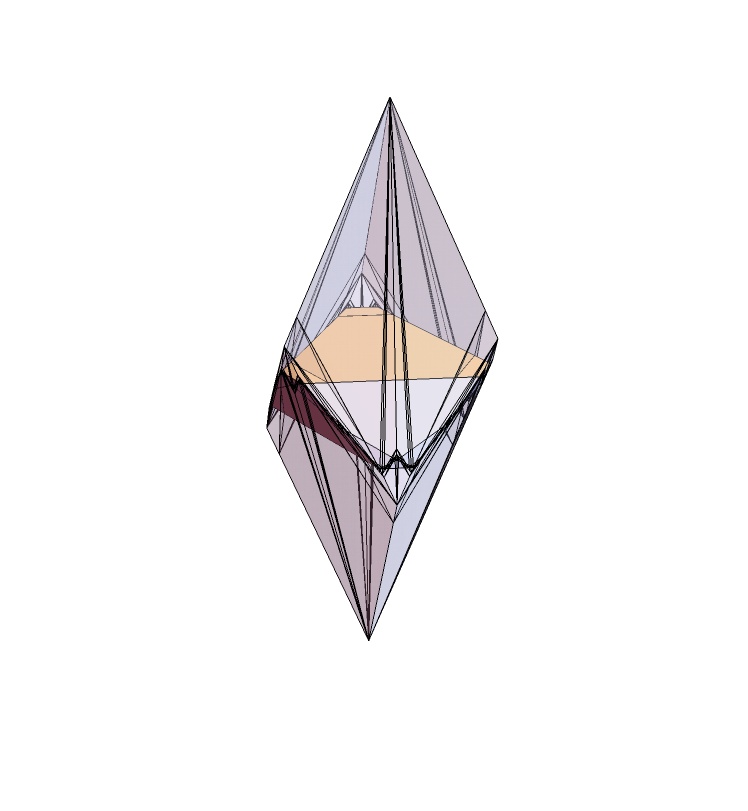}}
\hspace*{-5cm} \vspace{-2cm}
\caption{We consider a right-angled hexagon group $W_S$ and a representation $\rho : W_S\to\GL(4,\RR)$ of $W_S$ as a reflection group. Panels 1 and~2 show the sets $\Delta$ and $\Sigma$ in an affine chart of $\PP(\RR^4)$. Panels 3 and~4 show  the orbits $\bigcup_{\gamma \in W_S} \rho(\gamma) \cdot \Delta$ and $\bigcup_{\gamma \in W_S} \rho(\gamma) \cdot \Sigma$, whose respective interiors are $\OhmVin$ and~$\Omega_{\min}$. Panel 5 shows their superposition.}
\label{fig:Core}
\end{figure}

\begin{theorem} \label{thm:minimal_convex}
If $W_S$ is an irreducible Coxeter group and $\rho : W_S\to\GL(V)$ a representation of~$W_S$ as a reflection group with a Cartan matrix $\Cart = (\Cart_{i,j})_{1\leq i,j\leq N}$ of negative type, such that $\rho(W_S)$ is reduced and dual-reduced, then
$$\Omega_{\min} = \bigcup_{\gamma\in W_S} \rho(\gamma) \cdot \Sigma^{\flat}$$
and $\Sigma^{\flat}$ is a fundamental domain for the action of $\rho(W_S)$ on~$\Omega_{\min}$; it is homeomorphic to the quotient $\rho(W_S)\backslash \Omega_{\min}$.
\end{theorem}

Theorem~\ref{thm:minimal_convex} is a consequence of Lemma~\ref{lem:OmegaMin} and of the following two lemmas.

\begin{lemma} \label{lem:Sigma-flat-in-OhmVin}
In the setting of Theorem~\ref{thm:minimal_convex}, the set $\Sigma^{\flat}$ is contained in the set $\Delta^\flat$ of Fact~\ref{fact:Delta-flat} and $\bigcup_{\gamma\in W_S}\rho(\gamma)\cdot \Sigma^{\flat}$ is the interior of $\bigcup_{\gamma\in W_S}\rho(\gamma)\cdot \Sigma$.
\end{lemma}

\begin{lemma} \label{lem:transl-Sigma-flat}
In the setting of Theorem~\ref{thm:minimal_convex}, we have $\bigcup_{\gamma\in W_S}\rho(\gamma)\cdot \Sigma \subset \bigcap_{\gamma\in W_S} \rho(\gamma) \cdot \mathcal{D}^{\overline{*}}$.
\end{lemma}

\begin{proof}[Proof of Theorem~\ref{thm:minimal_convex} assuming Lemmas~\ref{lem:Sigma-flat-in-OhmVin} and~\ref{lem:transl-Sigma-flat}]
It follows from Lemma~\ref{lem:OmegaMin} that $\overline{\Omega_\mathrm{min}} \cap \Delta \subset \mathcal{D}^{\overline{*}} \cap \Delta = \Sigma$, hence $\overline{\Omega_\mathrm{min}} \cap \bigcup_{\gamma\in W_S}\rho(\gamma)\cdot \Delta \subset \bigcup_{\gamma\in W_S}\rho(\gamma)\cdot \Sigma$, and so $\overline{\Omega_\mathrm{min}} \cap \OhmVin \subset \bigcup_{\gamma\in W_S}\rho(\gamma)\cdot \Sigma$.
Taking interiors and applying Lemma~\ref{lem:Sigma-flat-in-OhmVin}, we find $\Omega_{\min} \subset \bigcup_{\gamma\in W_S}\rho(\gamma)\cdot  \Sigma^{\flat}$.

By Lemmas~\ref{lem:OmegaMin} and~\ref{lem:transl-Sigma-flat}, we have $\bigcup_{\gamma\in W_S}\rho(\gamma)\cdot \Sigma \subset \overline{\Omega_{\min}}$.
By taking interiors and applying Lemma~\ref{lem:Sigma-flat-in-OhmVin}, we obtain $\bigcup_{\gamma\in W_S}\rho(\gamma)\cdot \Sigma^{\flat} \subset \Omega_{\min}$.
\end{proof}

\subsection{Proof of Lemma~\ref{lem:Sigma-flat-in-OhmVin}} \label{subsec:proof-Sigma-flat-in-OhmVin}

By Fact~\ref{fact:Delta-flat}, the set $\widetilde \Delta^{\flat}$ is equal to $\widetilde \Delta$ minus its faces of infinite stabilizer.
Let us check that any point of $\Sigma^\flat = \Delta \cap \operatorname{Int}\big({\mathcal D}^{\overline{*}}\big)$ has finite stabilizer in $W_S$; equivalently, we can work in the vector space~$V$ and check that any point of $\widetilde \Sigma^\flat$ has finite stabilizer in $W_S$.
Recall that the stabilizer of any point $x \in \Delta$ is the standard subgroup $W_{S_x}$ where $S_x = \{s_i \in S \mid \alpha_i(x) = 0\}$ \cite[Th.\,2.(5)]{vin71}.

Suppose by contradiction that $x \in \widetilde \Sigma^\flat$ has infinite stabilizer.
Let $W_{T}$ be an irreducible standard subgroup of $W_{S_x}$ which is still infinite.
By Fact~\ref{fact:types}, the Cartan submatrix $\Cart_{T}:=(\Cart_{i,j})_{s_i,s_j\in T}$ is of negative or zero type. Since $x \in \widetilde \Sigma^\flat = \widetilde{\Delta} \cap \operatorname{Int}\big (\widetilde{\mathcal{D}}^{\overline{*}}\big)$, we can write $x = \sum_{j=1}^N t_j v_j$ where $\alpha_i(x)\leq 0$ and $t_j > 0$ for all $1\leq i,j\leq N$.
For any $s_i\in T$, we have
\begin{align}\label{eqn:split-up}
0 = \alpha_i(x)  &= \sum_{ s_j \in T } \Cart_{i,j} t_j +   \sum_{ s_j \in S \smallsetminus T } \Cart_{i,j} t_j.
\end{align}
The terms of the second sum are all nonpositive, hence the first sum is nonnegative.
By Fact~\ref{fact:type-sign-entries}.\eqref{item:neg-type-sign-entries}, the Cartan submatrix $\Cart_{T}$ cannot be of negative type, hence it is of zero type.
By Fact~\ref{fact:type-sign-entries}.\eqref{item:zero-type-sign-entries}, the first sum of~\eqref{eqn:split-up} must then be zero.
This implies that the second sum is zero, and so every term of the second sum is zero.
However, since $W_S$ is irreducible, there exist $s_i \in T$ and $s_j \in S \smallsetminus T$ such that $\Cart_{i,j} < 0$.
Then $\Cart_{i,j} t_j = 0$ implies $t_j = 0$, contradicting the fact that $t_j > 0$.
This completes the proof that $\Sigma^\flat \subset \Delta^\flat$.

Let us now check that $\bigcup_{\gamma\in W_S}\rho(\gamma)\cdot \Sigma^\flat$ is the interior of $\bigcup_{\gamma\in W_S}\rho(\gamma)\cdot \Sigma$.
Observe that for any $x \in \Delta^\flat$, an open neighborhood of $x$ in $\OhmVin$ is given by $\bigcup_{\gamma\in W_{S_x}} \rho(\gamma)\cdot U$ where $U$ is a relatively open neighborhood of $x$ in $\Delta^\flat$ which touches only the reflection walls containing~$x$, and $\rho(W_{S_x})$ is the (finite) subgroup of $\rho(W_S)$ generated by the reflections along walls containing~$x$.
It follows that for any subset $R$ of $\Delta^\flat$, the interior of $\bigcup_{\gamma\in W_S}\rho(\gamma)\cdot R$ is $\bigcup_{\gamma\in W_S}\rho(\gamma)\cdot R'$ where $R'$ is the relative interior of $R$ in $\Delta^\flat$. 
Since $\Sigma^\flat$ is the relative interior of $\Sigma \cap \Delta^\flat$ in $\Delta^\flat$, we obtain that $\bigcup_{\gamma\in W_S}\rho(\gamma)\cdot \Sigma^\flat$ is the interior of $\big(\bigcup_{\gamma\in W_S}\rho(\gamma)\cdot \Sigma\big) \cap \OhmVin$, hence of $\bigcup_{\gamma\in W_S}\rho(\gamma)\cdot \Sigma$.

\subsection{Proof of Lemma~\ref{lem:transl-Sigma-flat}}

It is enough to establish that
\begin{equation} \label{eqn:describe-Delta}
\widetilde{\Delta} = \big\{ x \in V ~|~ \rho(\gamma)x - x \in \widetilde{\mathcal{D}}^{\overline{*}} \quad\forall \gamma \in W_{S} \big\}.
\end{equation}
Indeed, for any $x \in \widetilde{\Sigma} = \widetilde{\Delta} \cap \widetilde{\mathcal{D}}^{\overline{*}}$ and any $\gamma\in W_S$ we then have $\rho(\gamma)x \in \widetilde{\mathcal{D}}^{\overline{*}} + x \subset \widetilde{\mathcal{D}}^{\overline{*}}$, hence $\bigcup_{\gamma\in W_S}\rho(\gamma)\cdot \Sigma \subset \mathcal{D}^{\overline{*}}$, and so $\bigcup_{\gamma\in W_S}\rho(\gamma)\cdot \Sigma \subset \bigcap_{\gamma\in W_S} \rho(\gamma) \cdot \mathcal{D}^{\overline{*}}$.

We now establish \eqref{eqn:describe-Delta}, following \cite[Prop.\,3.12]{kac90}.

We say that $\gamma = s_{i_1} \dotsm s_{i_m}$ is a \emph{reduced expression} for $\gamma \in W_S$ if $m$ is minimal among all possible expressions of $\gamma$ as a product of the $s_i$.
The integer $m$ is called the \emph{length} of $\gamma$ and is denoted by $\ell(\gamma)$.
Since all the relations of $W_S$ have even length, we have $|\ell(\gamma s_i) - \ell(\gamma)| = 1 $ for all $\gamma \in W_S$ and $s_i \in S$.
 
\begin{lemma}[{see \cite[Lem.\,3.11]{kac90}}] \label{lem:nonnegativity}
Let $W_S$ be a Coxeter group and $\rho: W_S \to \GL(V)$ a representation of $W_S$ as a reflection group in~$V$, associated to some $\alpha = (\alpha_1,\dots,\alpha_N) \in {V^*}^N$ and $v = (v_1,\dots,v_N) \in V^N$.
For any $\gamma \in W_S$ and $s_i \in S$, the following are equivalent:
\begin{enumerate}
  \item\label{item:nonneg-span-1} $\ell(\gamma s_i) > \ell(\gamma)$;
  \item\label{item:nonneg-span-2} $\rho^{*}(\gamma) \alpha_i$ belongs to the nonnegative span $\widetilde{\Delta}^{\overline{*}}$ of $\alpha_1, \ldots, \alpha_N$ in $V^*$;
  \item\label{item:nonneg-span-3} $\rho(\gamma) v_i$ belongs to the nonnegative span $\widetilde{\mathcal{D}}^{\overline{*}}$ of $v_1, \ldots, v_N$ in $V$.
\end{enumerate}
In particular, if $\gamma = s_{i_1} \dotsm s_{i_m}$ is a reduced expression for $\gamma \in W_S$, then
$$\rho (s_{i_1} \dotsm s_{i_{m-1}})(v_{i_m}) \in \widetilde{\mathcal{D}}^{\overline{*}}.$$
\end{lemma}

\begin{proof}[Proof of Lemma~\ref{lem:nonnegativity}]
We give a proof of \eqref{item:nonneg-span-1}~$\Leftrightarrow$~\eqref{item:nonneg-span-2}; the proof of \eqref{item:nonneg-span-1}~$\Leftrightarrow$~\eqref{item:nonneg-span-3} is similar using dual objects.
Recall that $
\widetilde{\Delta} = \{ x \in V \mid \alpha_i(x) \leq 0,\; \forall 1\leq i\leq N \}$.
The linear form $\alpha \in V^*$ lies in $\widetilde{\Delta}^{\overline{*}}$ if and only if $\alpha(\widetilde{\Delta}) \leq 0$.
Hence, $\rho^*(\gamma)\alpha_i \in \widetilde{\Delta}^{\overline{*}}$ if and only if $\alpha_i(\rho(\gamma^{-1})\widetilde{\Delta}) \leq 0$.
Moreover (see \cite[Lem.\,D.2.2]{dav08}), given $s_i \in S$ and $\gamma \in W_S$, the cones $\operatorname{Int}\,(\widetilde{\Delta})$ and $\rho(\gamma) \operatorname{Int}\,(\widetilde{\Delta})$ lie on opposite sides of the hyperplane $\mathrm{Ker}(\alpha_i) \cap \widetilde{\Omega}_{\scriptscriptstyle\mathrm{TV}}$ of $\widetilde{\Omega}_{\scriptscriptstyle\mathrm{TV}}$ fixed by $\rho(s_i)$ if and only if $\ell(\gamma) > \ell(s_i \gamma)$.
Hence, $\alpha_i(\rho(\gamma^{-1})\widetilde{\Delta}) \leq 0$ if and only if $\ell(s_i \gamma^{-1}) > \ell(\gamma^{-1})$, completing the proof.
\end{proof}

\begin{proof}[Proof of \eqref{eqn:describe-Delta}]
The inclusion $\supset$ is obvious, by~\eqref{eqn:dual-corners} and the definitions \eqref{eqn:rho-s-i} of $\rho(s_i)$ and \eqref{eqn:Delta-tilde} of~$\widetilde{\Delta}$.
We prove the reverse inclusion by induction on $ m = \ell(\gamma)$. For $m = 1$, it is the definition of $\widetilde{\Delta}$.
If $m  > 1$, let $\gamma = s_{i_1} \dotsm s_{i_m}$.
We have
$$\rho(\gamma)x - x = \left(\rho(s_{i_1} \dotsm s_{i_{m-1}})\,x - x\right) + \rho(s_{i_1} \dotsm s_{i_{m-1}})\left(\rho(s_{i_m})\,x - x \right) $$
and we apply the inductive assumption to the first summand and Lemma~\ref{lem:nonnegativity} to the second summand, using the equality $\rho(s_{i_m})x - x = -\alpha_{i_m}(x)\,  v_{i_m}$ from~\eqref{eqn:rho-s-i}.
\end{proof}

\subsection{The fundamental domain $\Sigma \cap \Delta^{\flat}$ of the closed convex core}

As in Section~\ref{subsec:proof-Sigma-flat-in-OhmVin}, the stabilizer of $x$ in~$W_S$ is the standard subgroup $W_{S_x}$ where
$$S_x := \big\{ s_i \in S ~|~ \alpha_i(x)= 0\big\}.$$
Given a decomposition $x = \sum_{j=1}^N t_j v_j \in V$ such that $t := (t_j)_{j=1}^N \in \RR_{\geq 0}^N$, we introduce 
$$
S^{+,t} := \big\{s_j\in S \,|\, t_j > 0\big\}, \quad
S_{x}^{0,t} := \{s_j\in S_x \,|\, t_j = 0\}, \quad \textrm{and} \quad
S_{x}^{+,t} := \{s_j\in S_x \,|\, t_j > 0\} . 
$$
Recall (Fact~\ref{fact:Delta-flat}) that a point $x\in \Delta$ belongs to $\partial \OhmVin$ if and only if $W_{S_x}$ is infinite.

Consider the following two conditions (negating conditions~\ref{item:no-Z2} of Section~\ref{subsec:intro-main} and~\ref{item:not-zero-type} of Theorem~\ref{thm:main}):
\begin{enumerate}
  \item[\namedlabel{item:exists-Z2}{\texttt{(IC)}}] there exist disjoint subsets $S',S''$ of~$S$ such that $W_{S'}$ and $W_{S''}$ are both infinite and commute;
  \item[\namedlabel{item:exists-zero-type}{\texttt{(ZT)}}] there exists an irreducible standard subgroup $W_{S'}$ of $W_S$ with $\emptyset\neq S'\subset S$ such that the Cartan submatrix $\Cart_{S'}$ is of zero type.
\end{enumerate}
The goal of this section is to establish the following.
We refer to Section~\ref{subsec:Coxeter-groups} for the notion of irreducible components of a Coxeter group.

\begin{proposition} \label{prop:boundary_sigma}
In the setting of Theorem~\ref{thm:minimal_convex}, we have $\Sigma \cap \partial \OhmVin \neq \emptyset$ if and only if \ref{item:exists-Z2} or \ref{item:exists-zero-type} holds.
More precisely,
\begin{enumerate}
  \item\label{item:boundary_of_sigma1} suppose that there exists $[x] \in \Sigma \cap \partial \OhmVin$ and write $x = \sum_{j=1}^N t_j v_j$ with $t \in \RR_{\geq 0}^N$; then either $W_{S^{+,t}}$ and $W_{S_x^{0,t}}$ are both infinite and commute (hence \ref{item:exists-Z2} holds), or $S^{+,t}_x$ has an irreducible component $S'$ such that the Cartan submatrix $\Cart_{S'}$ is of zero type (hence \ref{item:exists-zero-type} holds);
  \item\label{item:boundary_of_sigma2} conversely, suppose that \ref{item:exists-Z2} or \ref{item:exists-zero-type} holds with $W_{S'}$ irreducible, and let $t' = (t_j)_{s_j \in S'}$ be a Perron--Frobenius eigenvector of $2\,\mathrm{Id} - \Cart_{S'}$ as in Section~\ref{subsec:neg-type}; then $x' := \sum_{s_j \in S'} t_j v_j \in V$ satisfies $[x'] \in \Sigma \cap \partial \OhmVin$.
\end{enumerate}
\end{proposition}

{We note that if \ref{item:exists-Z2} or \ref{item:exists-zero-type} holds, then up to replacing $S'$ by a smaller subset we may always assume that $W_{S'}$ is irreducible as in~\eqref{item:boundary_of_sigma2}.}

As in Section~\ref{subsec:intro-main}, for any $T \subset S$ we set $T^\perp := \{ s_i \in S\smallsetminus T \, |\, m_{i,j}=2\ \forall s_j \in T \}$.
For two subsets $T_1,T_2 \subset S$, we write $T_1 \perp T_2$ when $T_1 \subset T_2^\perp$ (\ie $m_{i,j}=2$ for all $s_i \in T_1$ and $s_j \in T_2$).
Our proof of Proposition~\ref{prop:boundary_sigma} relies on the following lemma.

\begin{lemma}\label{lem:to_bounda_o_min}
In the setting of Theorem~\ref{thm:minimal_convex}, if $x = \sum_{j=1}^N t_j v_j$ satisfies $t \in \RR_{\geq 0}^N$ and $[x] \in \Sigma$, then
\begin{enumerate}
  \item\label{item:perp} $S^{+,t} \perp S_x^{0,t}$;
  \item\label{item:no_plus} $S^{+,t}$ has no irreducible component $S'$ such that the Cartan submatrix $\Cart_{S'}$ is of positive type;
  \item\label{item:no_minus} $S_x^{+,t}$ has no irreducible component $S'$ such that the Cartan submatrix $\Cart_{S'}$ is of negative type.
\end{enumerate}
In particular, $W_{S^{+,t}}$ is infinite and $W_{S_x} = W_{S_x^{0,t}} \times W_{S_x^{+,t}}$.
\end{lemma}

\begin{proof}[Proof of Lemma~\ref{lem:to_bounda_o_min}]
By definition, $t_j\geq 0$ and $\alpha_i(x)\leq 0$ for all $1\leq i,j\leq N$. 

\eqref{item:perp}
Let us check that any element of~$S_x^{0,t}$ commutes with any element of
$S^{+,t}$. For any $s_i \in S_x$ we have by definition
\begin{equation} \label{eqn:Sx-def}
0 = \alpha_i(x) = \sum_{j=1}^{ N } t_j \alpha_i(v_j) = \sum_{s_j \in S^{+,t}} t_j \Cart_{i,j}.
\end{equation}
If $s_i\in S_x^{0,t}$, then $s_i\notin S^{+,t}$, hence $\Cart_{i,j}\leq 0$ for all $s_j \in S^{+,t}$; thus each term of the right-hand sum in \eqref{eqn:Sx-def} is nonpositive, hence must be zero.
Thus for any $s_i\in S_x^{0,t}$ and $s_j\in S^{+,t}$ we have $\Cart_{i,j} = 0$, which means that $s_i$ and~$s_j$ commute.

\eqref{item:no_plus}
Suppose by contradiction that $S^{+,t}$ has an irreducible component $S'$ such that the Cartan submatrix $\Cart_{S'}$ is of positive type.
By Fact~\ref{fact:types}, we may assume that $\Cart_{S'}$ is symmetric, and positive definite.
Then
\begin{equation*}
\begin{array}{lll}
0 &\geq  \sum_{s_i\in S'} t_i \alpha_i(x) & \text{since $t_i\geq 0$ and $\alpha_i(x)\leq 0$ for all $s_i\in S$}\\
& =  \sum_{s_i\in S'} t_i \sum_{s_j \in S} t_j \, \alpha_i(v_j)\\
& =  \sum_{s_i\in S'} t_i \sum_{s_j \in S^{+,t}}  t_j \, \alpha_i(v_j) & \text{since $t_j = 0$ whenever $s_j \in S \smallsetminus S^{+,t}$}\\
&=  \sum_{s_i\in S'} \; \sum_{s_j \in S'} t_i t_j \, \alpha_i(v_j) & \text{since $S'$ is an irreducible component of $S^{+,t}$}\\
& =  {}^{\top}\!{t'} \Cart_{S'} t' & \text{where $t':=(t_i)_{s_i\in S'}$}.
\end{array}
\end{equation*}

 This contradicts the fact that the symmetric matrix $\Cart_{S'}$ is positive definite.

\eqref{item:no_minus}
Let $S'$ be an irreducible component of $S^{+,t}_x$.
For any $s_i\in S'$,
$$0 = \alpha_i(x) = \sum_{s_j\in S'} t_j \Cart_{i,j} + \sum_{s_j\in S\smallsetminus S'} t_j \Cart_{i,j} \, ;$$
since $t_j \Cart_{i,j} \leq 0$ for all $s_j\in S\smallsetminus S'$, we have $\sum_{s_j\in S'} t_j \Cart_{i,j} \geq 0$.
Thus the vector $t' = (t_j)_{s_j\in S'}$ has only positive entries and $\Cart_{S'} t' = (\sum_{s_j\in S'} t_j \Cart_{i,j})_{s_i\in S'}$ has only nonnegative entries.
By Fact~\ref{fact:type-sign-entries}.\eqref{item:neg-type-sign-entries}, the Cartan submatrix $\Cart_{S'}$ is not of negative type.
\end{proof}

\begin{proof}[Proof of Proposition~\ref{prop:boundary_sigma}]
\eqref{item:boundary_of_sigma1}
Since $[x] \in \partial \OhmVin$, the group $W_{S_x}$ is infinite by Fact~\ref{fact:Delta-flat}, hence $W_{S_x^{0,t}}$ is infinite or $W_{S_x^{+,t}}$ is infinite by Lemma~\ref{lem:to_bounda_o_min}.
If $W_{S_x^{0,t}}$ is infinite, then $W_{S^{+,t}}$ and $W_{S_x^{0,t}}$ are both infinite and commute by Lemma~\ref{lem:to_bounda_o_min}.\eqref{item:perp}--\eqref{item:no_plus}.
Otherwise $W_{S_x^{+,t}}$ is infinite, and then $S_x^{+,t}$ has an irreducible component $S'$ such that the Cartan submatrix $\Cart_{S'}$ is of zero type, since negative type is excluded by Lemma~\ref{lem:to_bounda_o_min}.\eqref{item:no_minus}.

\eqref{item:boundary_of_sigma2}
Suppose \ref{item:exists-Z2} holds, \ie there exist disjoint subsets $S',S'' \subset S$ such that $W_{S'}$ and $W_{S''}$ are both infinite and commute, and suppose $W_{S'}$ is irreducible.
Since $W_{S'}$ is infinite, the Cartan submatrix $\Cart_{S'}$ is not of positive type.
By definition (see Section~\ref{subsec:neg-type}), the Perron--Frobenius eigenvector $t' = (t_j)_{s_j \in S'}$ is a positive vector and $\Cart_{S'} t'$ is a zero or negative vector.
For $s_i \in S''$ we have $\alpha_i(x') = 0$; for $s_i \not \in S' \cup S''$ we have $\alpha_i(x') \leqslant 0$; and for $s_i \in S'$ we have
$$\alpha_i(x')  = \sum_{s_j \in S'} t_j \alpha_i(v_j) = \left(\textrm{the $i$-th entry of } \Cart_{S'} t' \right) \leqslant 0.$$
Therefore $S'' \subset S_{x'}$ and $[x'] \in \Sigma$.
Since $W_{S''}$ is infinite, Fact~\ref{fact:Delta-flat} implies $[x'] \in \partial \OhmVin$.

Suppose \ref{item:exists-zero-type} holds, \ie there exists an irreducible standard subgroup $W_{S'}$ with $S' \subset S$ such that $\Cart_{S'}$ is of zero type.
As above, $t' = (t_j)_{s_j \in S'}$ is a positive vector and $\Cart_{S'} t' = 0$.
For $s_i \not \in S'$ we have $\alpha_i(x') \leqslant 0$, and for $s_i \in S'$ we have
$$\alpha_i(x')  = \sum_{s_j \in S'} t_j \alpha_i(v_j) = \left(\textrm{the $i$-th entry of } \Cart_{S'} t'\right)  = 0.$$
Therefore $S' \subset S_{x'}$ and $[x'] \in \Sigma$.
Since $W_{S'}$ is infinite, Fact~\ref{fact:Delta-flat} implies $[x'] \in \partial \OhmVin$.
\end{proof}

\section{Proof of Theorem~\ref{thm:main}} \label{sec:main_thm}

The implication \ref{item:cc-cox}~$\Rightarrow$~\ref{item:naive-cc-cox} of Theorem~\ref{thm:main} is immediate from the definitions.
The equivalence \ref{item:not-zero-type}~$\Leftrightarrow$~\ref{item:not-zero-type-explicit} is contained in Fact~\ref{fact:types}.
We now establish the other implications of Theorem~\ref{thm:main}.

\subsection{The affine case}

When $W_S$ is irreducible and affine, Theorem~\ref{thm:main} is contained in the following.

\begin{proposition} \label{prop:affine}
Let $W_S$ be an affine irreducible Coxeter group and $\rho : W_S\to\GL(V)$ a representation of~$W_S$ as a reflection group with Cartan matrix $\Cart$.
Then the following are equivalent:
\begin{enumerate}
  \item\label{item:affine-1} $\rho(W_S)$ is convex cocompact in $\PP(V)$;
  \item\label{item:affine-1bis} $\rho(W_S)$ is naively convex cocompact in $\PP(V)$;
  \item\label{item:affine-2} $\rho(W_S)$ does not contain any unipotent element;
  \item\label{item:affine-3} the Cartan matrix $\Cart$ is \emph{not} of zero type;
  \item\label{item:affine-3bis} $\det\Cart \neq 0$;
  \item\label{item:affine-4} $W_S$ is of type $\widetilde{A}_{N-1}$ where $N\geq 2$ (see Table~\ref{table:affi_diag}) and the Cartan matrix $\Cart$ is of negative type.
\end{enumerate}
\end{proposition}

In this case, if $\rho(W_S)$ is reduced and dual-reduced, then $N=\dim(V)$ and $\OhmVin$ is a simplex in $\PP(V)$, divided by $\rho(W_S)$: see Example~\ref{ex:N=2} if $N=2$ and Lemma~\ref{lem:out_of_vinberg} if $N\geq 3$.

\begin{proof}
The implication \eqref{item:affine-1}~$\Rightarrow$~\eqref{item:affine-1bis} is immediate from the definitions.
For \eqref{item:affine-1bis}~$\Rightarrow$~\eqref{item:affine-2}, see Proposition~\ref{prop:cc-no-unipotent}.
For \eqref{item:affine-3}~$\Rightarrow$~\eqref{item:affine-3bis}~$\Rightarrow$~\eqref{item:affine-4}, see Fact~\ref{fact:types}.

For \eqref{item:affine-2}~$\Rightarrow$~\eqref{item:affine-3}, note that $\rho^\alpha(W_S) \subset \GL(V^\alpha)$ is reduced, where $\rho^\alpha: W_S\to\GL(V^\alpha)$  is the representation of $W_S$ as a reflection group in $V^\alpha$ induced by~$\rho$ (see Remarks~\ref{rem:notation_rhos} and~\ref{rem:rho-v-alpha-refl-group}).
If $\Cart$ is of zero type, then by Fact~\ref{fact:types}, the group $\rho^\alpha (W_{S})$ acts properly discontinuously and cocompactly on some affine chart of $\PP(V^{\alpha})$, preserving some Euclidean metric; in particular, $\rho^\alpha (W_{S})$ contains a translation of this affine chart, \ie a unipotent element of $\GL(V^\alpha )$.
Since the action of $\rho(W_S)$ on~$V_{\alpha}$ is trivial, $\rho(W_S)$ contains a unipotent element of $\GL(V)$.

Let us check \eqref{item:affine-4}~$\Rightarrow$~\eqref{item:affine-1}.
Suppose $W_S$ is of type $\widetilde{A}_{N-1}$ and $\Cart$ is of negative type.
By Fact~\ref{fact:types}, we have $\det\Cart\neq 0$, hence $V_v\cap V_{\alpha}=\{0\}$.
Let $\rho_v : W_S\to \GL(V_v) = \GL(V_v^\alpha)$ be the representation of $W_S$ as a reflection group induced by~$\rho$, as in Remarks~\ref{rem:notation_rhos} and~\ref{rem:rho-v-alpha-refl-group}; it is reduced and dual-reduced.
Let $\OhmVin^{v} \subset \PP(V_v)$ be the corresponding Tits--Vinberg domain.
By Lemma~\ref{lem:out_of_vinberg}, the group $\rho_v (W_S)$ divides $\OhmVin^{v}$, hence it is convex cocompact in $\PP(V_v)$.
Therefore $\rho(W_S)$ is convex cocompact in $\PP(V)$ by Corollary~\ref{cor:cc-subspace}.
\end{proof}

\subsection{Preliminary reductions} \label{subsec:rho-red}

We now consider an infinite Coxeter group $W_S$ which is not necessarily irreducible.
For any representation $\rho : W_S\to\GL(V)$ of $W_S$ as a reflection group, associated to some $\alpha \in {V^*}^N$ and $v \in V^N$, we denote by $\rho^{\alpha} : W_S\to\GL(V^{\alpha})$ and $\rho_v^\alpha : W_S\to\GL(V_v^\alpha)$ the induced representations of $W_S$ as in Remark~\ref{rem:notation_rhos}; they have the same Cartan matrix $\Cart$ as~$\rho$.
If the Cartan matrix associated to each irreducible factor of $W_S$ is of negative type, then $\rho^{\alpha}$ and~$\rho_v^{\alpha}$ are still representations of $W_S$ as a reflection group (Remark~\ref{rem:rho-v-alpha-refl-group-non-irred}), hence injective.

\begin{lemma} \label{lem:V-alpha-v}
Let $W_S$ be an infinite (not necessarily irreducible) Coxeter group in $N$ generators, and let $\rho : W_S\to\GL(V)$ be a representation of $W_S$ as a reflection group, associated to some $\alpha \in {V^*}^N$ and $v \in V^N$.
Suppose that for each irreducible factor $W_{S_{\ell}}$, the Cartan submatrix $\Cart_{S_{\ell}} = (\alpha_i(v_j))_{s_i,s_j\in S_{\ell}}$ is of negative type.
Then
\begin{enumerate}
  \item\label{item:V-alpha-v-cc} $\rho(W_S)$ is convex cocompact (\resp strongly convex cocompact) in $\PP(V)$ if and only if $\rho_v^\alpha (W_S)$ is convex cocompact (\resp strongly convex cocompact) in $\PP(V_v^\alpha)$;
  \item\label{item:V-alpha-v-naive-cc} if $\rho(W_S)$ is naively convex cocompact in $\mathbb{P} (V)$, then $\rho_v(W_S)$, $\rho^{\alpha}(W_S)$, and $\rho_v^{\alpha}(W_S)$ are naively convex cocompact in $\PP(V_v)$, $\PP(V^{\alpha})$, and $\PP(V_v^{\alpha})$, respectively.
\end{enumerate}
\end{lemma}

\begin{proof}
\eqref{item:V-alpha-v-cc} We refer to the matrices~\eqref{eqn:ninezeros}: by Proposition~\ref{prop:cc-quotient}, the group $\rho(W_S)$ is convex cocompact in $\PP(V)$ if and only if $\rho^\alpha(W_S)$ is convex cocompact in $\PP(V^\alpha)$ which, by Corollary~\ref{cor:cc-subspace}, happens if and only if $\rho^\alpha_v(W_S)$ is convex cocompact in $\PP(V_v^\alpha)$.
The similar statement on strong convex cocompactness follows by Proposition~\ref{prop:cc-hyp-strong-cc}.

\eqref{item:V-alpha-v-naive-cc} Let $\Omega$ be a $\rho(W_S)$-invariant properly convex open subset of $\PP(V)$ and $\C \subset \Omega$ a nonempty $\rho(W_S)$-invariant closed convex subset which has compact quotient by $\rho(W_S)$.
Then, by Proposition~\ref{prop:nonirred_ohmweird}, $\Omega$ is contained in one of the finitely many Tits--Vinberg domains for $\rho$ unless  $W_S$ is affine of type $\widetilde{A}_1$, in which case we conclude using Proposition~\ref{prop:affine}.
So we may assume $\Omega$ is contained in one of the finitely many Tits--Vinberg domains $\OhmVin$.

We first note that $\rho_v(W_S)$ is naively convex cocompact in $\mathbb{P} (V_v)$.
Indeed, the set $\C \cap\nolinebreak \PP(V_v)$ is nonempty by Corollary~\ref{cor:C-meets-P(Vv)}.\eqref{item:C-meets-P(Vv)}.
It is a $\rho_v (W_S)$-invariant closed subset of the $\rho_v (W_S)$-invariant properly convex open subset $\Omega \cap \mathbb{P} (V_v)$ of $\mathbb{P} (V_v)$, and the action of $\rho_v (W_S)$ on it is cocompact since it is a closed subset of~$\C$ and the action of $\rho(W_S)$ on~$\C$ is cocompact.

We next check that $\rho^\alpha(W_S)$ is naively convex cocompact in $\mathbb{P} (V^\alpha)$.
Every point of $\mathbb{P} (V_\alpha)$ has infinite stabilizer (all of $W_S$), hence $\mathbb{P} (V_\alpha) \cap \OhmVin = \emptyset$ by Fact~\ref{fact:Delta-flat}.
Let $\pi : \PP(V) \smallsetminus \PP(V_{\alpha}) \to \PP(V^\alpha)$ be the natural projection.
The projection $\pi(\C)$ is a $\rho^\alpha (W_S)$-invariant closed convex subset of the $\rho^\alpha (W_S)$-invariant convex open subset $\pi(\Omega)$ of $\mathbb{P} (V^\alpha)$.
Since each irreducible block of the Cartan matrix is of negative type, $\OhmVin^\alpha = \pi (\OhmVin)$ is properly convex by Proposition~\ref{prop:maximal-non-irred}.
From $\Omega \subset \OhmVin$, we get that $\pi(\Omega)$ is properly convex.
The image $\Delta^\alpha \subset \PP(V^\alpha)$ of the fundamental polyhedral cone $\widetilde{\Delta}^{\alpha}$ for $\rho^{\alpha}(W_S)$ satisfies $\pi^{-1}(\Delta^\alpha) = \Delta \smallsetminus \PP(V_\alpha)$. In particular, $\pi(\C) \cap \Delta^\alpha = \pi(\C \cap \Delta)$ is a compact fundamental domain for the action of $\rho^\alpha (W_S)$ on $\pi(\C)$.

Since $\rho_v (W_S)$ is naively convex cocompact in $\mathbb{P} (V_v)$ and since $(V_v)_\alpha = V_v \cap V_\alpha$ implies $(V_v)^\alpha = V_v^\alpha$, we get that the group $\rho^\alpha_v (W_S)$ is naively convex cocompact in $\mathbb{P} (V^\alpha_v)$ by applying the previous reasoning to $\rho_v(W_S)$.
\end{proof}

\subsection{A convex subset $\C$ of $\OhmVin$ with compact quotient by $\rho(W_S)$, for large~$W_S$} \label{subsec:cox-construct-C}

We now establish the equivalence \ref{item:naive-cc-cox}~$\Leftrightarrow$~\ref{item:not-zero-type} of Theorem~\ref{thm:main} in the case that the Coxeter group $W_S$ is irreducible and large and that $V_{\alpha}=\{0\}$ and $V_v=V$, \ie $\rho(W_S)$ is reduced and dual-reduced.

In this case, by Proposition~\ref{prop:maximal}, the Tits--Vinberg domain $\OhmVin$ is properly convex and contains all other $\rho(W_S)$-invariant open convex subsets of $\PP(V)$.
Let $\Omega_{\min}$ be the smallest nonempty $\rho(W_S)$-invariant convex open subset of $\OhmVin$, as given by Proposition~\ref{prop:benoist-irred-Gamma}.
By Theorem~\ref{thm:minimal_convex} and Lemma~\ref{lem:Sigma-flat-in-OhmVin}, the set $\Omega_{\min} = \rho(W_S)\cdot\Sigma^{\flat}$ is the interior of $\rho(W_S)\cdot\Sigma$, where $\Sigma^{\flat}\subset\Sigma\subset\PP(V)$ are the projections of the sets $\widetilde{\Sigma}^{\flat}\subset\widetilde{\Sigma}\subset V$ of \eqref{eqn:Sigma-tilde}--\eqref{eqn:Sigma-tilde-flat}, obtained by pruning $\widetilde{\Delta}$. 
Consider the following nonempty $\rho(W_S)$-invariant closed convex subset of~$\OhmVin$:
\begin{equation} \label{eqn:C-cox}
\C := \overline{\Omega_{\min}} \cap \OhmVin.
\end{equation}

\begin{proposition} \label{prop:sigma_is_inside}
Let $W_S$ be a large irreducible Coxeter group and $\rho : W_S\to\GL(V)$ a representation of $W_S$ as a reflection group such that $\rho(W_S)$ is reduced and dual-reduced.
Then the following are equivalent:
\begin{enumerate}
  \item\label{item:sigma_bounded} $\Sigma \subset \OhmVin$;
  \item\label{item:C_has_compact_quotient} the set $\C$ of \eqref{eqn:C-cox} has compact quotient by $\rho(W_S)$;
  \item\label{item:naive_cc} $\rho(W_S)$ is naively convex compact in $\mathbb{P}(V)$;
  \item\label{item:a_and_3} the following conditions both hold:
  \begin{enumerate}
    \item[\ref{item:no-Z2}] there do not exist disjoint subsets $S',S''$ of~$S$ such that $W_{S'}$ and $W_{S''}$ are both infinite and commute;
    \item[\ref{item:not-zero-type}] for any irreducible standard subgroup $W_{S'}$ of~$W_S$ with $\emptyset\neq S'\subset S$, the Cartan submatrix $\Cart_{S'}:=(\Cart_{i,j})_{s_i,s_j\in S'}$ is \emph{not} of zero type.
  \end{enumerate}
\end{enumerate}
\end{proposition}

The implication \eqref{item:a_and_3}~$\Rightarrow$\eqref{item:sigma_bounded} of Proposition~\ref{prop:sigma_is_inside} was established by the first three authors \cite[Lem.\,8.9]{dgk-ccHpq} when $W_S$ is right-angled and $\rho$ preserves a symmetric bilinear form, and by the last two authors \cite[Lem.\,4.8]{lm18} in general.

\begin{proof}
We first check the implication \eqref{item:sigma_bounded}~$\Rightarrow$~\eqref{item:C_has_compact_quotient}.
Using the notation of Section~\ref{subsec:refl-dual}, we have $\overline{\Omega_{\min}} \subset \mathcal{D}^{\overline{*}}$ by Lemma~\ref{lem:OmegaMin}, hence $\C\cap \Delta \subset \mathcal{D}^{\overline{*}}\cap \Delta=\Sigma$.
Therefore, taking $\rho(W_S)$-orbits, the set $\C$ is contained in $\C_{\Sigma} := \bigcup_{\gamma\in W_S} \rho(\gamma)\cdot \Sigma$.
The set $\Sigma$ is compact.
If $\Sigma\subset\OhmVin$, then the action of $\rho(W_S)$ on $\C_{\Sigma}\subset\OhmVin$ is properly discontinuous and cocompact, and the same holds for~$\C$ since $\C$ is a closed subset of~$\C_{\Sigma}$.

We now check \eqref{item:C_has_compact_quotient}~$\Rightarrow$~\eqref{item:sigma_bounded}.
By Fact~\ref{fact:Delta-flat}, the convex set $\Delta^\flat \cap \C$ is homeomorphic to the quotient $\rho(W_S) \backslash \C$.
Therefore, assuming~\eqref{item:C_has_compact_quotient}, it is a compact subset of $\mathbb{P}(V)$.
Since $\Sigma^\flat \subset \Delta^\flat \cap \overline{\Omega_{\min}} = \Delta^\flat \cap \C$, the closure $\overline{\Sigma^\flat}$ of $\Sigma^\flat$ in $\PP(V)$ is included in $\Delta^\flat \cap \overline{\Omega_{\min}}$, hence also in $\OhmVin$.
But $\Sigma = \overline{\Sigma^\flat}$, hence $\Sigma \subset \OhmVin$.

The implication \eqref{item:C_has_compact_quotient}~$\Rightarrow$~\eqref{item:naive_cc} holds by Definition~\ref{def:cc-group} of naive convex cocompactness.
The implication \eqref{item:naive_cc}~$\Rightarrow$~\eqref{item:C_has_compact_quotient} holds because any nonempty $\rho(W_S)$-invariant closed convex subset $\C'$ has to contain $\Omega_{\min}$, hence $\C$ by Proposition~\ref{prop:benoist-irred-Gamma}.
The equivalence \eqref{item:sigma_bounded}~$\Leftrightarrow$~\eqref{item:a_and_3} is contained in Proposition~\ref{prop:boundary_sigma}.
\end{proof}

\subsection{Cocompact convex sets are large enough}\label{subsec:large-enough}

We now establish the implication \ref{item:naive-cc-cox}~$\Rightarrow$~\ref{item:cc-cox} of Theorem~\ref{thm:main} in the case that the Coxeter group $W_S$ is irreducible and large and $\rho(W_S)$ is reduced and dual-reduced.
We use the following terminology.

\begin{definition}[{\cite[Def.\,1.19]{dgk-proj-cc}}]
Let $\C$ be a properly convex subset of $\PP(V)$ with nonempty interior.
The \emph{ideal boundary} of~$\C$ is $\partiali \C:=\overline{\C}\smallsetminus \C$.
The \emph{nonideal boundary} of~$\C$ is $\partialn \C:=\C\smallsetminus \operatorname{Int}\,(\C)$.
\end{definition}

The implication \ref{item:naive-cc-cox}~$\Rightarrow$~\ref{item:cc-cox} is contained in the following.

\begin{proposition} \label{prop:naive-cc-implies-cc}
Let $W_S$ be a large irreducible Coxeter group and $\rho : W_S\to\GL(V)$ a representation of $W_S$ as a reflection group such that $\rho(W_S)$ is reduced and dual-reduced.
Suppose there is a $\rho(W_S)$-invariant properly convex open subset $\Omega$ of $\PP(V)$ and a nonempty $\rho(W_S)$-invariant closed convex subset $\C$ of~$\Omega$ which has compact quotient by $\rho(W_S)$.
Then $\Lambdao_{\Omega}(\rho(W_S)) \subset \partiali \C$, hence $\rho(W_S)$ is convex cocompact in $\PP(V)$.
\end{proposition}

\begin{proof}
By Proposition~\ref{prop:maximal}.\eqref{item:max-3}, the set $\Omega$ is contained in the Tits--Vinberg domain~$\OhmVin$, and so we may assume $\Omega=\OhmVin$.
Suppose by contradiction that $\Lambdao_{\OhmVin}(\rho(W_S)) \not\subset \partiali \C$.
Let $\C_0$ be the convex hull of $\partiali\C$ in~$\C$ and, for $t > 0$, let $\C_t$ be the closed uniform $t$-neighborhood of $\C_0$ in $\OhmVin$ with respect to the Hilbert metric $d_{\OhmVin}$.
The set $\C_t$ is properly convex \cite[(18.9)]{bus55}; it is $\rho(W_S)$-invariant and has compact quotient by $\rho(W_S)$, for the set $\C_t$ is the union of the $\rho(W_S)$-translates of the closed uniform $t$-neighborhood of a compact fundamental domain of $\C_0$.

Since $\Lambdao_{\OhmVin}(\rho(W_S)) \not\subset \partiali \C$, there exists $y_0 \in \OhmVin$ whose $\rho(W_S)$-orbit admits an accumulation point $\zeta \in \partial \OhmVin \smallsetminus \partiali \C$; necessarily $y_0\notin\C$.
Let $s = d_{\OhmVin}(y_0, \C_0) > 0$.
Let $\mathcal{E}_s$ be the convex hull of $\partiali \C_s$ in $\OhmVin$; it is a closed $\rho(W_S)$-invariant subset of $\C_s$, hence it has compact quotient by $\rho(W_S)$.
The set $\partiali \mathcal{E}_s = \partiali \C_s$ contains $\zeta$, hence is strictly larger than $\partiali \C$.
Therefore $\mathcal{E}_s$ is strictly larger than~$\C_0$ and $\partialn \mathcal{E}_s$ contains a point not in~$\C_0$.
Since $\mathcal{E}_s$ has compact quotient by $\rho(W_S)$, there is a point $y \in \partialn \mathcal{E}_s$ achieving maximum distance $0 < t \leq s$ to~$\C_0$.
By maximality of~$t$, we have $\mathcal{E}_s \subset \C_t$.
Let $H_y$ be a hyperplane supporting $\C_t$ (and therefore also $\mathcal{E}_s$) at~$y$.
The intersection $\C' := H_y \cap \mathcal{E}_s$ is a nonempty convex set which is the convex hull of some subset of $\partiali \mathcal{E}_s$.
Observe that $\C'$ is contained in $\partialn \C_t$ and is therefore disjoint from~$\C_0$.

\begin{figure}[ht!]
\centering
\labellist
\small\hair 2pt
\pinlabel {$\mathcal{H}(r)$} at 345 450
\pinlabel {$H$} at 80 290
\pinlabel {$z$} at 290 320
\pinlabel {$x'$} at 183 297
\pinlabel {$x$} at 403 340
\pinlabel {$\C'$} at 475 400
\pinlabel {$\OhmVin$} at 407 100
\pinlabel {$\mathcal{E}_s$} at 70 200
\endlabellist
\includegraphics[scale=0.34]{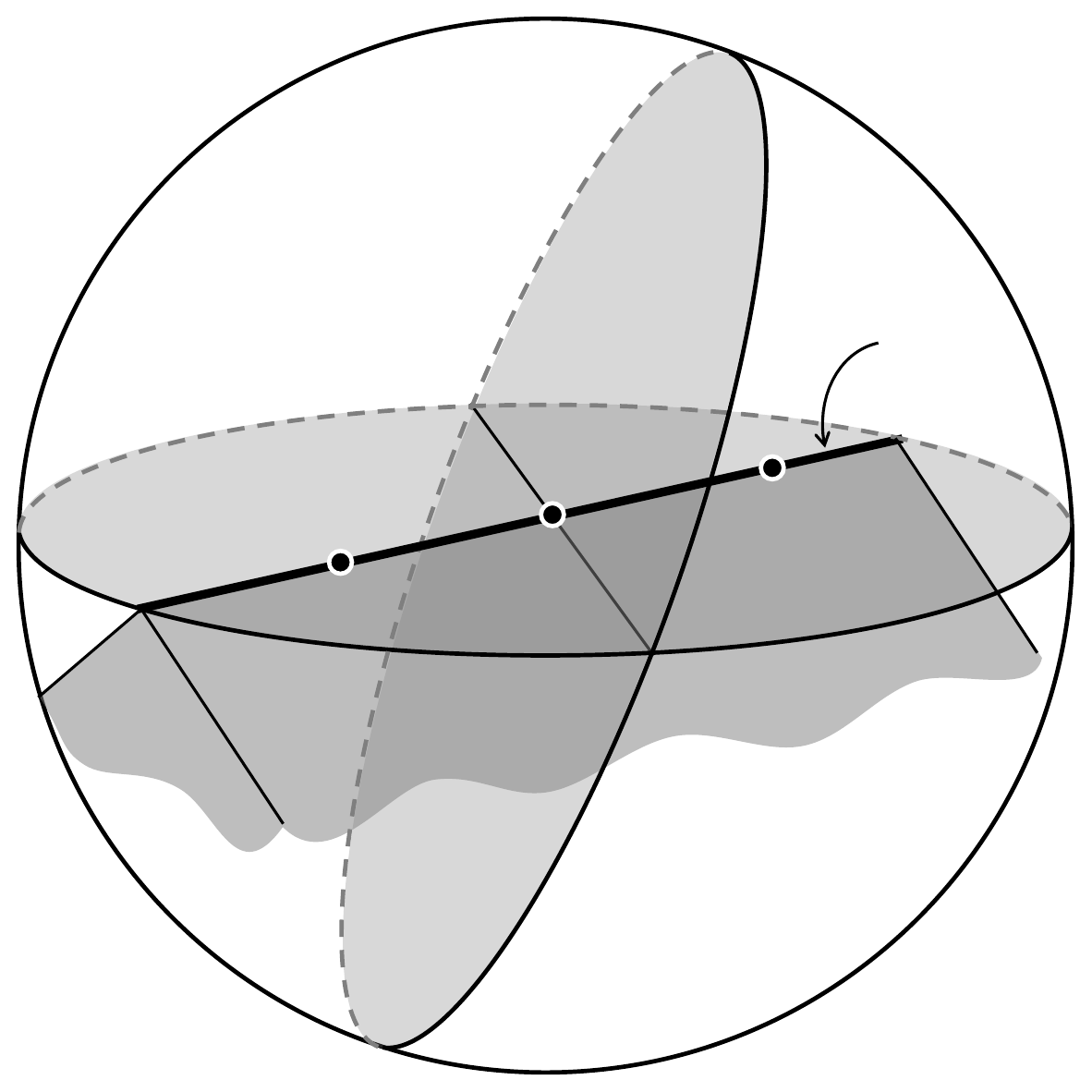}
\caption{Illustration for the proof of Proposition~\ref{prop:naive-cc-implies-cc}}
\label{fig:Hr-C'}
\end{figure}

We claim that any hyperplane $H$ supporting $\C_t$ along $\C'$ is invariant under the group $\Gamma'$ generated by all the reflections $\rho(r)$, $r\in W_S$, whose fixed hyperplane $\mathcal{H}(r)$ separates $\C'$ into two connected components. 
Indeed, consider such a reflection~$r$.
First, we note that $H \cap \mathcal{H}(r) = (\rho(r)\cdot H) \cap \mathcal{H}(r)$.
Second, we show that $\rho(r) \cdot H$ also contains~$\C'$.
Consider any point $x \in \C' \smallsetminus \mathcal{H}(r)$ and choose a second point $x' \in \C' \smallsetminus \mathcal{H}(r)$ on the opposite side of $\mathcal{H}(r)$ (see Figure~\ref{fig:Hr-C'}).
The segment $[x,x']$ crosses $\mathcal{H}(r)$ at some point $z \in (x,x')$.
Since $\rho(r)$ fixes~$z$, and $H$ is a supporting hyperplane to $\C_t \supset \rho(r)\cdot [x,x']$ at $z$, it follows that $\rho(r)\cdot [x,x'] \subset H$, in particular $\rho(r)\cdot x \in H$, hence $x \in \rho(r)\cdot H$.
Since $H$ is spanned by $H \cap \mathcal{H}(r)$ and $\C'$, we deduce that $\rho(r)\cdot H = H$.
Therefore $H$ is invariant under the group $\Gamma'$ generated by the set of such reflections, as claimed.

In particular, the convex sets $\C' = H_y \cap \mathcal{E}_s$ and $\Omega' := H_y \cap \OhmVin$ of $H_y$ are invariant under $\Gamma'$.
Moreover, since the action of $\rho(W_S)$ on $\mathcal{E}_s$ is cocompact, the intersection of $\C'$ with the tiling of $\OhmVin$ decomposes $\C'$ into compact polytopes, hence the action of $\Gamma'$ on~$\C'$ is also cocompact.
In particular, since $\C'$ is noncompact, $\Gamma'$ must be infinite.

The above construction finds a proper subspace $\PP(V'):=\PP(\mathrm{span}(\C'))$, an infinite subgroup $\Gamma'$ of $\rho(W_S)$ generated by reflections and preserving $\PP(V')$, and a nonempty closed convex subset $\C' \subset \Omega' := \PP(V') \cap \OhmVin$ such that
\begin{itemize}
  \item $\C'$ is the convex hull of a closed $\Gamma'$-invariant subset of $\partial \Omega'$, and
  \item the action of $\Gamma'$ on $\C'$ is properly discontinuous and cocompact, but
  \item $\C'$ is disjoint from $\C_0$.
\end{itemize}

To find a contradiction, consider $\PP(V')$, $\Gamma' < \rho(W_S)$, and $\C' \subset \Omega' = \PP(V') \cap \OhmVin$ satisfying the above and so that the dimension of $\PP(V')$ is minimal.
There are two cases to consider: (i) $\C' = \Omega'$ and (ii) $\C' \neq \Omega'$.
Note that $\dim(V') \geqslant 2$ (by proper discontinuity), and that $\dim(V') = 2$ implies~(i).

In case~(i), the group $\Gamma'$ acts on $\Omega'$ cocompactly.
If $\Lambda_{\Gamma'}$ denotes the proximal limit set of $\Gamma'$ in $\PP(V')$, then the convex hull $\mathrm{Conv}(\Lambda_{\Gamma'})$ of $\Lambda_{\Gamma'}$ in $\Omega'$ is equal to $\Omega' = \C'$ by Fact~\ref{fact:vey}.
Therefore $\C' = \mathrm{Conv}(\Lambda_{\Gamma'}) \subset \mathrm{Conv}(\Lambda_{\rho(W_S)}) \subset \C_0$, where $\mathrm{Conv}(\Lambda_{\rho(W_S)})$ is the convex hull of $\Lambda_{\rho(W_S)}$ in $\OhmVin$, contradicting the fact that $\C'$ is disjoint from $\C_0$.

In case~(ii), there exists a point $x \in \partialn \C' \subset \Omega'$ and a supporting hyperplane $H_x$ of $\C'$ at $x$ (in $\PP(V')$).
Then $\PP(V'') := \PP(V') \cap H_x$, $\C'' := H_x \cap \C'$, and $\Omega'' := H_x \cap \Omega'$ are invariant under the group $\Gamma''$ generated by all the reflections $\rho(r)$ whose fixed hyperplane $\mathcal{H}(r)$ separates $\C''$ into two connected components and the $\Gamma''$-action on $\C''$ is cocompact.
However, $\PP(V'') = \PP(V') \cap H_x$ has dimension one less than that of $\PP(V')$, contradicting minimality.
\end{proof}

\subsection{Proof of Theorem~\ref{thm:main}}

Since $W_S$ is infinite and satisfies \ref{item:no-Z2}, we can write $W_S = W_{S'} \times W_{S''}$ for $S = S'\sqcup S''$ where $W_{S'}$ is infinite and irreducible and $W_{S''}$ is finite (possibly trivial).
By Proposition~\ref{prop:finite-index} and Fact~\ref{fact:types}.\eqref{item:spherical_case}, it is sufficient to check Theorem~\ref{thm:main} for $W_{S'}$.
So we now assume that $S = S'$, \ie the infinite Coxeter group $W_S$ is irreducible.

When $W_S$ is affine, Theorem~\ref{thm:main} is contained in Proposition~\ref{prop:affine}.
We now assume that $W_S$ is large, satisfying \ref{item:no-Z2} and~\ref{item:aff-only-Ak}.
Let $\rho_v^\alpha : W_S\to\GL(V^\alpha_v)$ be the representation of $W_S$ as a reflection group induced by~$\rho$ as in Remarks~\ref{rem:notation_rhos} and~\ref{rem:rho-v-alpha-refl-group}; the group $\rho_v^{\alpha}(W_S)$ is reduced and dual-reduced.
The equivalences \ref{item:naive-cc-cox}~$\Leftrightarrow$~\ref{item:cc-cox}~$\Leftrightarrow$~\ref{item:not-zero-type}~$\Leftrightarrow$~\ref{item:not-zero-type-explicit} then follow from Lemma~\ref{lem:V-alpha-v}, Propositions ~\ref{prop:sigma_is_inside} and~\ref{prop:naive-cc-implies-cc}, and Fact~\ref{fact:types}, as in the following diagram:

\begin{center}
\begin{tikzcd}
\rho(W_S) \ref{item:naive-cc-cox}\text{in } \PP(V) 
	\arrow[Leftarrow, d,"\text{by definition}" left] 
	\arrow[rrr, Rightarrow,"\text{Lemma~\ref{lem:V-alpha-v}.\eqref{item:V-alpha-v-naive-cc}}" below] &  &  &
\rho_v^\alpha(W_S) \ref{item:naive-cc-cox}\text{in } \PP(V_v^\alpha) 
	\arrow[Rightarrow, d,  "\text{Proposition~\ref{prop:naive-cc-implies-cc}}" left] 
	\arrow[Leftrightarrow, rrr,  "\text{Proposition~\ref{prop:sigma_is_inside}}" below]  &  &  &
\ref{item:not-zero-type} 	
	\arrow[Leftrightarrow, d,  "\text{Fact~\ref{fact:types}}" left]
	\\
\rho (W_S) \ref{item:cc-cox} \text{in } \PP(V) 
	\arrow[Leftrightarrow, rrr,  "\text{Lemma~\ref{lem:V-alpha-v}.\eqref{item:V-alpha-v-cc}}" above] & & &  
\rho_v^\alpha(W_S) \ref{item:cc-cox} \text{in }\PP(V_v^\alpha) & & &
\ref{item:not-zero-type-explicit}.
\end{tikzcd}
\end{center}

\section{Proof of Theorem~\ref{thm:naive-cc-exists} and consequences} \label{sec:mainproof-easy}

We now prove~Theorem~\ref{thm:naive-cc-exists} and its consequences Corollaries \ref{cor:rel-hyp}, \ref{cor:cc-standard-subgroup}, \ref{cor:main2}, and~\ref{cor:main-racg}.

\subsection{Naive convex cocompactness implies only one infinite irreducible factor} \label{subsec:naive-cc->not-IC}

We first make the following observation.

\begin{proposition} \label{prop:W-prod-not-cc}
Let $W_S$ be an infinite Coxeter group and $\rho : W_S\to\GL(V)$ a representation of $W_S$ as a reflection group.
If $\rho(W_S)$ is naively convex cocompact in $\PP(V)$, then the Coxeter group $W_S$ has only one infinite irreducible factor.
\end{proposition}

\begin{proof}
Suppose that $\rho(W_S)$ is naively convex cocompact in $\PP(V)$: there exist a $\rho(W_S)$-invariant properly convex open subset $\Omega$ of $\PP (V)$ and a nonempty $\rho(W_S)$-invariant closed convex subset $\C$ of~$\Omega$ such that $\rho(W_S) \backslash \C$ is compact.

For each infinite irreducible factor $W_{S_{\ell}}$ of~$W_S$, the corresponding Cartan submatrix is of negative type, otherwise $\rho(W_{S_\ell})$ would contain a unipotent element by Fact~\ref{fact:types}.\eqref{item:large_case} and Proposition~\ref{prop:affine}, contradicting Proposition~\ref{prop:cc-no-unipotent}.
By Lemma~\ref{lem:V-alpha-v}.\eqref{item:V-alpha-v-naive-cc}, up to replacing $\rho$ by $\rho^{\alpha}$, we may assume $V_{\alpha} := \bigcap_{s\in S} \mathrm{Ker}(\alpha_s) = \{0\}$.

Suppose by contradiction that $W_S$ has more than one infinite irreducible factors: we can write $W_S = W_{S'}\times W_{S''}$ for $S=S'\sqcup S''$ where $W_{S'}$ and $W_{S''}$ are both infinite and $W_{S'}$ is irreducible.
Note that
$$\mathrm{span} \big\{ v_i \,|\, s_i \in S'\big\} \cap \bigcap_{s_i\in S'} \mathrm{Ker}(\alpha_i) = \{0\}.$$
Indeed, $\mathrm{span}\{ v_i \,|\, s_i \in S'\} \subset \bigcap_{s_j\in S''} \mathrm{Ker}(\alpha_j)$ and $V_{\alpha} = \{0\}$ by assumption.
Thus $(\rho|_{W_{S'}})_v$ is reduced and dual-reduced.
It is still a representation as a reflection group by Remark~\ref{rem:rho-v-alpha-refl-group}.

Since $W_{S'}$ and $W_{S''}$ are infinite, Proposition~\ref{prop:nonirred_ohmweird} shows that $\Omega$ is included in one of the finitely many Tits--Vinberg domains $\OhmVin$.
Note that $\OhmVin \subset \Omega_{\scriptscriptstyle\mathrm{TV}, S'}$ by Fact~\ref{fact:Vinberg-cones}.
By Corollary~\ref{cor:C-meets-P(Vv)}.\eqref{item:C-meets-P(Vv)}, the set $\C \cap \PP(\mathrm{span}\{ v_i \,|\, s_i \in S'\})$ has nonempty interior in $\PP(\mathrm{span}\{ v_i \,|\, s_i \in S'\})$.
By Theorem~\ref{thm:minimal_convex}, the closure $\overline{\C}$ of $\C$ in $\PP(V)$ contains
$$\Sigma_{S'} := \PP\Big(\Big\{ x = \sum_{s_i\in S'} t_i v_i \in V ~\Big|~ \alpha_i(x) \leq 0 \ \text{ and }\ t_i \geq 0 \quad\forall s_i\in S' \Big\} \Big).$$
Let $t' = (t_i)_{s_i \in S'}$ be a Perron--Frobenius eigenvector of $2\,\mathrm{Id} - \Cart_{S'}$ and $x' := \sum_{s_i \in S'} t_i v_i$, so that $[x'] \in \Sigma_{S'}$. 
Then, arguing as in the proof of Proposition~\ref{prop:boundary_sigma}.\eqref{item:boundary_of_sigma2}, we have that $[x'] \in \Sigma_{S'} \cap \partial \OhmVin$. Hence, $[x'] \in \partiali\C \cap \Sigma_{S'} \subset \partiali\C \cap \Delta$, which shows that $\C \cap \Delta$ is noncompact.
Since $\C \cap \Delta$ is homeomorphic to the quotient $\rho(W_S)\backslash \C$ (Fact~\ref{fact:Delta-flat}), we obtain a contradiction.
\end{proof}

\subsection{Proof of Theorem~\ref{thm:naive-cc-exists}.}

Suppose there exist a finite-dimensional real vector space $V$ and a representation $\rho\in\Hom^{\mathrm{ref}}(W_S,\GL(V))$ such that $\rho(W_S)$ is naively convex cocompact in $\PP(V)$.
By Proposition~\ref{prop:W-prod-not-cc}, we can write $W_S = W_{S'} \times W_{S''}$ for $S = S'\sqcup S''$ where $W_{S'}$ is infinite and irreducible and $W_{S''}$ is finite (possibly trivial).
It is sufficient to check conditions \ref{item:no-Z2} and~\ref{item:aff-only-Ak} for~$W_{S'}$, so we now assume $S = S'$, \ie $W_S$ is irreducible.
By Lemma~\ref{lem:V-alpha-v}.\eqref{item:V-alpha-v-naive-cc}, we may assume that $\rho(W_S)$ is reduced and dual-reduced.
By Propositions~\ref{prop:affine} and~\ref{prop:sigma_is_inside}, condition~\ref{item:no-Z2} holds.
Moreover, condition~\ref{item:aff-only-Ak} also holds because for any affine irreducible Coxeter group which is \emph{not} of type~$\widetilde{A}_k$, the corresponding Cartan matrix is of zero type by Fact~\ref{fact:types}.

Conversely, suppose $W_S$ is any Coxeter group which satisfies the conditions \ref{item:no-Z2} and~\ref{item:aff-only-Ak}.
Let $N=\# S$ be the number of generators of~$W_S$.
Consider the matrix $\mathrm{Cos}(W_S) = (-2\cos(\pi/m_{i,j}))_{1\leq i,j\leq N}$, and modify it into a matrix $\Cart \in \mathcal{M}_N(\RR)$ in the following way:
\begin{itemize}
  \item for each pair $(i,j)$ with $m_{i,j}=\infty$, replace the entry $-2\cos(\pi/m_{i,j}) = -2$ by some negative number $<-2$;
  \item let $P$ be the set of pairs $(i,j)$ with $i<j$ and $m_{i,j}=3$.
  For each $(i,j)\in P$, choose a number $t_{i,j}>1$, in such a way that for any disjoint subsets $P'$ and $P''$ of $P$, we have
  $$\prod_{(i,j)\in P'} t_{i,j} \times \prod_{(i,j)\in P''} t_{i,j}^{-1}\neq 1.$$
  Then, for any pair $(i,j)\in P$, multiply the $(i,j)$-entry $-2\cos(\pi/m_{i,j}) = -1$ by $t_{i,j}$ and the $(j,i)$-entry $-2\cos(\pi/m_{j,i}) = -1$ by $t_{i,j}^{-1}$ to obtain the matrix~$\Cart$.
\end{itemize}
By construction of~$\Cart$, all Cartan submatrices of~$\Cart$ corresponding to subgroups $W_{S'}$ of type~$\widetilde{A}_k$ are nonsymmetrizable, hence $\Cart$ does not have any Cartan submatrix of zero type.
By \cite[Cor.\,1]{vin71}, there exists a representation $\rho\in\Hom^{\mathrm{ref}}(W_S,\GL(\RR^N))$ with Cartan matrix~$\Cart$.
By Theorem~\ref{thm:main}, the group $\rho(W_S)$ is convex cocompact in $\PP(\RR^N)$.
This completes the proof of Theorem~\ref{thm:naive-cc-exists}.

\subsection{Consequences of Theorems~\ref{thm:naive-cc-exists} and~\ref{thm:main}}

We now prove Corollaries \ref{cor:rel-hyp}, \ref{cor:cc-standard-subgroup}, \ref{cor:main2}, and~\ref{cor:main-racg}.

\begin{proof}[Proof of Corollary~\ref{cor:rel-hyp}]
We first check the implication \eqref{item:naive-cc-exists-again}~$\Rightarrow$~\eqref{item:relatively-hyperbolic}.
Suppose that there exist $V$ and $\rho\in\Hom^{\mathrm{ref}}(W_S,\GL(V))$ such that $\rho(W_S)$ is naively convex cocompact in $\PP(V)$.
By Theorem~\ref{thm:naive-cc-exists}, the Coxeter group $W_S$ satisfies conditions \ref{item:no-Z2} and~\ref{item:aff-only-Ak}.

The relative hyperbolicity of $W_S$ follows from a theorem of Caprace \cite[Cor.\,D]{cap09}: the Coxeter group $W_S$ is relatively hyperbolic with respect to a collection $\mathcal{P}$ of virtually abelian subgroups of rank at least $2$ if and only if for any disjoint subsets $S', S''$ of~$S$ with $W_{S'}$ and $W_{S''}$ both infinite and commuting, the subgroup $W_{S' \cup S''}$ is virtually abelian.
Here this criterion is vacuously satisfied due to~\ref{item:no-Z2}. 

By~\cite[Th.\,B]{cap09} and \cite{cap15}, every $P \in \mathcal{P}$ may be chosen to be one of finitely many standard subgroups $W_{T_i}$, $1\leq i \leq\ell$.
It only remains to explain why the $W_{T_i}$ are of the form claimed.

Fix $1\leq i \leq\ell$.
The Coxeter group $W_{T_i}$ is a product of irreducible standard Coxeter groups.
Since $W_{T_i}$ is virtually abelian, each irreducible factor is either affine (hence infinite) or spherical (hence finite), see Section~\ref{subsec:Coxeter-type}.
Condition~\ref{item:no-Z2} implies that $W_{T_i}$ has exactly one affine irreducible factor $W_U$, which must be of type~$\widetilde{A}_k$ for some $k\geq 2$ by condition~\ref{item:aff-only-Ak}.
We have $W_{T_i} = W_U \times W_{U^\perp}$ where $W_{U^\perp}$ is the standard subgroup of~$W_S$ generated by $U^\perp := \{ s \in S \, |\, m_{u,s}=2\ \forall u \in U\}$, and as seen above $W_{U^\perp}$ is a product of spherical irreducible Coxeter subgroups.

We now check the converse implication \eqref{item:relatively-hyperbolic}~$\Rightarrow$~\eqref{item:naive-cc-exists-again}.
For this we use the following result of Caprace, \cite[Th.\,A]{cap09} and \cite{cap15}: if $W_S$ is relatively hyperbolic with respect to a collection $\mathcal{P}$ of standard subgroups $W_{T}$, and if we set $\mathcal{T} := \{ T \subset S \mid W_T \in \mathcal{P} \}$, then:
\begin{itemize}
\item for any disjoint subsets $S'$, $S''$ of~$S$ such that $W_S'$ and $W_{S''}$ are both infinite and commute, there exists $T \in \mathcal{T}$ such that $S' \cup S''\subset T$;
\item for any subset $S'$ of~$S$ with $\# S' \geq 3$ such that $W_{S'}$ is irreducible and affine, there exists $T \in \mathcal{T}$ such that $S' \subset T$. 
\end{itemize}
This result implies that if $W_S$ is relatively hyperbolic with respect to a collection of virtually abelian subgroups which are the standard subgroups of~$W_S$ of the form $W_U \times W_{U^\perp}$ with $W_U$  of type~$\widetilde{A}_k$ for some $k\geq 2$ and $W_{U^\perp}$ finite, then conditions \ref{item:no-Z2} and~\ref{item:aff-only-Ak} are satisfied.
Theorem~\ref{thm:naive-cc-exists} then implies that there exist $V$ and $\rho\in\Hom^{\mathrm{ref}}(W_S,\GL(V))$ such that $\rho(W_S)$ is convex cocompact in $\PP(V)$.
\end{proof}

\begin{proof}[Proof of Corollary~\ref{cor:cc-standard-subgroup}]
Suppose $\rho(W_S)$ is convex cocompact in $\PP(V)$ and let $W_{S'}$ be an infinite standard subgroup  of $W_S$.
By Theorem~\ref{thm:main}, for any $\emptyset\neq S''\subset S$, the Cartan submatrix $\Cart_{S''}:=(\Cart_{i,j})_{s_i,s_j\in S''}$ is \emph{not} of zero type.
In particular, this holds for any $\emptyset\neq S''\subset S'$.
Therefore, if $W_{S'}$ is an irreducible Coxeter group, then Theorem~\ref{thm:main} yields that $\rho(W_{S'})$ is convex cocompact in $\PP(V)$.

In general, $W_{S'}$ may not be irreducible, but since $W_S$ satisfies condition~\ref{item:no-Z2} of Theorem~\ref{thm:naive-cc-exists} we know that $W_{S'}$ has a finite-index subgroup $W_{S''}$ which is standard and irreducible.
Then $\rho(W_{S''})$ is convex cocompact in $\PP(V)$, and so $\rho(W_{S'})$ is convex cocompact in $\PP(V)$ (see Proposition \ref{prop:finite-index}).
\end{proof}

\begin{proof}[Proof of Corollary~\ref{cor:main2}]
If $\rho(W_S)$ is strongly convex cocompact in $\PP(V)$, then $W_S$ is word hyperbolic by Theorem~\ref{prop:cc-hyp-strong-cc}, and $\Cart_{i,j} \Cart_{j,i} > 4$ for all $i\neq j$ with $m_{i,j} = \infty$ by Theorem~\ref{thm:main}.

Conversely, suppose $W_S$ is word hyperbolic and $\Cart_{i,j}\Cart_{j,i} > 4$ for all $i\neq j$ with $m_{i,j}=\infty$.
Since $W_S$ is word hyperbolic, it does not contain any subgroup isomorphic to~$\ZZ^2$, and so conditions \ref{item:no-Z2} and~\ref{item:aff-only-Ak} of Theorem~\ref{thm:naive-cc-exists} both hold (the latter vacuously).
Condition~\ref{item:not-zero-type-explicit} of Theorem~\ref{thm:main} is also satisfied because for any $i\neq j$ with $m_{i,j} = \infty$, the Cartan submatrix $\big(\begin{smallmatrix} 2 & \Cart_{i,j}\\ \Cart_{j,i} & 2\end{smallmatrix}\big)$ is \emph{not} of zero type if and only if $\Cart_{i,j} \Cart_{j,i} > 4$.
Thus $\rho(W_S)$ is convex cocompact in $\PP(V)$ by Theorem~\ref{thm:main}, and strongly convex cocompact in $\PP(V)$ by Theorem~\ref{prop:cc-hyp-strong-cc}.
\end{proof}

\begin{proof}[Proof of Corollary~\ref{cor:main-racg}]
Corollary~\ref{cor:main2} gives the equivalence \ref{item:word-hyp-+}~$\Leftrightarrow$~\ref{item:strong-cc-cox}, and by definition we have \ref{item:strong-cc-cox}~$\Rightarrow$~\ref{item:cc-cox}~$\Rightarrow$~\ref{item:naive-cc-cox}.
We now check the implication\linebreak \ref{item:naive-cc-cox}~$\Rightarrow$~\ref{item:strong-cc-cox}.
By Theorem~\ref{thm:naive-cc-exists}, if $\rho$ satisfies~\ref{item:naive-cc-cox}, then $W_S$ must satisfy \ref{item:no-Z2} and~\ref{item:aff-only-Ak}.
In this case, $\rho$ actually satisfies~\ref{item:cc-cox} by Theorem~\ref{thm:main}.
By Theorem~\ref{prop:cc-hyp-strong-cc}, in order to get~\ref{item:strong-cc-cox} it is sufficient to check that $W_S$ is word hyperbolic. 
By Moussong's criterion \cite{mou87}, recalled in Remark~\ref{rem:moussong}, we only need to check that $W_S$ satisfies~\ref{item:no-Z2} (already done) and~\ref{item:non_aff}.
But \ref{item:non_aff} is the conjunction of~\ref{item:aff-only-Ak} and of our assumption that $W_S$ contains \emph{no} standard subgroup of type $\widetilde{A}_k$ for $k\geq 2$.
\end{proof}

\begin{remark}
In \cite[\S\,8]{dgk-ccHpq}, the first three authors proved Corollary~\ref{cor:main-racg} in the special case that $n=N$ and that the Cartan matrix $\Cart = (\Cart_{i,j})_{1\leq i,j\leq N}$ for $\rho$ is symmetric and defines a nondegenerate quadratic form on~$V$ as in Remark~\ref{rem:Tits-geom}. 
This approach was used by the last two authors \cite{lm18} to construct the first examples of discrete subgroups of $\OO(p,2)$ which are strongly convex cocompact in $\PP(\RR^{p+2})$ and whose (proximal) limit set is homeomorphic to the $(p-1)$-dimensional sphere, but which are not quasi-isometric to lattices of $\OO(p,1)$ as abstract groups.
\end{remark}

\section{The deformation space of convex cocompact representations} \label{sec:deformation}

We now prove Corollaries~\ref{coro:nra_cc-int-refl},~\ref{coro:cc-int-refl}, and~\ref{cor:hyp-CG-Anosov}.

\subsection{Convex cocompactness and the interior of $\chi^{\mathrm{ref}}(W_S,\mathrm{GL}(V))$} \label{subsec:cc_character}

\begin{proof}[Proof of Corollary~\ref{coro:nra_cc-int-refl}]
By Remark~\ref{rem:Cartan-matrix}, the space of characters of $W_S$ defined by data $(\alpha,v)$ satisfying~\eqref{eqn:conn-comp-ref} is an open subset of $\chi(W_S, \GL(V))$ containing $\chi^{\mathrm{ref}}(W_S,\mathrm{GL}(V))$.
By Fact~\ref{fact:parameterization}, the map assigning the conjugacy class of Cartan matrix to a conjugacy class of semisimple representations is a homeomorphism from the space consisting of representations defined by data $(\alpha,v)$ satisfying~\eqref{eqn:conn-comp-ref} to the space of $N \times N$ matrices of rank at most $\dim(V)$ that are weakly compatible with~$W_S$ (Definition~\ref{def:weak-compat}), considered up to conjugation by positive diagonal matrices.

The reflection characters $\chi^{\mathrm{ref}}(W_S, \GL(V))$ correspond to the subset defined by~\eqref{eqn:refl-group}, namely that $\Cart_{i,j}\Cart_{i,j} \geq 4$ whenever $m_{i,j} = \infty$, and \eqref{eqn:not_empt_int}, namely that the cone $\widetilde{\Delta} =\linebreak \{ x \in\nolinebreak V \mid \alpha_i(x) \leq 0\ \forall 1\leq i\leq N\}$ has nonempty interior. Note that the validity of \eqref{eqn:not_empt_int} is almost automatic since $W_S$ is word hyperbolic.
Indeed, if $N \geqslant 3$ then $W_S$ is large, hence the Cartan matrix $\Cart$ is of negative type by Fact~\ref{fact:types}, and so \eqref{eqn:not_empt_int} holds as explained in Remark~\ref{rem:neg-type-interpret}.\eqref{item:neg-type-interpret}.
If $N=2$, then \eqref{eqn:not_empt_int} may fail only when $\Cart$ is of zero type, which is equivalent to $\Cart_{1,2}\Cart_{2,1} = 4$. 

In any case, the reflection characters which are convex cocompact in $\mathbb{P}(V)$ are an open subset of $\chi(W_S, \GL(V))$ which, by Corollary~\ref{cor:main2} and Proposition~\ref{prop:cc-hyp-strong-cc}, corresponds to the subset of Cartan matrices compatible with $W_S$ defined by the strict inequalities $\Cart_{i,j} \Cart_{j,i} >4$ whenever $m_{i,j}= \infty$.
Let us now check that they are precisely the interior of $\chi^{\mathrm{ref}}(W_S, \GL(V))$ if $\dim(V) \geq N$.

Consider a semisimple representation $\rho: W_S \to \GL(V)$ of~$W_S$ as a reflection group which is not convex cocompact in $\PP(V)$.
Then the associated Cartan matrix $\Cart$ satisfies $\Cart_{i,j} \Cart_{j,i} = 4$ for some $i,j$ with $m_{i,j} = \infty$. 
Deforming the entry $\Cart_{i,j}$ to become less negative gives a matrix $\Cart$ remaining within the space of weakly compatible matrices, for which $\Cart_{i,j} \Cart_{j,i}$ becomes smaller than $4$.
However, under the assumption $\dim(V) \geq N$, such a small deformation of the Cartan matrix corresponds to a small deformation in $\chi(W_S, \GL(V))$ which is outside of $\chi^{\mathrm{ref}}(W_S, \GL(V))$. 
This shows that the character of $\rho$ is not in the interior of $\chi^{\mathrm{ref}}(W_S, \GL(V))$. 
\end{proof}

\subsection{Convex cocompactness and the interior of $\Hom^{\mathrm{ref}}(W_S,\mathrm{GL}(V))$} \label{subsec:cc-int}

As discussed in the introduction, it is a much more subtle problem to determine the interior of the space $\Hom^{\mathrm{ref}}(W_S,\mathrm{GL}(V))$ which includes many nonsemisimple representations.

\begin{proof}[Proof of Corollary~\ref{coro:cc-int-refl}]
Since convex cocompactness in $\PP(V)$ is an open condition \cite[Th.\,1.16.(D)]{dgk-proj-cc}, the set of representations $\rho \in \Hom^{\mathrm{ref}}(W_S,\GL(V))$ for which $\rho(W_S)$ is convex cocompact in $\PP(V)$ is included in the interior of $\Hom^{\mathrm{ref}}(W_S,\GL(V))$ in\linebreak $\Hom(W_S,\GL(V))$.
Let us prove the reverse inclusion.

Suppose $\rho \in \Hom^{\mathrm{ref}}(W_S,\GL(V))$ has Cartan matrix $\Cart$ satisfying $\Cart_{i,j}\Cart_{j,i}=4$ for some pair $(i,j)$ such that $m_{i,j} = \infty$, and let us find conditions for the existence of a small deformation of $\rho$ outside $\Hom^{\mathrm{ref}}(W_S,\mathrm{GL}(V))$.
The pair $(i,j)$ being fixed, define
\begin{equation} \label{eqn:4subspaces}
\begin{array}{ll}
V_i := \underset{m_{i,k}<\infty}{\mathrm{span}} (v_k) \subset V, & 
W_i^\star := \mathrm{Ann} (V_{i}) = \underset{m_{i,k}<\infty}\bigcap \mathrm{Ann}(v_k) \subset V^*, \\
V^\star_j := \underset{m_{k,j}<\infty}{\mathrm{span}}(\alpha_k) \subset V^*, &
W_j := \mathrm{Ann}(V^\star_{j}) = \underset{m_{k,j}<\infty}\bigcap \mathrm{Ker}(\alpha_k) \subset V.
\end{array}
\end{equation}
If there exists $(\beta,w) \in W_i^\star\times W_j$ such that $\beta(w) = 1$, then for any $t\in\RR$ the elements $\alpha_i^t := \alpha_i + t\beta$ and $v_j^t := v_j + t w$ satisfy $\alpha_i^t(v_k) = \alpha_i(v_k)$ for all $k$ with $m_{i,k} \neq \infty$, and $\alpha_k(v_j^t) = \alpha_k(v_j)$ for all $k$ with $m_{k,j} \neq \infty$; but $\alpha_i^t(v_j^t) = \alpha_i(v_j) + t(\lambda + t)$ where $\lambda=\alpha_i(w)+\beta(v_j)$. 
For sufficiently small $t \neq 0$ such that $\lambda t \geqslant 0$, we have $\alpha_i^t(v_j^t) > \alpha_i(v_j)$.
Replacing $(\alpha_i,v_j) $ with $(\alpha_i^t,v_j^t)$ thus yields a representation $\rho_t$ whose Cartan matrix $\Cart^t$ satisfies $\Cart^t_{i,j} = \alpha_i^t(v_j^t) > \Cart_{i,j}$ and $\Cart^t_{j,i} = \Cart_{j,i}$,~hence $\Cart^t_{i,j}\Cart^t_{j,i} < \Cart_{i,j} \Cart_{j,i}=4$.
Thus $\rho_t \notin \Hom^{\mathrm{ref}}(W_S, \GL(V))$ for such~$t$, but $\rho_t \underset{t\rightarrow 0}\rightarrow \rho$.

Assuming that $n\geqslant 2N-2$, or that $W_S$ is right-angled and $n\geqslant N$, we will show that there must exist a pair $(\beta,w) \in W_i^\star\times W_j$ such that $\beta(w) = 1$.
Indeed, suppose not.
This means that 
\begin{equation} \label{eqn:subspace}
W_i^\star \subset \mathrm{Ann}(W_j) = V_j^\star .
\end{equation}

Firstly, we assume that $n \geqslant 2N-2$.
We have $\dim(V_j^\star) \leqslant N-1$ and $\dim(W_i^\star) \geqslant\linebreak n-(N-1) \geqslant N-1$.
Note also that $\alpha_j \in V_j^\star$ and $\alpha_j \notin W_i^\star$ because $\Cart_{j,i} \neq 0$.
Therefore $W_i^\star$ is a strict subspace of $V_j^\star$, which implies $\dim W_i^\star \leqslant N-2$: contradiction.

Secondly, we assume that $W$ is right-angled and $n \geqslant N$.
Let $R_\ell = \{k \,|\,   m_{\ell, k}  = 2 \}$ and $R_{i,j} = R_i\cap R_j$.
Let $r_\ell = \# R_{\ell}$ and $r_{i,j} = \# R_{i,j}$.
Then
\begin{equation} \label{eqn:dimensions}
\left \{ \begin{array}{rcl}
\dim V^\star_j & \leq & 1 + r_j ,\\
\dim W_i^\star & \geq & n-1- r_i ~\text{ (since $\dim V_i \leq 1 + r_i$). }
\end{array} \right .
\end{equation}
By the pigeonhole principle, since the sets $R_i$ and $R_j$ are disjoint from $\{i,j\}$ in $\{ 1,\dots,N\}$, we have
\begin{equation} \label{eqn:rij}
r_{i,j} \geq r_i + r_j - (N-2).
\end{equation}
Since $W$ is word hyperbolic, for any $k,\ell \in R_{i,j}$ the generators $s_k, s_\ell$ must commute (otherwise $s_k s_\ell$ and $s_i s_j$ generate a copy of $\ZZ^2$).
Hence, the pairing $V^\star_j \times V_i \to \RR$ has rank $> r_{i,j}$, because its restriction to indices $(R_{i,j}\sqcup \{j\})\times(R_{i,j}\sqcup \{i\})$ has matrix $\mathrm{Diag}(2,\dots, 2, \Cart_{j,i})$.
But the pairing $W_i^\star\times V_i \rightarrow \RR$ is zero by \eqref{eqn:4subspaces}, and so the inclusion \eqref{eqn:subspace} forces $\dim V_j^\star -\dim W_i^\star > r_{i,j}$.
Using \eqref{eqn:dimensions} and \eqref{eqn:rij}, we obtain $n< N$: contradiction.
\end{proof}
 
\begin{remark}
Consider the map $\underline{\Cart} : V^N \times (V^*)^N \to \mathcal{M}_N(\RR)$ sending $((v_k)_{k=1}^N,(\alpha_k)_{k=1}^N)$ to the matrix $(\alpha_i(v_j))_{i,j=1}^N$.
In the case that $v_1,\dots,v_N$ are linearly independent in~$V$ (\resp $\alpha_1,\dots,\alpha_N$ are linearly independent in~$V^*$), the entries of the matrix $\underline{\Cart}((v_k)_{k=1}^N,(\alpha_k)_{k=1}^N)$ may each be perturbed independently by perturbing the linear forms $(\alpha_k)_{k=1}^N$ (\resp by perturbing the vectors $(v_k)_{k=1}^N$); thus $\underline{\Cart}$ is an open map near $((v_k)_{k=1}^N,(\alpha_k)_{k=1}^N)$ and the conclusion of Corollary~\ref{coro:cc-int-refl} is immediate.
However, $\underline{\Cart}$ fails to be an open map near some inputs $((v_k)_{k=1}^N,(\alpha_k)_{k=1}^N)$ for which both $(v_k)_{k=1}^N$ and $(\alpha_k)_{k=1}^N$ are linearly dependent.
In some degenerate cases outside of the context of reflection groups, this may happen even for $\dim V$ as large as $2N - 2$.
This explains why the proof above of Corollary~\ref{coro:cc-int-refl} is needed.
\end{remark}
 
\begin{remark}
In the case that $n :=\dim V < N$, it is a priori possible that the conclusion of Corollary~\ref{coro:cc-int-refl} could fail. 
Specifically, we cannot rule out that some product $\Cart_{i,j}\Cart_{j,i}$ is constant equal to~$4$ on an irreducible component of $\Hom^{\mathrm{ref}}(W_S, \GL(V))$ with nonempty interior in $\Hom(W_S,\GL(V))$ nor that $\Cart_{i,j}\Cart_{j,i}$ achieves the value $4$ as a local minimum in the interior of $\Hom^{\mathrm{ref}}(W_S,\GL(V))$.
We are not aware of an example demonstrating such behavior where $W_S$ is word hyperbolic and right-angled, but there exist examples where $W_S$ is right-angled (see Example~\ref{ex:right-angled}) or word hyperbolic (see Example~\ref{ex:word-hyperbolic}).
\end{remark}

\subsection{The Anosov condition and the interior of $\Hom^{\mathrm{ref}}(W_S,\mathrm{GL}(V))$} \label{subsec:proof-Ano-cc-Coxeter}

Here is a consequence of Fact~\ref{fact:Anosov}.

\begin{lemma} \label{lem:Ano-cc-Coxeter}
Let $W_S$ be an infinite, word hyperbolic, irreducible Coxeter group and let $\rho : W_S\to\GL(V)$ be a representation of $W_S$ as a reflection group.
Then $\rho$ is $P_1$-Anosov if and only if $\rho(W_S)$ is strongly convex cocompact in $\PP(V)$.
\end{lemma}

Corollary~\ref{cor:hyp-CG-Anosov} follows directly from Corollaries~\ref{cor:main2} and~\ref{coro:cc-int-refl}, and Lemma~\ref{lem:Ano-cc-Coxeter}.

\begin{proof}[Proof of Lemma~\ref{lem:Ano-cc-Coxeter}]
If $\rho(W_S)$ is strongly convex cocompact in $\PP(V)$, then $\rho$ is $P_1$-Anosov by Fact~\ref{fact:Anosov}. 
Conversely, suppose $\rho$ is $P_1$-Anosov.
We cannot immediately apply Fact~\ref{fact:Anosov} because $\rho(W_S)$ might not preserve a \emph{properly} convex open subset of $\PP(V)$.
However, consider the induced representation $\rho^{\alpha}_v : W_S\to\GL(V^{\alpha}_v)$ as in Remarks \ref{rem:notation_rhos} and~\ref{rem:rho-v-alpha-refl-group}, with the same Cartan matrix $\Cart$ as~$\rho$.
It is easy to check (see \cite[Prop.\,4.1]{gw12} and \eqref{eqn:ninezeros}) that $\rho^{\alpha}_v : W_S\to\GL(V^{\alpha}_v)$ is still $P_1$-Anosov. 
The representation $\rho^{\alpha}_v$ is reduced (Definition~\ref{def:V-v-alpha}), hence preserves a properly convex open subset of $\PP(V^{\alpha}_v)$, namely the Tits--Vinberg domain for $\rho^{\alpha}_v(W_S)$ (Proposition~\ref{prop:maximal}.\eqref{item:max-1}).
By Fact~\ref{fact:Anosov}, the group $\rho^{\alpha}_v(W_S)$ is convex cocompact in $\PP(V^{\alpha}_v)$.
Note that the Cartan matrix $\Cart$ is of negative type: this follows from Proposition~\ref{prop:affine} if $N=2$, and from the fact that the infinite, word hyperbolic, irreducible Coxeter group $W_S$ is large if $N=3$ (use Fact~\ref{fact:types}).
Therefore $\rho(W_S)$ is convex cocompact in $\PP(V)$ by Lemma~\ref{lem:V-alpha-v}.\eqref{item:V-alpha-v-cc}.
\end{proof}

\subsection{Three examples}

First, we give an example of a right-angled but not word hyperbolic Coxeter group in $5$ generators such that the cyclic products $\Cart_{1,2} \Cart_{2,1}$ and $\Cart_{4,5} \Cart_{5,4}$ take only the value $4$ on all of $\Hom^{\mathrm{ref}}(W_S, \GL(\RR^4))$.

\begin{example}\label{ex:right-angled}
Let $W_S$ be the right-angled Coxeter group in $N = 5$ generators for which the Coxeter diagram is given by
\begin{center}
\begin{tikzpicture}[thick,scale=0.9, every node/.style={transform shape}]

\node[draw, circle, inner sep=1pt] (1) at (-3,0)   {$s_1$};
\node[draw, circle, inner sep=1pt] (2) at (-1.5,0) {$s_2$};
\node[draw, circle, inner sep=1pt] (3) at (0,0)    {$s_3$};
\node[draw, circle, inner sep=1pt] (4) at (1.5,0)  {$s_4$};
\node[draw, circle, inner sep=1pt] (5) at (3,0)    {$s_5$};

\draw (1) -- (2) node[midway,above] {$\infty$};
\draw (2) -- (3) node[above,midway] {$\infty$};
\draw (3) -- (4) node[above,midway] {$\infty$};
\draw (4) -- (5) node[above,midway] {$\infty$};
\end{tikzpicture}
\end{center}
Since there exist disjoint subsets $S' = \{s_1,s_2 \}$ and $S'' = \{s_4,s_5 \}$ of $S$ such that $W_{S'}$ and $W_{S''}$ are both infinite and commute, the group $W_S$ contains a subgroup isomorphic to~$\ZZ^2$, hence is not word hyperbolic.
Any compatible Cartan matrix for $W_S$ must be conjugate, by diagonal matrices, to
\begin{align*}
\Cart =  \begin{pmatrix}
  2 & -2x & 0  &  0 & 0  \\
 -2 & 2  & -2y &  0 & 0  \\
  0 & -2 & 2  & -2z & 0  \\
  0 & 0  & -2 &  2 & -2u \\
  0 & 0  & 0  & -2 & 2  \\
 \end{pmatrix}
\end{align*}
for some $x,y,z,u \geqslant 1$.
The $(2,4)-$minor of $\Cart$ is $16$, hence $\Cart$ always has rank $\geq 4$.
Rank equal to~$4$ is achieved, for example by the Cartan matrix for the Tits geometric representation, \ie $x=y=z=u=1$.
Since the rank is $4$, this Cartan matrix is also realized as the Cartan matrix for a reflection group in $V$ for $\dim V = 4$: for example, let $V$ be the span of the columns of $\Cart$ inside $\RR^5$.
A further linear algebra calculation shows that if $\mathrm{rank}\,\Cart = 4$, then $x =u = 1$. Indeed, it is an exercise to show that $\det(\Cart) = 32\,(xu+xz+yu -x-y-z-u+1)=0$ if and only if $x=u=1$, using the inequalities $x,y,z,u \geqslant 1$.
Hence if $\dim V = 4$, then all products $\Cart_{1,2}\Cart_{2,1}$ and $\Cart_{4,5}\Cart_{5,4}$ take only the value $4$ on all of $\Hom^{\mathrm{ref}}(W_S, \GL(V))$, which, up to conjugation, is two-dimensional.
\end{example}

Second, we give an example of a word hyperbolic but not right-angled Coxeter group in $6$ generators such that the set of characters $[\rho] \in \chi^{\mathrm{ref}}(W_S,\GL(\RR^5))$ for which $\rho(W_S)$ is convex cocompact in $\PP(\RR^5)$ is \textit{not} the interior of $\chi^{\mathrm{ref}}(W_S,\GL(\RR^5))$ in $\chi(W_S,\GL(\RR^5))$. In particular, the set of representations $\rho \in \Hom^{\mathrm{ref}}(W_S,\GL(\RR^5))$ for which $\rho(W_S)$ is convex cocompact in $\PP(\RR^5)$ is \textit{not} the interior of $\Hom^{\mathrm{ref}}(W_S,\GL(\RR^5))$ in $\Hom(W_S,\GL(\RR^5))$.

\begin{example}\label{ex:word-hyperbolic}
Let $W_S$ be the Coxeter group in $N = 6$ generators for which the Coxeter diagram is given by
\begin{center}
\begin{tikzpicture}[thick,scale=0.9, every node/.style={transform shape}]
\node[draw, circle, inner sep=1pt] (1) at (-1,1)  {$s_1$};
\node[draw, circle, inner sep=1pt] (2) at (-2,0)  {$s_2$};
\node[draw, circle, inner sep=1pt] (3) at (-1,-1) {$s_3$};
\node[draw, circle, inner sep=1pt] (4) at (0,0)   {$s_4$};
\node[draw, circle, inner sep=1pt] (5) at (1.5,0) {$s_5$};
\node[draw, circle, inner sep=1pt] (6) at (3,0)   {$s_6$};

\draw (1) -- (2) node[above,near end]   {$6$};
\draw (2) -- (3) node[below,near start] {$6$};
\draw (3) -- (4) node[below,near end]   {$6$};
\draw (4) -- (1) node[above,near start] {$6$};
\draw (4) -- (5) node[above,midway]     {$\infty$};
\draw (5) -- (6) node[above,midway]     {$6$};
\end{tikzpicture}
\end{center}
We can easily check by Moussong's hyperbolicity criterion (see Remark~\ref{rem:moussong}) that the Coxeter group $W_S$ is word hyperbolic, and see that any compatible Cartan matrix for $W_S$ must be conjugate, by diagonal matrices, to
\begin{align*}
\Cart =  \begin{pmatrix}
2                 & -\sqrt{3} &  0        &  -\sqrt{3} x &  0         &  0\\
-\sqrt{3}         &  2        & -\sqrt{3} &  0           &  0         &  0\\
 0                & -\sqrt{3} &  2        & -\sqrt{3}    &  0         &  0\\
 -\sqrt{3} x^{-1} &  0        & -\sqrt{3} &  2           & -2y        &  0\\
 0                &  0        &           0              & -2y        &  2 &  -\sqrt{3}\\
 0                &  0        &  0        & 0            &  -\sqrt{3} &  2\\
\end{pmatrix}
\end{align*}
for some $x > 0$ and $y \geq 1$.
The $(1,1)-$minor of $\Cart$ is $-4(2y^2+1)\neq 0$, hence $\Cart$ always has rank $\geq 5$.
In particular, all representations of~$W_S$ as a reflection group in $\RR^5$ must be irreducible, hence discussing character or equality up to conjugation is equivalent.
Rank equal to~$5$ is achieved by the Cartan matrix $\Cart$ with $\det(\Cart) = 32 y^2 - 9 (x+x^{-1}) - 14 = 0$.
Hence if $\dim V = 5$, then the space of representations in $\Hom^{\mathrm{ref}}(W_S, \GL(V))$ up to conjugation is homeomorphic to the line~$\RR$.
Moreover, the intersection of $\Hom^{\mathrm{ref}}(W_S, \GL(V))$ in those coordinates with the line $\{ y = 1 \}$ is reduced to the point $(x,y)=(1,1)$.
Hence, by Theorem~\ref{thm:main}, the line $\Hom^{\mathrm{ref}}(W_S, \GL(V))$ except one point corresponds to convex cocompact representations (see Figure~\ref{fig:ex1}).
\begin{figure}[h]
\centering
\labellist
\small\hair 2pt
\pinlabel {$x$} [u] at 345 25
\pinlabel {$y$} [u] at 25 345
\pinlabel {$(1,1)$} [u] at 205 100
\endlabellist
\includegraphics[scale=0.5]{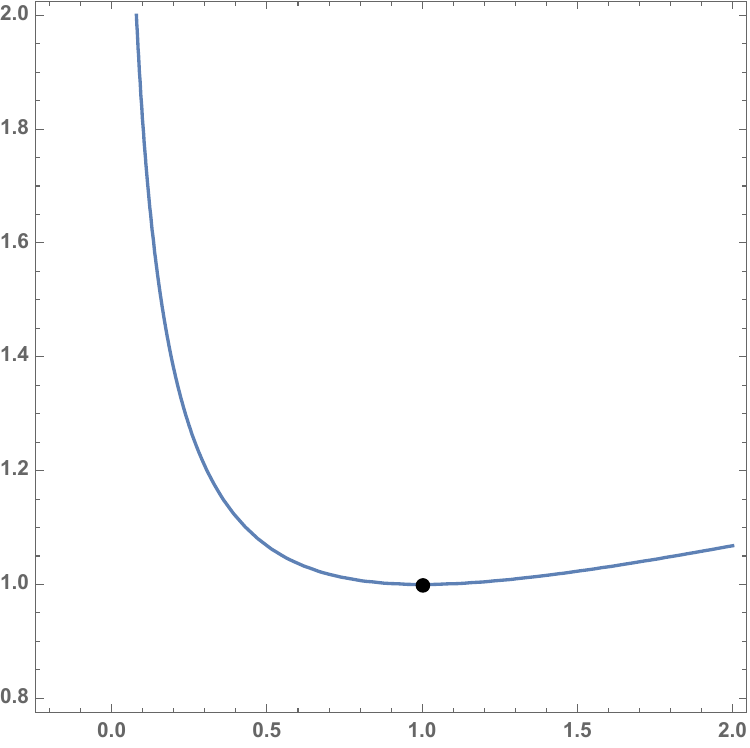}
\caption{The curve $\det(\Cart) = 0$ in terms of $x > 0$ and $y \geq 1$ in Example~\ref{ex:word-hyperbolic}. Every point of the curve except the black dot corresponds to a convex cocompact reflection group (up to conjugation).}
\label{fig:ex1}
\end{figure}
\end{example}

Lastly, we give an example of a not word hyperbolic and not right-angled Coxeter group in $4$ generators such that the set of characters $[\rho] \in \chi^{\mathrm{ref}}(W_S,\GL(\RR^4))$ for which $\rho(W_S)$ is convex cocompact in $\PP(\RR^4)$ is \textit{not} the interior of $\chi^{\mathrm{ref}}(W_S,\GL(\RR^4))$ in $\chi(W_S,\GL(\RR^4))$.

\begin{example} \label{ex:non-hyperbolic}
Let $W_S$ be the Coxeter group in $N = 4$ generators for which the Coxeter diagram is given by
\begin{center}
\begin{tikzpicture}[thick,scale=0.9, every node/.style={transform shape}]
\node[draw, circle, inner sep=1pt] (1) at (-1,1)  {$s_1$};
\node[draw, circle, inner sep=1pt] (2) at (-2.7,0)  {$s_2$};
\node[draw, circle, inner sep=1pt] (3) at (-1,-1) {$s_3$};
\node[draw, circle, inner sep=1pt] (4) at (0.7,0)   {$s_4$};
\draw (1) -- (2) ;
\draw (2) -- (3) ;
\draw (3) -- (4) ;
\draw (4) -- (1) node[above,midway]   {$4$};
\draw (1) -- (3) ;
\end{tikzpicture}
\end{center}
Since there is a subset $S' = \{s_1, s_2, s_3 \}$ of $S$ with $W_{S'}$ of type~$\widetilde{A}_2$, the group $W_S$ contains a subgroup isomorphic to~$\ZZ^2$, hence is not word hyperbolic.
We can easily see that any compatible Cartan matrix for $W_S$ must be conjugate, by diagonal matrices, to
\begin{align*}
\Cart =  \begin{pmatrix}
2         & -1      &  -1     &  -\sqrt{2} \\
-1        &  2      & -x      &  0         \\
-1        & -x^{-1} &  2      & -y         \\
-\sqrt{2} &  0      & -y^{-1} &  2         \\
\end{pmatrix}
\end{align*}
for some $x, y > 0$.
A simple calculation shows that
$$\det(\Cart)=- \Big( 2(x+x^{-1}) +2 \sqrt{2} (y+y^{-1}) + \sqrt{2} (xy + (xy)^{-1}) + 5  \Big) < 0.$$
In particular, all representations of~$W_S$ as a reflection group in $\RR^4$ must be irreducible, hence discussing character or equality up to conjugation is equivalent.
If $\dim V = 4$, then the space of representations in $\Hom^{\mathrm{ref}}(W_S, \GL(V))$ up to conjugation is homeomorphic to $\RR_{>0}^2$; by Theorem~\ref{thm:main}, convex cocompact representations correspond exactly to the complement of the line $\{x=1\}$ in $\RR_{>0}^2$.
\end{example}

\newpage

\appendix

\section{The spherical and affine Coxeter diagrams}\label{appendix:classi_diagram}

Here are all spherical and all affine irreducible diagrams, reproduced from~\cite{classi_cox_poly_coxeter}.

For spherical Coxeter groups (Table~\ref{table1}), the index --- in particular the $n$ for types $A_n$, $B_n$, $D_n$ --- is the number of nodes of the diagram, and the standard representation acts irreducibly on the $(n-1)$-dimensional sphere with a spherical simplex as a fundamental domain.

For affine Coxeter groups (Table~\ref{table2}), the index --- in particular the $n$ for types $\widetilde{A}_n$, $\widetilde{B}_n$, $\widetilde{C}_n$, $\widetilde{D}_n$ --- is \emph{one less than} the number of nodes; the group acts cocompactly on Euclidean $n$-space with a Euclidean simplex as a fundamental domain.

\newcommand{\scalee}{1.5}
\begin{table}[ht!]
\centering
\begin{minipage}[b]{7.7cm}
\centering
\begin{tikzpicture}[thick,scale=0.55, every node/.style={transform shape}]
\node[draw,circle] (A1) at (0,0.6) {};
\node[draw,circle,right=.8cm of A1] (A2) {};
\node[draw,circle,right=.8cm of A2] (A3) {};
\node[draw,circle,right=1cm of A3] (A4) {};
\node[draw,circle,right=.8cm of A4] (A5) {};
\node[left=.8cm of A1,scale=\scalee] {$\Huge A_n\ (n\geq 1)$};

\draw (A1) -- (A2)  node[above,midway] {};
\draw (A2) -- (A3)  node[above,midway] {};
\draw[loosely dotted,thick] (A3) -- (A4) node[] {};
\draw (A4) -- (A5) node[above,midway] {};


\node[draw,circle,below=1.7cm of A1] (B1) {};
\node[draw,circle,right=.8cm of B1] (B2) {};
\node[draw,circle,right=.8cm of B2] (B3) {};
\node[draw,circle,right=1cm of B3] (B4) {};
\node[draw,circle,right=.8cm of B4] (B5) {};
\node[left=.8cm of B1,scale=\scalee] {$B_n\ (n\geq 2)$};

\draw (B1) -- (B2)  node[above,midway] {$4$};
\draw (B2) -- (B3)  node[above,midway] {};
\draw[loosely dotted,thick] (B3) -- (B4) node[] {};
\draw (B4) -- (B5) node[above,midway] {};


\node[draw,circle,below=1.6cm of B1] (D1) {};
\node[draw,circle,right=.8cm of D1] (D2) {};
\node[draw,circle,right=1cm of D2] (D3) {};
\node[draw,circle,right=.8cm of D3] (D4) {};
\node[draw,circle, above right=.8cm of D4] (D5) {};
\node[draw,circle,below right=.8cm of D4] (D6) {};
\node[left=.8cm of D1,scale=\scalee] {$D_n\ (n\geq 4)$};

\draw (D1) -- (D2)  node[above,midway] {};
\draw[loosely dotted] (D2) -- (D3);
\draw (D3) -- (D4) node[above,midway] {};
\draw (D4) -- (D5) node[above,midway] {};
\draw (D4) -- (D6) node[below,midway] {};


\node[draw,circle,below=1.2cm of D1] (I1) {};
\node[draw,circle,right=.8cm of I1] (I2) {};
\node[left=.8cm of I1,scale=\scalee] {$I_2(p)\ (p\geq 5)$};

\draw (I1) -- (I2)  node[above,midway] {$p$};


\node[draw,circle,below=1.2cm of I1] (H1) {};
\node[draw,circle,right=.8cm of H1] (H2) {};
\node[draw,circle,right=.8cm of H2] (H3) {};
\node[left=.8cm of H1,scale=\scalee] {$H_3$};

\draw (H1) -- (H2)  node[above,midway] {$5$};
\draw (H2) -- (H3)  node[above,midway] {};


\node[draw,circle,below=1.2cm of H1] (HH1) {};
\node[draw,circle,right=.8cm of HH1] (HH2) {};
\node[draw,circle,right=.8cm of HH2] (HH3) {};
\node[draw,circle,right=.8cm of HH3] (HH4) {};
\node[left=.8cm of HH1,scale=\scalee] {$H_4$};

\draw (HH1) -- (HH2)  node[above,midway] {$5$};
\draw (HH2) -- (HH3)  node[above,midway] {};
\draw (HH3) -- (HH4)  node[above,midway] {};


\node[draw,circle,below=1.2cm of HH1] (F1) {};
\node[draw,circle,right=.8cm of F1] (F2) {};
\node[draw,circle,right=.8cm of F2] (F3) {};
\node[draw,circle,right=.8cm of F3] (F4) {};
\node[left=.8cm of F1,scale=\scalee] {$F_4$};

\draw (F1) -- (F2)  node[above,midway] {};
\draw (F2) -- (F3)  node[above,midway] {$4$};
\draw (F3) -- (F4)  node[above,midway] {};


\node[draw,circle,below=1.2cm of F1] (E1) {};
\node[draw,circle,right=.8cm of E1] (E2) {};
\node[draw,circle,right=.8cm of E2] (E3) {};
\node[draw,circle,right=.8cm of E3] (E4) {};
\node[draw,circle,right=.8cm of E4] (E5) {};
\node[draw,circle,below=.8cm of E3] (EA) {};
\node[left=.8cm of E1,scale=\scalee] {$E_6$};

\draw (E1) -- (E2)  node[above,midway] {};
\draw (E2) -- (E3)  node[above,midway] {};
\draw (E3) -- (E4)  node[above,midway] {};
\draw (E4) -- (E5)  node[above,midway] {};
\draw (E3) -- (EA)  node[left,midway] {};


\node[draw,circle,below=1.8cm of E1] (EE1) {};
\node[draw,circle,right=.8cm of EE1] (EE2) {};
\node[draw,circle,right=.8cm of EE2] (EE3) {};
\node[draw,circle,right=.8cm of EE3] (EE4) {};
\node[draw,circle,right=.8cm of EE4] (EE5) {};
\node[draw,circle,right=.8cm of EE5] (EE6) {};
\node[draw,circle,below=.8cm of EE3] (EEA) {};
\node[left=.8cm of EE1,scale=\scalee] {$E_7$};

\draw (EE1) -- (EE2)  node[above,midway] {};
\draw (EE2) -- (EE3)  node[above,midway] {};
\draw (EE3) -- (EE4)  node[above,midway] {};
\draw (EE4) -- (EE5)  node[above,midway] {};
\draw (EE5) -- (EE6)  node[above,midway] {};
\draw (EE3) -- (EEA)  node[left,midway] {};


\node[draw,circle,below=1.8cm of EE1] (EEE1) {};
\node[draw,circle,right=.8cm of EEE1] (EEE2) {};
\node[draw,circle,right=.8cm of EEE2] (EEE3) {};
\node[draw,circle,right=.8cm of EEE3] (EEE4) {};
\node[draw,circle,right=.8cm of EEE4] (EEE5) {};
\node[draw,circle,right=.8cm of EEE5] (EEE6) {};
\node[draw,circle,right=.8cm of EEE6] (EEE7) {};
\node[draw,circle,below=.8cm of EEE3] (EEEA) {};
\node[left=.8cm of EEE1,scale=\scalee] {$E_8$};

\draw (EEE1) -- (EEE2)  node[above,midway] {};
\draw (EEE2) -- (EEE3)  node[above,midway] {};
\draw (EEE3) -- (EEE4)  node[above,midway] {};
\draw (EEE4) -- (EEE5)  node[above,midway] {};
\draw (EEE5) -- (EEE6)  node[above,midway] {};
\draw (EEE6) -- (EEE7)  node[above,midway] {};
\draw (EEE3) -- (EEEA)  node[left,midway] {};

\draw (0,-18.5) node[]{} ;
\end{tikzpicture}
\caption{The diagrams of the spherical irreducible Coxeter groups \label{table1}}
\label{table:spheri_diag}
\end{minipage}
\begin{minipage}[t]{7.7cm}
\centering
\begin{tikzpicture}[thick,scale=0.55, every node/.style={transform shape}]
\node[draw,circle] (A1) at (0,0) {};
\node[draw,circle,above right=.8cm of A1] (A2) {};
\node[draw,circle,right=.8cm of A2] (A3) {};
\node[draw,circle,right=.8cm of A3] (A4) {};
\node[draw,circle,right=.8cm of A4] (A5) {};
\node[draw,circle,below right=.8cm of A5] (A6) {};
\node[draw,circle,below left=.8cm of A6] (A7) {};
\node[draw,circle,left=.8cm of A7] (A8) {};
\node[draw,circle,left=.8cm of A8] (A9) {};
\node[draw,circle,left=.8cm of A9] (A10) {};

\node[left=.8cm of A1,scale=\scalee] {$\widetilde{A}_n\ (n\geq 2)$};

\draw (A1) -- (A2)  node[above,midway] {};
\draw (A2) -- (A3)  node[above,midway] {};
\draw (A3) -- (A4) node[] {};
\draw (A4) -- (A5) node[above,midway] {};
\draw (A5) -- (A6) node[] {};
\draw (A6) -- (A7) node[] {};
\draw (A7) -- (A8) node[] {};
\draw[loosely dotted,thick] (A8) -- (A9) node[] {};
\draw (A9) -- (A10) node[] {};
\draw (A10) -- (A1) node[] {};


\node[draw,circle,below=1.7cm of A1] (B1) {};
\node[draw,circle,right=.8cm of B1] (B2) {};
\node[draw,circle,right=.8cm of B2] (B3) {};
\node[draw,circle,right=1cm of B3] (B4) {};
\node[draw,circle,right=.8cm of B4] (B5) {};
\node[draw,circle,above right=.8cm of B5] (B6) {};
\node[draw,circle,below right=.8cm of B5] (B7) {};
\node[left=.8cm of B1,scale=\scalee] {$\widetilde{B}_n\ (n\geq 3)$};

\draw (B1) -- (B2)  node[above,midway] {$4$};
\draw (B2) -- (B3)  node[above,midway] {};
\draw[loosely dotted,thick] (B3) -- (B4) node[] {};
\draw (B4) -- (B5) node[above,midway] {};
\draw (B5) -- (B6) node[above,midway] {};
\draw (B5) -- (B7) node[above,midway] {};


\node[draw,circle,below=1.5cm of B1] (C1) {};
\node[draw,circle,right=.8cm of C1] (C2) {};
\node[draw,circle,right=.8cm of C2] (C3) {};
\node[draw,circle,right=1cm of C3] (C4) {};
\node[draw,circle,right=.8cm of C4] (C5) {};
\node[left=.8cm of C1,scale=\scalee] (CCC) {$\widetilde{C}_n\ (n\geq 3)$};

\draw (C1) -- (C2)  node[above,midway] {$4$};
\draw (C2) -- (C3)  node[above,midway] {};
\draw[loosely dotted,thick] (C3) -- (C4) node[] {};
\draw (C4) -- (C5) node[above,midway] {$4$};


\node[draw,circle,below=0.8cm of C1] (D1) {};
\node[draw,circle,below right=0.8cm of D1] (D3) {};
\node[draw,circle,below left=0.8cm of D3] (D2) {};
\node[draw,circle,right=.8cm of D3] (DA) {};
\node[draw,circle,right=1cm of DA] (DB) {};
\node[draw,circle,right=.8cm of DB] (D4) {};
\node[draw,circle, above right=.8cm of D4] (D5) {};
\node[draw,circle,below right=.8cm of D4] (D6) {};
\node[below=0.8cm of CCC,scale=\scalee] {$\widetilde{D}_n\ (n\geq 4)$};

\draw (D1) -- (D3)  node[above,midway] {};
\draw (D2) -- (D3)  node[above,midway] {};
\draw (D3) -- (DA) node[above,midway] {};
\draw[loosely dotted] (DA) -- (DB);
\draw (D4) -- (DB) node[above,midway] {};
\draw (D4) -- (D5) node[above,midway] {};
\draw (D4) -- (D6) node[below,midway] {};


\node[draw,circle,below=2.5cm of D1] (I1) {};
\node[draw,circle,right=.8cm of I1] (I2) {};
\node[left=.8cm of I1,scale=\scalee] {$\widetilde{A}_1$};

\draw (I1) -- (I2)  node[above,midway] {$\infty$};


\node[draw,circle,below=1.2cm of I1] (H1) {};
\node[draw,circle,right=.8cm of H1] (H2) {};
\node[draw,circle,right=.8cm of H2] (H3) {};
\node[left=.8cm of H1,scale=\scalee] {$\widetilde{B}_2=\widetilde{C}_2$};

\draw (H1) -- (H2)  node[above,midway] {$4$};
\draw (H2) -- (H3)  node[above,midway] {$4$};


\node[draw,circle,below=1.2cm of H1] (HH1) {};
\node[draw,circle,right=.8cm of HH1] (HH2) {};
\node[draw,circle,right=.8cm of HH2] (HH3) {};
\node[left=.8cm of HH1,scale=\scalee] {$\widetilde{G}_2$};

\draw (HH1) -- (HH2)  node[above,midway] {$6$};
\draw (HH2) -- (HH3)  node[above,midway] {};


\node[draw,circle,below=1.2cm of HH1] (F1) {};
\node[draw,circle,right=.8cm of F1] (F2) {};
\node[draw,circle,right=.8cm of F2] (F3) {};
\node[draw,circle,right=.8cm of F3] (F4) {};
\node[draw,circle,right=.8cm of F4] (F5) {};
\node[left=.8cm of F1,scale=\scalee] {$\widetilde{F}_4$};

\draw (F1) -- (F2)  node[above,midway] {};
\draw (F2) -- (F3)  node[above,midway] {$4$};
\draw (F3) -- (F4)  node[above,midway] {};
\draw (F4) -- (F5)  node[above,midway] {};


\node[draw,circle,below=1.6cm of F1] (E1) {};
\node[draw,circle,right=.8cm of E1] (E2) {};
\node[draw,circle,right=.8cm of E2] (E3) {};
\node[draw,circle,above right=.8cm of E3] (E4) {};
\node[draw,circle,right=.8cm of E4] (E5) {};
\node[draw,circle,below right =.8cm of E3] (EA) {};
\node[draw,circle,right=.8cm of EA] (EB) {};
\node[left=.8cm of E1,scale=\scalee] {$\widetilde{E}_6$};

\draw (E1) -- (E2)  node[above,midway] {};
\draw (E2) -- (E3)  node[above,midway] {};
\draw (E3) -- (E4)  node[above,midway] {};
\draw (E4) -- (E5)  node[above,midway] {};
\draw (E3) -- (EA)  node[left,midway] {};
\draw (EB) -- (EA)  node[left,midway] {};


\node[draw,circle,below=1.6cm of E1] (EE1) {};
\node[draw,circle,right=.8cm of EE1] (EEB) {};
\node[draw,circle,right=.8cm of EEB] (EE2) {};
\node[draw,circle,right=.8cm of EE2] (EE3) {};
\node[draw,circle,right=.8cm of EE3] (EE4) {};
\node[draw,circle,right=.8cm of EE4] (EE5) {};
\node[draw,circle,right=.8cm of EE5] (EE6) {};
\node[draw,circle,below=.8cm of EE3] (EEA) {};
\node[left=.8cm of EE1,scale=\scalee] {$\widetilde{E}_7$};

\draw (EE1) -- (EEB)  node[above,midway] {};
\draw (EE2) -- (EEB)  node[above,midway] {};
\draw (EE2) -- (EE3)  node[above,midway] {};
\draw (EE3) -- (EE4)  node[above,midway] {};
\draw (EE4) -- (EE5)  node[above,midway] {};
\draw (EE5) -- (EE6)  node[above,midway] {};
\draw (EE3) -- (EEA)  node[left,midway] {};


\node[draw,circle,below=1.8cm of EE1] (EEE1) {};
\node[draw,circle,right=.8cm of EEE1] (EEE2) {};
\node[draw,circle,right=.8cm of EEE2] (EEE3) {};
\node[draw,circle,right=.8cm of EEE3] (EEE4) {};
\node[draw,circle,right=.8cm of EEE4] (EEE5) {};
\node[draw,circle,right=.8cm of EEE5] (EEE6) {};
\node[draw,circle,right=.8cm of EEE6] (EEE7) {};
\node[draw,circle,right=.8cm of EEE7] (EEE8) {};
\node[draw,circle,below=.8cm of EEE3] (EEEA) {};
\node[left=.8cm of EEE1,scale=\scalee] {$\widetilde{E}_8$};

\draw (EEE1) -- (EEE2)  node[above,midway] {};
\draw (EEE2) -- (EEE3)  node[above,midway] {};
\draw (EEE3) -- (EEE4)  node[above,midway] {};
\draw (EEE4) -- (EEE5)  node[above,midway] {};
\draw (EEE5) -- (EEE6)  node[above,midway] {};
\draw (EEE6) -- (EEE7)  node[above,midway] {};
\draw (EEE8) -- (EEE7)  node[above,midway] {};
\draw (EEE3) -- (EEEA)  node[left,midway] {};

\draw (0,-22.5) node[]{} ;
\end{tikzpicture}
\caption{The diagrams of the affine irreducible Coxeter groups \label{table2}}
\label{table:affi_diag}
\end{minipage}
\end{table}

\clearpage


\end{document}